\documentclass[10pt,letterpaper]{amsart}
\usepackage[utf8]{inputenc}\usepackage[left=2.7cm,right=2.7cm,top=3.5cm,bottom=3cm]{geometry}
\usepackage{amssymb,latexsym,amsmath,amsthm,amscd}
\usepackage{graphicx}
\usepackage{dsfont}
\usepackage{stmaryrd}
\usepackage{bm}
\usepackage{comment}
\usepackage[pagebackref=false,colorlinks]{hyperref}
\hypersetup{pdffitwindow=true,linkcolor=blue,citecolor=blue,urlcolor=cyan}
\usepackage{hyperref}
\usepackage[all]{xy}
\usepackage{mathrsfs}
\usepackage{color}
\definecolor{Blue}{rgb}{0.3,0.3,0.9}

\usepackage[T2A,OT1]{fontenc}
\DeclareSymbolFont{cyrillic}{T2A}{cmr}{m}{n}
\DeclareMathSymbol{\Sha}{\mathalpha}{cyrillic}{216}

\theoremstyle{plain}
\newtheorem{thm}{Theorem}[subsection] 

\usepackage{mathrsfs}

\theoremstyle{definition}

\newtheorem{defn}[thm]{Definition}

\newtheorem{rmk}[thm]{Remark}

\newtheorem{assumption}{Assumption}[section]

\theoremstyle{definition}

\theoremstyle{plain}
\newtheorem{prop}[thm]{Proposition}
\theoremstyle{plain}
\newtheorem{lem}[thm]{Lemma}
\theoremstyle{plain}
\newtheorem{cor}[thm]{Corollary}
\newtheorem*{thm-intro}{Theorem}
\newtheorem*{cor-intro}{Corollary}

\theoremstyle{remark}
\newtheorem{notation}[thm]{Notation}

\numberwithin{equation}{section}

\newcommand{\eps}{\varepsilon}

\newcommand{\cc}{\mathbf{c}}

\newcommand{\ch}{\xi}

\newcommand{\Sel}{Sel}
\newcommand{\rH}{H}

\newcommand{\pp}{\mathfrak{p}}
\newcommand{\ppbar}{\bar{\mathfrak{p}}}
\newcommand{\fkm}{\mathfrak{m}}
\newcommand{\fkmbar}{\overline{\mathfrak{m}}}
\newcommand{\fkf}{\mathfrak{f}}
\newcommand{\fkfbar}{\overline{\mathfrak{f}}}
\newcommand{\fkl}{\mathfrak{l}}
\newcommand{\fklbar}{\overline{\mathfrak{l}}}
\newcommand{\fkn}{\mathfrak{n}}

\newcommand{\qq}{\mathfrak{q}}
\newcommand{\fa}{\mathfrak{a}}

\newcommand{\relstr}{{{\rm rel},{\rm str}}}
\newcommand{\strrel}{{{\rm str},{\rm rel}}}
\newcommand{\ord}{{{\rm ord},{\rm ord}}}

\newcommand{\ts}{\textsuperscript}

\DeclareMathOperator{\pa}{\mathfrak{p}}

\DeclareMathOperator{\et}{\acute{e}t}
\DeclareMathOperator{\F}{\mathscr{F}}
\DeclareMathOperator{\Ll}{\mathscr{L}}
\DeclareMathOperator{\Ss}{\mathscr{S}}
\DeclareMathOperator{\p}{\mathfrak{P}}
\DeclareMathOperator{\pr1}{\text{pr}_{1*}}
\DeclareMathOperator{\tpr}{\text{pr}}

\DeclareMathOperator{\Norm}{\mathrm{Norm}}
\DeclareMathOperator{\Frob}{\mathrm{Frob}}
\DeclareMathOperator{\ur}{\mathrm{ur}}
\DeclareMathOperator{\bal}{\mathrm{bal}}
\DeclareMathOperator{\BDP}{BDP} 
\DeclareMathOperator{\BD}{BD} 

\DeclareMathAlphabet\mathbfcal{OMS}{cmsy}{b}{n}
\DeclareTextSymbolDefault{\uhorn}{T5}
\DeclareTextSymbolDefault{\ocircumflex}{T5}
\DeclareTextSymbolDefault{\acircumflex}{T5}

\newcommand{\bQ}{\mathbf{Q}}
\newcommand{\bZ}{\mathbf{Z}}
\newcommand{\bC}{\mathbf{C}}

\def\cA{{\mathcal A}}  

\def\cL{{\mathcal L}}

\def\cO{\mathcal O}

\def\Lcal{\mathcal L}

\newcommand{\VQdag}{{\mathbf{V}_{\underline{Q}}^\dagger}}
\newcommand{\Vsdag}{{\mathbf{V}_{Q_0}^\dagger}}
\newcommand{\Vdag}{{\mathbf{V}^\dagger}}
\newcommand{\Adag}{{\mathbf{A}^\dagger}}
\newcommand{\VVdag}{{\mathbb{V}^\dagger}}

\newcommand{\unb}{{\boldsymbol{f}}}

\newcommand{\Q}{\mathbf{Q}}
\newcommand{\Z}{\mathbf{Z}}

\def\makeop#1{\expandafter\def\csname#1\endcsname
	{\mathop{\rm #1}\nolimits}\ignorespaces}
\makeop{Hom}   \makeop{End}   \makeop{Aut}   
\makeop{Pic} \makeop{Gal}       \makeop{Div} \makeop{Lie}
\makeop{PGL}   \makeop{Corr} \makeop{PSL} \makeop{sgn} \makeop{Spf}
\makeop{Tr} \makeop{Nr} \makeop{Fr} \makeop{disc}
\makeop{Proj} \makeop{supp} \makeop{ker}   \makeop{Im} \makeop{dom}
\makeop{coker} \makeop{Stab} \makeop{SO} \makeop{SL} \makeop{SL}
\makeop{Cl}    \makeop{cond} \makeop{Br} \makeop{inv} \makeop{rank}
\makeop{id}    \makeop{Fil} \makeop{Frac}  \makeop{GL} \makeop{SU}
\makeop{Trd}   \makeop{Sp} \makeop{Tr}    \makeop{Trd} \makeop{Res}
\makeop{ind} \makeop{depth} \makeop{Tr} \makeop{st} \makeop{Ad}
\makeop{Int} \makeop{tr}    \makeop{Sym} \makeop{can} \makeop{SO}
\makeop{torsion} \makeop{GSp} \makeop{Tor}\makeop{Ker} \makeop{rec}
\makeop{Ind} \makeop{Coker}
\makeop{vol} \makeop{Ext} \makeop{gr} \makeop{ad}
\makeop{Gr}\makeop{corank} \makeop{Ann}
\makeop{Hol} 
\makeop{Fitt} \makeop{Mp} \makeop{CAP}

\def\Sel{{\rm Sel}}

\def\Ord{{\mathrm{ord}}}

\def\iso{\simeq}

\newcommand{\dBr}[1]{\llbracket{#1}\rrbracket}

\newcommand{\bT}{\mathbb{T}}

\newcommand{\cR}{\mathbb{I}}

\newcommand{\any}{?}

\newcommand{\bfff}{{\boldsymbol{f}}}
\newcommand{\bff}{{\boldsymbol{f}}}
\newcommand{\bfg}{{\boldsymbol{g}}}
\newcommand{\bfh}{{\boldsymbol{h}}}

\begin{document}
	
\title{Diagonal cycles and anticyclotomic Iwasawa theory of modular forms}
\author[F.\,Castella]{Francesc Castella}
\author[K.\,T.\,Do]{Kim Tuan Do}
	
\date{\today} 
		
\address[]{Department of Mathematics, University of California Santa Barbara, CA 93106, USA}
\email{castella@ucsb.edu}
\address[]{Department of Mathematics, University of California Los Angeles, CA 90095, USA}
\address[]{Department of Mathematics, University of California Santa Barbara, CA 93106, USA}
\email{ktdo@ucsb.edu}



\begin{abstract}
We construct a new Euler system for the Galois representation $V_{f,\chi}$ attached to a newform $f$ of weight $2r\geq 2$ twisted by an anticyclotomic Hecke character $\chi$. The Euler system is anticyclotomic in the sense of Jetchev--Nekov\'a{\v r}--Skinner. We then show some arithmetic applications of the constructed Euler system, including new results on the Bloch--Kato conjecture in ranks zero and one, and a divisibility towards the Iwasawa--Greenberg main conjecture for $V_{f,\chi}$. 

In particular, in the case where the base-change of $f$ to our imaginary quadratic field  has root number $+1$  and $\chi$ has higher weight (which implies that the complex $L$-function $L(V_{f,\chi},s)$ vanishes at the center), our results show that the Bloch--Kato Selmer group of $V_{f,\chi}$ is nonzero, as predicted by the Bloch--Kato conjecture; and if in addition a certain distinguished class $\kappa_{f,\chi}$ is nonzero, then the Selmer group is one-dimensional. Such applications to the Bloch--Kato conjecture for $V_{f,\chi}$ were left wide open by the earlier approaches using Heegner cycles and/or Beilinson--Flach elements. Our construction is based instead on a generalization of the Gross--Kudla--Schoen diagonal cycles.
\end{abstract}


\maketitle
\setcounter{tocdepth}{2}
\newpage
\tableofcontents

\newpage
\section{Introduction}

\addtocontents{toc}{\setcounter{tocdepth}{-10}}

{
\renewcommand{\thethm}{\Alph{thm}}

Let $f=\sum_{n=1}^\infty a_nq^n\in S_{2r}(\Gamma_0(N_f))$ be an elliptic newform of even weight $2r\geq 2$,  and let $p\nmid 6N_f$ be a prime. Let $K$ be an imaginary quadratic field in which $p$ splits. Let $L$ be a number field containing $K$ and the Fourier coefficients of $f$, and let $\mathfrak{P}$ be a prime of $L$ above $p$ at which $f$ is ordinary, i.e. $v_{\mathfrak{P}}(a_p)=0$. Let $\chi$ be an anticyclotomic Hecke character of $K$, and consider the conjugate self-dual $G_K={\mathrm{Gal}}(\overline{\Q}/K)$-representation
\[
V_{f,\chi}:=V_f^\vee(1-r)\otimes\chi^{-1},
\] 
where $V_f^\vee$ is the contragredient of Deligne's $\mathfrak{P}$-adic Galois representation associated to $f$. 

We prove the following applications to the Bloch--Kato conjecture for $V_{f,\chi}$. Under mild hypotheses on $f$ and $\chi$, the nonvanishing of the Rankin--Selberg $L$-function $L(f/K,\chi,s)$ at the center $s=r$ implies that the associated Bloch--Kato Selmer group is $0$; and when this central $L$-value vanishes, the nonvanishing of a distinguished class $\kappa_{f,\chi}$ implies that the dimension of the associated Bloch--Kato Selmer group is $1$. In addition, we also prove a divisibility in the Iwasawa main conjecture for $V_{f,\chi}$, both in the definite and in the indefinite settings. These results are deduced as applications of the main contribution of this paper, which is the  construction of a new anticyclotomic Euler system for $V_{f,\chi}$. By exploiting the decomposition of certain triple products, our construction is based on a generalization of the diagonal cycles introduced by Gross--Kudla \cite{gross-kudla} and Gross--Schoen \cite{gross-schoen}, and studied more recently by Darmon--Rotger and Bertolini--Seveso--Venerucci (see \cite{BDRSV}). 

%
%

\subsection{Main results} 

We fix once and for all complex and $p$-adic embeddings  $i_\infty:\overline{\Q}\hookrightarrow\overline{\mathbf{C}}$ and $i_p:\overline{\Q}\hookrightarrow\overline{\mathbf{C}}_p$. Assume that the discriminant $D_K$ of $K$ is coprime to $N_f$, and writing  $N_f=N^+N^-$ with $N^+$ (resp. $N^-$) divisible only by primes that are split (resp. inert) in $K$, assume that 
\begin{equation}\label{eq:sq}
\textrm{$N^-$ is squarefree.} \tag{sq}
\end{equation} 
Assume also that 
\begin{equation}\label{eq:spl}
\textrm{$p=\pp\ppbar$ splits in $K$},\tag{spl}
\end{equation}
with $\pp$ the prime of $K$ above $p$ induced by $i_p$. Let $\Gamma^-$ be the Galois group of the anticyclotomic $\Z_p$-extension of $K$. We consider anticyclotomic Hecke characters of $K$ of the form $\chi=\chi_0\phi$, with $\chi_0$ a ring character such that
\begin{equation}\label{eq:cond}
\textrm{$\chi_0$ has conductor $c\cO_K$ with $(c,pN_f)=1$,} \tag{cond}
\end{equation}
and $\phi$ an anticyclotomic Hecke character of $K$ whose $p$-adic avatar (still denoted $\phi$) factors through $\Gamma^-$. Denote by $\nu(N^-)$ the number of prime factors of $N^-$. 
Under hypotheses \eqref{eq:sq}, \eqref{eq:spl}, and \eqref{eq:cond},  the sign $\epsilon(f,\chi)$ in the functional equation for $L(f/K,\chi,s)$ (relating its values at $s$ and $2r-s$) depends only on the global root number of the base-change of $f$ to $K$, given by
\[
\epsilon(f/K)=(-1)^{\nu(N^-)+1},
\]
and the infinity type, say $(-j,j)$ with $j\in\Z$, of $\chi$. Because  $L(f/K,\chi,s)=L(f/K,\chi^\cc,s)$, where $\chi^\cc$ is the composition of $\chi$ with the action of complex conjugation, without loss of generality we may assume $j\geq 0$. Accordingly, the values of $\epsilon(f,\chi)$ are as in the following table. 

\begin{center}
\begin{tabular}{c|c|c|}
\cline{2-3}
  & $\epsilon(f/K)=-1$ & $\epsilon(f/K)=+1$ \\
 \hline
\multicolumn{1}{|c|}{$0\leq j<r$} & $-1$ & $+1$  \\
 \hline
 \multicolumn{1}{|c|}{$j\geq r$} & +1 & $-1$  \\
 \hline
\end{tabular}
\end{center}

\subsubsection{The Euler system}

Assume in addition that 
\begin{equation}\label{eq:ord}
\textrm{$f$ is ordinary at $\mathfrak{P}$},\tag{ord}
\end{equation}
with $\mathfrak{P}\mid\pp$, 
and that
\begin{equation}\label{eq:cn}
\textrm{$p\nmid h_K$, where $h_K$ is the class number of $K$.}\tag{cn}
\end{equation}
%

For every positive integer $n$, let $K[n]$ denote the maximal $p$-subextension of the ring class field of $K$ of conductor $n$. Denote by $\mathcal{L}$ the set of rational primes $\ell\neq p$ split in $K$. For each $\ell\in\mathcal{L}$, we fix a prime $\mathfrak{l}$ of $K$ lying above it, and let $\mathcal{N}$ be the set of squarefree products of primes $\ell\in\mathcal{L}$ coprime to $pN_fc$ (with  $1\in\mathcal{N}$ by convention, corresponding to the empty product). Let $\cO$ be the ring of integers in the completion $L_\mathfrak{P}$. 
 
%
%

\begin{thm}[Theorem~\ref{maintheorem2}]
\label{thmA}
Assume \eqref{eq:spl}, \eqref{eq:cond}, \eqref{eq:ord}, and \eqref{eq:p-nmid-h}. There exists a family of cohomology classes
\[
\bigl\{z_{f,\chi,m,s}\in H^1(K[mp^s],T_{f,\chi})\;\lvert\; m\in\mathcal{N},\, s\geq 0\bigr\},
\]
where $T_{f,\chi}$ is a certain $G_K$-stable $\mathcal{O}$-lattice inside $V_{f,\chi}$, such that 
\[
\Norm^{K[mp^{s+1}]}_{K[mp^s]}(z_{f,\chi,m,s+1})=z_{f,\chi,m,s}
\]
for all $s\geq 0$, and for every $m\in\mathcal{N}$ and $\ell\in\mathcal{L}$ with $m\ell\in\mathcal{N}$, we have the tame norm relation
\[
\Norm^{K[m\ell p^s]}_{K[mp^s]}(z_{f,\chi,m\ell,s})=
P_{\fkl}(\Frob_{\fkl})\,z_{f,\chi,m,s},
\] 
where $P_{\fkl}(X)=\det(1-\Frob_{\fkl}X\,\vert\,V_{f,\chi}^\vee(1))$, and $\Frob_\fkl$ is a geometric Frobenius.
\end{thm}

The collection of classes of Theorem~\ref{thmA} defines an \emph{anticyclotomic Euler system} in the sense of Jetchev--Nekov{\'a}{\v{r}}--Skinner \cite{JNS} for the conjugate self-dual representation $V_{f,\chi}$. %
Significantly extending Kolyvagin's methods, the general theory developed in \emph{op.\,cit.} provides a machinery that bounds Selmer groups attached to conjugate self-dual representations $V$ from the input of a non-trivial anticyclotomic Euler system. 
The Selmer group being bounded depends on the local conditions at the primes $w\mid p$ satisfied by the Euler system classes, and 
in this paper we produce in fact \emph{two different} anticyclotomic Euler systems for $V_{f,\chi}$, differing in their local conditions at the primes above $p$.

To describe this further, recall that by $p$-ordinarity of $f$, the Galois representation $V_f^\vee$ restricted to a decomposition group $G_{\Q_p}\subset G_\Q$ fits into a short exact sequence
\[
0\rightarrow V_f^{\vee,+}\rightarrow V_f^\vee\rightarrow V_f^{\vee,-}\rightarrow 0,
\] 
with each $V_f^{\vee,\pm}$ $1$-dimensional over $L_\mathfrak{P}$, and with the $G_{\Q_p}$-action on $V_f^{\vee,-}$ given by the unramified character sending an arithmetic Frobenius $\Frob^{-1}_p$ to $\alpha_p$, the unit root of $x^2-a_px+p^{2r-1}$. Put 
\[
V_{f,\chi}^\pm:=V_f^{\vee,\pm}(1-r)\otimes\chi^{-1}.
\] 
In terms of this, the construction in Theorem~\ref{maintheorem2} yields in fact:
\begin{itemize}
\item An anticyclotomic Euler system 
$\{z_{f,\chi,m,s}^{\ord}\}_{m,s}$
%
with local conditions at the primes $w\mid p$ given by 
\[
H^1_{\Ord}(K[mp^s]_w,V_{f,\chi}):=\mathrm{ker}\bigl(H^1(K[mp^s]_w,V_{f,\chi})\rightarrow H^1(K[mp^s]_w,V_{f,\chi}^-)\bigr).
\]
\item An anticyclotomic Euler system 
$\{z_{f,\chi,m,s}^{\relstr}\}_{m,s}$ with local conditions at the primes $w\mid p$ given by
\begin{equation}\label{intro:rel-str}
\begin{cases}
H^1(K[mp^s]_w,V_{f,\chi})&\textrm{if $w\mid\pp$,}\\[0.2em]
0&\textrm{if $w\mid\ppbar$.}
\end{cases}\nonumber
\end{equation}
\end{itemize}

Depending on the infinity type of $\chi$, we show that one of these classes always lands in the Bloch--Kato Selmer group $\Sel_{\mathrm{BK}}(K[mp^s],V_{f,\chi})$, namely the class
\[
\kappa_{f,\chi,m,s}:=\begin{cases}
z_{f,\chi,m,s}^{\relstr} & \textrm{if 
$j\geq r$,}\\[0.3em]
z_{f,\chi,m,s}^{\ord} & 
\textrm{
if $0\leq j<r$.}
\end{cases}
\]


\subsubsection{Applications to the Bloch--Kato conjecture in rank $1$} 

Let $\kappa_\mathfrak{P}$ denote the residue field of $L_\mathfrak{P}$, and let 
\[
\bar{\rho}_f:G_\Q\rightarrow{\mathrm{GL}}_2(\kappa_\mathfrak{P})
\]
be the residual representation associated to $f$. By $p$-ordinarity, the restriction $\bar{\rho}_f\vert_{G_{\Q_p}}$ is reducible, and we say that $\bar{\rho}_f$ is \emph{$p$-distinguished} when the semi-simplification of $\bar{\rho}_f\vert_{G_{\Q_p}}$ is non-scalar. Put
\[
\kappa_{f,\chi}:=
\kappa_{f,\chi,1,0}\in\Sel_{\mathrm{BK}}(K,V_{f,\chi}), 
\]
using that $K[1]=K$ as a consequence of \eqref{eq:p-nmid-h}. From the general results of \cite{JNS}  applied to the construction of Theorem~\ref{thmA} we deduce in particular the following result. 


\begin{thm}[Theorem~\ref{thm:BK-def-1}] 
\label{thmB}
Under the hypotheses of Theorem~\ref{thmA}, assume in addition  that
\begin{itemize}
    \item \eqref{eq:sq} holds;
    \item $\bar{\rho}_f$ is absolutely irreducible and $p$-distinguished;
    \item $p>2r-2$;
    \item $f$ is not of CM-type.
\end{itemize}
Assume also that
\[
\epsilon(f/K)=+1\quad\textrm{and}\quad j\geq r,
\]
which implies $L(f/K,\chi,r)=0$. Then 
%
%
\[
\dim_{L_\mathfrak{P}}\Sel_{\mathrm{BK}}(K,V_{f,\chi})\geq 1.
\] 
Moreover, if the class $\kappa_{f,\chi}$ 
is nonzero, then
\[
\Sel_{\mathrm{BK}}(K,V_{f,\chi})=L_\mathfrak{P}\cdot\kappa_{f,\chi}.
\]
\end{thm}

By the Gross--Zagier formula for the modified diagonal cycles introduced in \cite{gross-kudla,gross-schoen}
(a special case of the arithmetic Gan--Gross--Prasad conjecture for ${\mathrm{SO}}(3)\times\mathrm{SO}(4)$) proved by Yuan--Zhang--Zhang \cite{YZZ} under some assumptions, the non-triviality of $\kappa_{f,\chi}$ can be related to the nonvanishing of $L'(f/K,\chi,r)$, and hence Theorem~\ref{thmB} yields  evidence towards the Bloch--Kato conjecture for $V_{f,\chi}$ in analytic rank $1$. 

Our methods also yield an analogue of Theorem~\ref{thmB} in the 
``indefinite case'' $\epsilon(f/K)=-1$ and $0\leq j<r$ (see  Theorem~\ref{thm:BK-indef-1}), 
but we note that in this case such result can also be obtained from the Euler system of (generalized) Heegner cycles \cite{nekovar-invmath,cas-hsieh1}.

\subsubsection{Applications to the Bloch--Kato conjecture in rank $0$}

We now turn our attention to the cases where $\epsilon(f,\chi)=+1$, so the central value $L(f/K,\chi,r)$ is expected to be generically nonzero. Put
\begin{equation}\label{intro-non-Sel}
\kappa'_{f,\chi,m,s}:=\begin{cases}
z_{f,\chi,m,s}^{\ord} & 
\textrm{if $j\geq r$,}\\[0.3em]
z_{f,\chi,m,s}^{\relstr} & 
\textrm{if $0\leq j<r$,}
\end{cases}\nonumber\quad\quad\kappa_{f,\chi}':=\kappa'_{f,\chi,1,0}.
\end{equation}

Building on the explicit reciprocity law by Bertolini--Seveso--Venerucci for diagonal cycles \cite{BSV}, 
we show the equivalence 
\[
\kappa_{f,\chi}'\in\Sel_{\mathrm{BK}}(K,V_{f,\chi})\quad\Longleftrightarrow\quad L(f/K,\chi,r)=0.
\]
Hence when $L(f/K,\chi,r)\neq 0$, the classes $\kappa_{f,\chi}'$ provide non-trivial annihilators of classes in $\Sel_{\mathrm{BK}}(K,V_{f,\chi})$ via global duality. Together with the general results of \cite{JNS} applied to the construction of Theorem~\ref{thmA} extending $\kappa_{f,\chi}'$, this leads in particular to the following cases of the Bloch--Kato conjecture for $V_{f,\chi}$. 

%

\begin{thm}[Theorem~\ref{thm:BK-def}] 
\label{thmC}
Under the hypotheses of Theorem~\ref{thmA}, assume in addition  that
\begin{itemize}
    \item \eqref{eq:sq} holds;
    \item $\bar{\rho}_f$ is absolutely irreducible and $p$-distinguished;
    \item $p>2r-2$;
    \item $f$ is not of CM-type.
\end{itemize}
Assume also that
\[
\epsilon(f/K)=+1\quad\textrm{and}\quad 0\leq j< r,
\]
which implies $\epsilon(f,\chi)=+1$. Then
\[
L(f/K,\chi,r)\neq 0\quad\Longrightarrow\quad\Sel_{\mathrm{BK}}(K,V_{f,\chi})=0,
\]
and hence the Bloch--Kato conjecture holds in this case.
\end{thm}

We also obtain an analogue of Theorem~\ref{thmC} for $\epsilon(f/K)=-1$ and $j\geq r$ 
(see Theorem~\ref{thm:BK-indef}), but in these cases the result was previously known using generalized Heegner cycles \cite{cas-hsieh1}. Finally, we note 
that our results also include the proof of a divisibility in the anticyclotomic Iwasawa main conjecture for $V_{f,\chi}$, 
both in the definite and in the indefinite settings, giving in particular a new proof of the main result of \cite{bdIMC} (see Theorem\,\ref{thm:IMC-def}) dispensing with their ``level-raising'' ramification hypotheses.

\subsection{Relation to previous works}\label{subsec:previous}

Starting with the landmark results by Gross--Zagier and Kolyvagin \cite{GZ,kol88} (see also \cite{BD-crelle}), and followed by their vast generalizations by Zhang \cite{zhang130}, Tian \cite{Tian-PhD}, Nekov{\'a}{\v{r}} \cite{nekovar-CM}, Yuan--Zhang--Zhang \cite{YZZ-GZ} and others, the Euler system of Heegner points and Heegner cycles has been a key ingredient in the study of the arithmetic of $V_{f,\chi}$ under the Heegner hypothesis
\[
\epsilon(f/K)=-1.
\]
Classical Heegner cycles account for the cases where the anticyclotomic character $\chi$ has finite order (i.e., $j=0$), but using their new variant by Bertolini--Darmon--Prasanna \cite{bdp1}, one obtains classes controlling the arithmetic of $\Sel_{\mathrm{BK}}(K,V_{f,\chi})$ in the following cases:
\begin{equation}\label{eq:1}
\epsilon(f/K)=-1,\quad 0\leq j<r.\tag{1\ts{st} quadrant}
\end{equation} 

In another major advance, Bertolini--Darmon \cite{bdIMC} exploited congruences between modular forms on different quaternion algebras and the Cerednik--Drinfeld theory of interchange of invariants to realize the Galois representation (on finite quotients of) $T_{f,\chi}$ in the torsion of the Jacobian of certain Shimura curves. This allowed them to still use the Heegner point construction in situations where $\epsilon(f/K)=+1$. Together with the extension to higher weights by Chida--Hsieh \cite{ChHs2}, these methods yielded a proof of many cases of the Bloch--Kato conjecture in analytic rank $0$ when
\begin{equation}\label{eq:2}
\epsilon(f/K)=+1,\quad j=0\tag{2\ts{nd} quadrant}
\end{equation} 
under certain `level-raising' hypotheses. More recently, the Euler system of Beilinson--Flach elements  constructed by Lei--Loeffler--Zerbes \cite{LLZ,LLZ-K} and Kings--Loeffler--Zerbes \cite{KLZ,KLZ-AJM} (inspired in part by earlier results of Bertolini--Darmon--Rotger \cite{BDR1,BDR2}) provided an alternative approach to similar rank $0$ results (among other applications) under different mild hypotheses.

On the other hand, 
exploiting the variation of Heegner cycles in $p$-adic families, the first-named  author and Hsieh \cite{cas-hsieh1,cas-2var} (see also related work by Magrone \cite{Magr} and Kobayashi \cite{kobayashi-GHC}),  obtained results on the Bloch--Kato conjecture for $V_{f,\chi}$ in rank $0$ in the cases
\begin{equation}\label{eq:3}
\epsilon(f/K)=-1,\quad j\geq r.\tag{3\ts{rd} quadrant}
\end{equation} 

Contrastingly, in the cases where
\begin{equation}\label{eq:4}
\epsilon(f/K)=+1,\quad j\geq r,\tag{4\ts{th} quadrant}
\end{equation} 
the conjectures of Beilinson--Bloch and Bloch--Kato  predict the existence of nonzero classes in $\Sel_{\mathrm{BK}}(K,V_{f,\chi})$ coming from geometry (since $\epsilon(f,\chi)=-1$ and therefore $L(f/K,\chi,r)=0$), but the construction of such classes seems to fall outside of all the aforementioned methods. 
%
%

The `degenerate' diagonal cycle 
 Euler system constructed in this paper allows us to fill this gap, yielding new evidence towards the Bloch--Kato conjecture for $V_{f,\chi}$ in analytic rank $1$ in this case (i.e., $\epsilon(f/K)=+1$ and $j\geq r$), while also providing a new approach 
to the aforementioned results in the other cases: 
\vspace{5pt}

\begin{center}
\begin{tabular}{c|c|c|}
\cline{2-3}
  & $\epsilon(f/K)=-1$ & $\epsilon(f/K)=+1$ \\
 \hline
\multicolumn{1}{|c}{} & \multicolumn{1}{|c}{1\ts{st} quadrant} & \multicolumn{1}{|c|}{2\ts{nd} quadrant} \\ 
\multicolumn{1}{|c}{$0\leq j<r$} & \multicolumn{1}{|c}{ \cite{kol88}, \cite{BD-crelle}, \cite{Tian-PhD}, \cite{nekovar-CM}, etc.} & \multicolumn{1}{|c|}{\cite{bdIMC}, \cite{ChHs2}, \cite{KLZ}, \cite{KLZ}, etc.} \\ 
\multicolumn{1}{|c}{} & \multicolumn{1}{|c}{Theorem~\ref{thm:BK-indef-1}} & \multicolumn{1}{|c|}{Theorem~\ref{thm:BK-def}} \\
 \hline
\multicolumn{1}{|c}{} & \multicolumn{1}{|c}{3\ts{rd} quadrant} & \multicolumn{1}{|c|}{4\ts{th} quadrant} \\ 
\multicolumn{1}{|c}{$j\geq r$} & \multicolumn{1}{|c}{\cite{cas-hsieh1}, \cite{cas-2var}, \cite{Magr}, \cite{kobayashi-GHC}, etc.} & \multicolumn{1}{|c|}{$-$} \\ 
\multicolumn{1}{|c}{} & \multicolumn{1}{|c}{Theorem~\ref{thm:BK-indef}} & \multicolumn{1}{|c|}{Theorem~\ref{thm:BK-def-1}} \\
\hline
\end{tabular}
\end{center}

\vspace{5pt}

To obtain our anticyclotomic Euler system classes, we exploit   diagonal cycles attached to triple products of modular forms with two of the factors having CM by $K$. This setting is also  considered in the work of the first-named author with Hsieh \cite{cas-hsieh-ord} on the conjectures of Darmon--Rotger \cite{DR2.5} in the `adjoint CM case'.   
The method of construction in this paper using diagonal cycles for triple products has recently been adapted by the second-named author \cite{Do-biquadratic} to the case where two of the factors have CM by \emph{different} imaginary quadratic fields, 
resulting in an anticyclotomic Euler system for modular forms based-changed to a biquadratic CM field, together with results on the Bloch--Kato conjecture, and a divisibility towards the Iwasawa--Greenberg main conjecture. 


In future work, 
we intend to generalize our construction to totally real fields, a setting in which one finds even more cases where the arithmetic of Rankin--Selberg convolutions falls outside the scope of Heegner cycles and/or Beilinson--Flach elements. 

\subsection{Acknowledgements} 
The present article grew out of the second-named author's PhD thesis \cite{Do-PhD}, supervised by Christopher Skinner. We heartily thank him for inspiring this collaboration, his guidance,  and optimism. We would also like to thank Ra\'{u}l Alonso, Haruzo Hida, Kartik Prasanna, \'Oscar Rivero, and Romyar Sharifi for helpful exchanges related to various aspects of this work, and the anomymous referee for a very careful reading of the paper,  whose comments and suggestions helped us to signficantly improve the exposition. 

During the preparation of this paper, the first-named author was partially supported by the NSF grants DMS-1946136, DMS-2101458, and DMS-2401321.

\addtocontents{toc}{\setcounter{tocdepth}{2}}

}

\newpage
\part{The Euler system}
\label{part:ES}
\section{Preliminaries}

\subsection{Modular curves and Hecke operators}
 We give a precise description of the modular curves and Hecke operators 
 that will appear in our construction. The main references for this section are \cite[\S 2]{kato}, \cite[\S 2]{BSV}, and \cite[\S 2]{ACR}, where more details can be found.
 
\subsubsection{Modular curves}Let $M,N,u,v$ be positive integers such that $M+N\ge 5$. Define $Y(M,N)$ to be the affine modular curve over $\Z[1/MN]$ representing the functor
\begin{equation*}
    S\mapsto\left\{
  \begin{array}{ll} \text{isomorphism classes of triples } (E,P,Q) \text{ where $E$ is an elliptic curve over $S$,} \\
  \text{$P$, $Q$ are sections of $E$ over $S$ such that $M\cdot P=N\cdot Q=0$; and the map} \\
  \Z/M\Z\times \Z/N\Z\rightarrow E, \text{ sending }(a,b)\mapsto a\cdot P+b\cdot Q\text{ is injective} \end{array} \right\}
\end{equation*}
for $\Z[1/MN]$-schemes $S$. More generally, define the affine modular curve $Y(M(u),N(v))$ over $\Z[1/MNuv]$ representing the functor
\begin{equation*}
  S\mapsto\left\{
  \begin{array}{ll} \text{isomorphism classes of quintuples } (E,P,Q,C,D) \text{ where }\\ \text{$(E,P,Q)$ is as above,} \\
  \text{$P\in C$ is a cyclic subgroup of $E$ of order $Mu$,} \\ \text{$Q\in D$ is a cyclic subgroup of $E$ of order $Nv$ such that} \\
   \text{$C$ is complementary to $Q$ and $D$ is complementary to $P$} \end{array} \right\}
\end{equation*}
for $\Z[1/MNuv]$-schemes $S$. When either $u=1$ or $v=1$, we drop them from the notation.

Let $\mathbf{H}$ be the Poincar\'e upper half-plane and define the modular group:
\begin{equation*}
    \Gamma(M(u),N(v))=\left\{
  \gamma\in \mathrm{SL}_2(\Z) \text{ such that } \gamma\equiv\begin{pmatrix}
1 & 0\\
0 & 1 
\end{pmatrix}\text{ mod } \begin{pmatrix}
M & Mu\\
Nv & N 
\end{pmatrix} \right\}.
\end{equation*}
The Riemann surface $Y(M,N)(\bC)$ admits a complex uniformization:
\begin{equation*}
\begin{array}{ccc}
       (\Z/M\Z)^{\times}\times \Gamma(M,N)\backslash \mathbf{H}&\xrightarrow{\sim} &Y(M,N)(\bC)  \\
    (m,z) &\mapsto &(\bC/\Z+\Z z,mz/M,1/N), 
\end{array}
\end{equation*}
and similarly for $Y(M(u),N(v))(\bC)$.

Let $\ell$ be a prime. There is an isomorphism of $\Z[1/\ell MN]$-schemes:
\begin{equation*}
    \begin{array}{ccl}
         \varphi_\ell:Y(M,N(\ell))&\rightarrow &Y(M(\ell),N)  \\
        (E,P,Q,C)&\mapsto &(E/NC,P+NC,\ell^{-1}(Q)\cap C+NC,(\ell^{-1}(\Z\cdot P+NC)/NC)),
    \end{array}
\end{equation*}
which under the complex uniformization is induced by the map $(m,z)\mapsto (m,\ell\cdot z)$.

\subsubsection{Degeneracy maps}
We have the natural degeneracy maps
\[
\xymatrix{
    Y(M,N\ell)  \ar[r]^-{\mu_\ell} & Y(M,N(\ell)) \ar[d]_-{\varphi_\ell} \ar[r]^-{\nu_\ell} & Y(M,N)  \\
   Y(M\ell,N) \ar[r]^-{\check{\mu}_\ell} & Y(M(\ell),N) \ar[r]^-{\check{\nu}_\ell} & Y(M,N),  
}
\]
where $\mu_\ell(E,P,Q)=(E,P,\ell\cdot Q,\Z\cdot Q)$, $\nu_\ell(E,P,Q,C)=(E,P,Q)$, and $\check{\mu}_\ell$, $\check{\nu_\ell}$ are defined similarly. Put
\begin{equation}
\begin{split}
\text{pr}_1:=\nu_\ell\circ \mu_\ell: 
Y(M,N\ell)&\rightarrow Y(M,N),    \\
(E,P,Q)&\mapsto (E,P,\ell\cdot Q)
\end{split}\nonumber
\end{equation}
and
\begin{align*}
\text{pr}_\ell:=\check{\nu}_\ell\circ\varphi_\ell\circ \mu_\ell:
Y(M,N\ell)&\rightarrow Y(M,N)   \\
       (E,P,Q)&\mapsto (E/N\Z\cdot Q ,P+N\Z\cdot Q, Q+N\Z\cdot Q). 
  \end{align*}
On the complex upper half plane $\mathbf{H}$, the map $\text{pr}_1$ (resp.   $\text{pr}_\ell$) is induced by the identity (resp. multiplication by $\ell$). Moreover,  $\mu_\ell,\check{\mu}_\ell,\nu_\ell,\check{\nu}_\ell,\text{pr}_1,\text{pr}_\ell$ are all finite \'etale morphisms of $\Z[1/MN\ell]$-schemes.

\subsubsection{Relative Tate modules and Hecke operators}\label{relativeTate}

Let $S$ be a $\Z[1/MN\ell p]$-scheme where $p$ is a fixed prime. For each $\Z[1/MN\ell p]$-scheme $X$, denote the base change \[X_S=X\times_{\Z[1/MN\ell p]}S.\] Notate $A=A_X$ to be either the locally constant sheaf $\Z/p^m\Z(j)$ or the locally constant $p$-adic sheaf $\Z_p(j)$ on $X_{\et}$ 
for some fixed  $j\in \Z$ and $m\geq 1$. 

For the ease of notation, we may write $\cdot$ for $M(u),N(v)$ (i.e. $Y(\cdot)=Y(M(u),N(v))$). Denote by $E(\cdot)$ the universal elliptic curve over $Y(\cdot)$. Under the base change by \[\varphi_\ell^{*}E(M(\ell),N)\rightarrow Y(M,N(\ell)),\] one obtains a natural degree $\ell$ isogeny of universal elliptic curves:
\begin{equation*}
    \lambda_\ell: E(M,N(\ell))\rightarrow \varphi_\ell^{*}(E(M(\ell),N).
\end{equation*}
Denote by $v_{\cdot}: E(\cdot)_S\rightarrow Y(\cdot)_S$ the structure map. We also use $\nu_\ell,\check{\nu}_\ell$ and $\lambda_\ell$ for the base change to $S$ of the corresponding degeneracy maps. Set:
\begin{equation*}
    \mathscr{T}_{\cdot}(A)=R^1v_{\cdot*}\Z_p(1)\otimes_{\Z_p}A \text{  and  } \mathscr{T}_{\cdot}^{*}(A)=\text{Hom}_{A}(\mathscr{T}_{\cdot}(A),A)
\end{equation*}
where $R^qv_{\cdot*}$ is the $q$-th right derivative of $v_{\cdot*}:E(\cdot)_{\et}\rightarrow Y(\cdot)_{\et}$. When $A=\Z_p$, this gives the relative Tate module of the universal elliptic curve, in which case we will drop the $\Z_p$ from the notation. 

Fix an integer $r\ge 0$. The (perfect) cup product pairing combined with the relative trace
\begin{equation*}
\mathscr{T}_{\cdot}\otimes_{\Z_p}\mathscr{T}_{\cdot}\rightarrow R^2v_{\cdot*}\Z_p(2)\cong \Z_p(1)
\end{equation*}
allows us to identify $\mathscr{T}_{\cdot}(-1)$ with $\mathscr{T}_{\cdot}^{*}$. 
Put
\begin{equation*}
    \mathscr{L}_{\cdot,r}(A)=\text{TSym}^{r}_A\mathscr{T}_{\cdot}(A),\quad\mathscr{S}_{\cdot,r}(A)=\text{Sym}^{r}_A\mathscr{T}_{\cdot}^{*}(A),
\end{equation*}
where $\text{TSym}^{r}_R M$ is the $R$-submodule of the symmetric tensors in $M^{\otimes r}$, $\text{Sym}^{r}_R M$ is the maximal symmetric quotient of $M^{\otimes r}$, and $M$ is any finite free module over a profinite $\Z_p$-algebra $R$. When the level is clear, we shall simplify the notations, e.g. writing:
\begin{equation}\label{LlSs}
\begin{aligned}
    \mathscr{L}_r(A)=\mathscr{L}_{M(u),N(v),r}(A),\quad\mathscr{L}_r=\mathscr{L}_r(\Z_p),\\
    \mathscr{S}_r(A)=\mathscr{S}_{M(u),N(v),r}(A),\quad \mathscr{S}_r=\mathscr{S}_r(\Z_p).
\end{aligned}
\end{equation}

Let $\F_{\cdot}^r$ be either $\mathscr{L}_{\cdot,r}(A)$ or $\mathscr{S}_{\cdot,r}(A)$. 
Then there are natural isomorphisms of sheaves
\begin{equation*}
    \nu_\ell^{*}(\F_{M,N}^r)\cong \F_{M,N(\ell)}^r,\quad\check{\nu}_\ell^{*}(\F_{M,N}^r),\cong \F_{M(\ell),N}^r,
\end{equation*}
and these induce pullback maps
\begin{equation*}
\xymatrixrowsep{0.18in}
\xymatrixcolsep{0.06in}
  \xymatrix{
& H^i_{\et}(Y(M,N)_S,\F_{M,N}^r) \ar[dl]_{\nu_{\ell}^{*}} \ar[dr]^{\check{\nu}_{\ell}^{*}} & \\
H^i_{\et}(Y(M,N({\ell}))_S,\F_{M,N({\ell})}^r) &  & H^i_{\et}(Y(M({\ell}),N)_S,\F_{M({\ell}),N}^r) }
\end{equation*}
and

\begin{equation*}
\xymatrixrowsep{0.18in}
\xymatrixcolsep{0.06in}
  \xymatrix{
& H^i_{\et}(Y(M,N)_S,\F_{M,N}^r) & \\
H^i_{\et}(Y(M,N({\ell}))_S,\F_{M,N({\ell})}^r)\ar[ur]^{\nu_{{\ell}*}} &  & \ar[ul]_{\check{\nu}_{{\ell}*}} H^i_{\et}(Y(M({\ell}),N)_S,\F_{M({\ell}),N}^r) }
\end{equation*}
which are trace maps.

The finite \'etale isogeny $\lambda_\ell$ induces morphisms
\begin{equation*}
    \lambda_{\ell*}:\F_{M,N(\ell)}^r\rightarrow \varphi_{\ell}^{*}(\F_{M(\ell),N}^r),\quad  \lambda_{\ell}^{*}:\varphi_{\ell}^{*}(\F_{M(\ell),N}^r) \rightarrow \F_{M,N(\ell)}^r
\end{equation*}
and this allows us to define a pushforward 
\begin{equation*}
    \Phi_{\ell*}:=\varphi_{\ell*}\circ\lambda_{\ell*}: H^i_{\et}(Y(M,N(\ell))_S,\F_{M,N(\ell)}^r)\rightarrow H^i_{\et}(Y(M(\ell),N)_S,\F_{M(\ell),N}^r)
\end{equation*}
and a pullback
\begin{equation*}
\Phi_{\ell}^{*}:=\lambda_{\ell}^{*}\circ\varphi_{\ell}^{*}:H^i_{\et}(Y(M(\ell),N)_S,\F_{M(\ell),N}^r)\rightarrow H^i_{\et}(Y(M,N(\ell))_S,\F_{M,N(\ell)}^r).
\end{equation*}

The Hecke operator $T_\ell$ and the dual Hecke operator $T_\ell'$ acting on $H^i_{\et}(Y(M,N)_S,\F_{M,N}^r)$ are defined  by
\begin{equation*}
T_\ell:=\check{\nu}_{\ell*}\circ\Phi_{\ell*}\circ\nu_\ell^{*},\quad T_\ell':=\nu_{\ell*}\circ\Phi_\ell^{*}\circ\check{\nu}_\ell^{*}.
\end{equation*}

\begin{rmk}
Note the relations
\begin{equation*}
   \text{deg}(\mu_\ell) T_\ell=\text{pr}_{\ell*}\circ \text{pr}_1^{*},\quad\text{deg}(\mu_\ell)T_\ell'=\text{pr}_{1*}\circ\text{pr}_\ell^{*},
\end{equation*}
as it follows immediately from the definitions.
\end{rmk}


For $d\in (\Z/MN\Z)^{*}$, the diamond operator $\langle d\rangle$ on $Y(\cdot)$ is defined in terms of moduli by
\begin{equation*}
    (E,P,Q,C,D)\mapsto (E,d^{-1}\cdot P,d\cdot Q,C,D).
\end{equation*}
There is also a unique diamond operator $\langle d\rangle $ on the universal elliptic curve making the following diagram cartesian:
\begin{equation*}
    \xymatrix{
E(\cdot)_S \ar[d]_{v_{\cdot}} \ar[r]^{\langle d\rangle} & E(\cdot)_S \ar[d]^{v_{\cdot}} \\
Y(\cdot)_S \ar[r]^{\langle d\rangle} & Y(\cdot)_S,}
\end{equation*}
and this induces automorphisms $\langle d\rangle=\langle d\rangle^{*}$ and $\langle d\rangle'=\langle d\rangle_{*}$ on $H^i_{\et}(Y(\cdot)_S,\mathscr{F^r_{\cdot}})$.

For any profinite $\Z_p$-algebra $R$ and finite free $R$-module $M$, the evaluation map induces a perfect pairing
\begin{equation*}
    \text{TSym}^r_RM\otimes_R\text{Sym}^r_R M^{*}\rightarrow R,
\end{equation*}
where $M^{*}=\text{Hom}_R(M,\Z_p)$. This gives a perfect pairing $\mathscr{L}_r\otimes_{\Z_p}\mathscr{S}_r\rightarrow \Z_p$, and therefore a cup product
\begin{equation*}
    \langle \cdot,\cdot\rangle: H^1_{\et}(Y(\cdot)_{\overline{\Q}},\Ll_r(1))\otimes_{\Z_p}H^1_{\et,c}(Y(\cdot)_{\overline{\Q}},\Ss_r)\rightarrow H^2_{\et}(Y(\cdot)_{\overline{\Q}},\Z_p(1))\cong \Z_p,
\end{equation*}
which is perfect 
after inverting $p$. Moreover, the Hecke operators $T_\ell$, $T_\ell'$, $\langle d\rangle$, $\langle d\rangle '$ induce endomorphisms on the compactly supported cohomology groups $H^1_{\et,c}(Y(\cdot)_{\overline{\Q}},\Ss_r)$, and by construction, $(T_\ell,T_\ell')$ and $(\langle d\rangle,\langle d\rangle ')$ are adjoint pairs under $\langle \cdot,\cdot\rangle$. The Eichler--Shimura isomorphism \cite{shimura-pink}
\begin{equation*}
    H^1_{\et}(Y_1(N)_{\overline{\Q}},\Ll_r)\otimes_{\Z_p}\bC\cong M_{r+2}(N,\bC)\oplus \overline{S_{r+2}(N,\bC)}
\end{equation*}
commutes with the action of the Hecke operators on both sides.

\subsection{Galois representations associated to newforms} 
\label{subsec:Galreps}

Let $f=\sum_{n=1}^\infty a_n q^n$ be a normalized newform of weight $k\ge 2$, level $\Gamma_1(N_f)$, and nebentype $\chi_f$. Let $p\nmid N_f$ be an odd prime. Fix embeddings $i_\infty: \overline{\Q}\hookrightarrow \bC$, $i_p: \overline{\Q}\hookrightarrow \overline{\Q}_p$, and let $L/\Q$ be a finite extension containing all values $i_\infty^{-1}(a_n)$ and $i_\infty^{-1}\circ \chi_f$. 
Let $\mathfrak{P}$ be the prime of $L$ above $p$ with respect to $i_p$. 
Then Eichler--Shimura 
and Deligne 
construct a $p$-adic Galois representation associated to $f$:
\begin{equation}\label{ES}
    \rho_f=\rho_{f,\mathfrak{P}}:G_{\Q}\rightarrow {\mathrm{GL}}_2(L_{\mathfrak{P}})\nonumber
\end{equation}
which is unramified outside $pN_f$, and characterized by the property for all finite primes $\ell\nmid pN_f$, 
\[
\mathrm{trace}(\rho_{f}(\mathrm{Frob}_\ell))=i_p(a_\ell),\quad
\det(\rho_{f}(\mathrm{Frob}_\ell))=i_p(\chi_f(\ell)l^{k-1}),
\]
where $\Frob_\ell$ is a geometric Frobenius. Moreover, $\rho_{f,\mathfrak{P}}$ is known to be irreducible \cite{Ri}, hence absolutely irreducible since the image of the complex conjugation has eigenvalues $1$ and $-1$.

\subsubsection{Geometric realizations} The representation $\rho_{f,\mathfrak{P}}$ can be geometrically realized as the largest subspace $V_f$ of
\[
H^1_{\et,c}(Y_1(N_f)_{\overline{\Q}},\Ss_{k-2})\otimes L_{\mathfrak{P}}
\]
on which $T_\ell$ acts as multiplication by $a_\ell$ for all $\ell\nmid N_fp$ and $\langle d\rangle= \langle d\rangle^*$ acts as multiplication by $\chi_f(d)$ for all $d\in (\Z/N_f\Z)^\times$. If $N$ is any multiple of $N_f$, then the above subspace with $N_f$ replaced by $N$ gives rise to a representation $V_f(N)$ isomorphic (non-canonically) to a finite number of copies of $V_f$.

The dual $V_f^{\vee}={\mathrm{Hom}}(V_f,L_{\mathfrak{P}})$ 
 can be interpreted as the maximal quotient of 
\[
H^1_{\et}(Y_1(N_f)_{\overline{\Q}},\Ll_{k-2}(1))\otimes L_{\mathfrak{P}}
\]
on which the dual Hecke operator $T_\ell'$ acts as multiplication by $a_\ell$ for all $\ell\nmid N_fp$ and $\langle d\rangle'= \langle d\rangle_{*}$ acts as multiplication by $\chi_f(d)$ for all $d\in (\Z/N_f\Z)^\times$. 

Let $\cO$ be the ring of integers of $L_\mathfrak{P}$. In this paper we shall be mostly working with $V_f^\vee$ and the $G_{\Q}$-stable $\cO$-lattice $T_f^\vee\subset V_f^\vee$ defined as the image of 
$H^1_{\et}(Y_1(N_f)_{\overline{\Q}},\Ll_{k-2}(1))\otimes\cO$ in $V_f^\vee$.

\subsubsection{The $p$-ordinary case}\label{subsec:p-ord}

If $f$ is ordinary at $p$, i.e. $i_p(a_p)\in\cO^\times$, then the restriction of $V_f$ to $G_{\Q_p}\subset G_{\Q}$ is reducible, fitting into an exact sequence of $L_{\mathfrak{P}}[G_{\Q_p}]$-modules
\begin{equation*}
0\rightarrow V_f^{+}\rightarrow V_f\rightarrow V_f^{-}\rightarrow 0 
\end{equation*}
with $\dim_{L_\mathfrak{P}}V_f^{\pm}=1$, and with $G_{\Q_p}$-action on the subspace $V_f^+$ 
given by the unramified character sending  $\mathrm{Frob}_p$ to $\alpha_p$, the unit root of $x^2-a_px+\chi_f(p)p^{k-1}$. 
%
By duality, we also obtain an exact sequence for $V_f^{\vee}$ restricted to $G_{\Q_p}$
\begin{equation}
0\rightarrow V_f^{\vee,+}\rightarrow V_f^{\vee}\rightarrow V_f^{\vee,-}\rightarrow 0 
\end{equation}
with $V_f^{\vee,+}\simeq(V_f^-)^\vee(1-k)(\chi_f^{-1})$, and with the $G_{\Q_p}$-action on the quotient $V_{f}^{\vee,-}$ given by the unramified character sending arithmetic Frobenius $\Frob_p^{-1}$ to $\alpha_p$.
%

\subsection{Patched CM Hecke modules}\label{LLZ-K}

In this section, after explaining our conventions on Hecke characters, we recall the construction of certain patched CM Hecke modules from \cite[\S{5.2}]{LLZ-K}.

\subsubsection{Hecke characters and theta series }\label{theta}

Let $K$ be an imaginary quadratic field of discriminant $-D_K<0$  in which $p=\pp\ppbar$ splits, with $\pp$ the prime  above $p$ induced by $i_p$. 

We say that a Hecke character $\psi:\mathbb{A}_K^{\times}/K^{\times}\rightarrow \mathbb{C}^{\times}$ has infinity type $(a,b)\in\Z^2$ if $\psi_{\infty}(x_{\infty})=x_{\infty}^a\bar{x}_{\infty}^b$ for all $x_\infty\in K\otimes_{\Q}\mathbf{R}\cong\mathbf{C}$  under the identification induced by $i_\infty$. 
Then the character sending $x\mapsto\psi(x)x_{\infty}^{-a}\bar{x}_{\infty}^{-b}$ is a ray class character, hence it takes value in a finite extension $L/K$. For $\mathfrak{P}\mid\pp$ the prime of $L$ above $p$ induced by $i_p$, we define the \emph{$p$-adic avatar} $\psi_\mathfrak{P}:G_K\rightarrow L_\mathfrak{P}^\times$ of $\psi$ as follows. Let  $\text{rec}_K:\mathbb{A}_K^{\times}/K\rightarrow G_K^{\text{ab}}$ be the \emph{geometrically} normalized Artin reciprocity map. For $g\in G_K$, 
we take $x\in\mathbb{A}_K^{\times}$ such that $\text{rec}_K(x)=g\vert_{K^{\mathrm{ab}}}$ and define
\begin{equation}
        \psi_{\p}(g)=i_p\circ i_{\infty}^{-1}(\psi(x)x_{\infty}^{-a}\bar{x}_{\infty}^{-b})x_{\pa}^a x_{\bar{\pa}}^b.\nonumber
\end{equation}
Since there should be no confusion, in the following we shall also use $\psi$ to denote its $p$-adic avatar.

Let $\psi$ be a Hecke character of $K$ of infinity type $(-1,0)$, conductor dividing $\fkf$, and with $x\mapsto\psi(x)x_{\infty}$ taking values in a finite extension $L/K$. Viewing $\psi$ as an $L$-valued character on the group of fractional ideals of $K$ coprime to $\mathfrak{f}$, the theta series attached to $\psi$ is
\begin{equation*}
\theta_{\psi}=\sum_{(\mathfrak{a},\fkf)=1}\psi(\mathfrak{a})q^{N_{K/\Q}(\mathfrak{a})} \in S_2(\Gamma_1(N_\psi),\chi_\psi\epsilon_K)
\end{equation*}
where $N_\psi=N_{K/\Q}(\fkf)D_K$, $\chi_\psi$ is the unique Dirichlet character modulo $N_{K/\Q}(\fkf)$ such that 
$\psi((n))=n\chi_\psi(n)$ for all $n\in\Z$ with $(n,N_{K/\Q}(\fkf))=1$, and $\epsilon_K$ is the quadratic Dirichlet character attached to $K/\Q$. 
The cuspform $\theta_\psi$ is new of level $N_\psi$ if $\fkf$ is the exact conductor of $\psi$, and its $p$-adic $G_\Q$-representation satisfies
\[
V_{\theta_{\psi}}\cong{\mathrm{Ind}}^{\Q}_KL_{\mathfrak{P}}(\psi),
\quad\quad V_{\theta_{\psi}}^{\vee}\cong {\mathrm{Ind}}^{\Q}_KL_{\mathfrak{P}}(\psi^{-1}).
\]
%

\subsubsection{Hecke algebras and norm maps}\label{LLZrecall}

Let $\fkn\subset\cO_K$ be an ideal divisible by $\fkf$. Put $N=N_{K/\Q}(\fkn)D_K$. 
Let $K_\fkn$ be the ray class field of $K$ with conductor $\fkn$, and $H_\fkn={\mathrm{Gal}}(K_\fkn/K)$ be the ray class group of $K$ modulo $\fkn$. For an ideal $\mathfrak{k}$ of $K$ coprime to $\fkn$, let $[\mathfrak{k}]$ be the class of $\mathfrak{k}$ in $H_{\fkn}$.

We denote by $K(\fkn)$ the largest $p$-extension of $K$ contained in $K_\fkn$, so ${\mathrm{Gal}}(K(\fkn)/K)\cong H_{\fkn}^{(p)}$, where $H_{\fkn}^{(p)}$ is the largest $p$-power quotient of $H_{\fkn}$. 

\begin{prop}[{\cite[Prop.~3.2.1]{LLZ-K}}]\label{phi_n}
Let the subalgebra of ${\mathrm{End}}_{\Z}(H^1(Y_1(N)(\bC),\Z))$ generated by $\langle d\rangle'$ and $T_\ell'$ for all primes $\ell$ be denoted as $\mathbb{T}'(N)$. There exists a homomorphism $\phi_{\fkn}: \mathbb{T}'(N)\rightarrow \mathcal{O}[H_{\fkn}]$ 
defined by 
    \begin{align*}
        \phi_{\fkn}(T_\ell')&=\sum_{\fkl}[\fkl]\psi(\fkl),\\
        \phi_{\fkn}(\langle d\rangle ')&=\chi_\psi(d)\epsilon_{K}(d)[(d)],
    \end{align*}
where the sum is over the ideals $\mathfrak{l}\subset\cO_K$ with $\fkl\nmid\mathfrak{n}$ and $N_{K/\Q}(\fkl)=\ell$.
\end{prop}


For $\fkn'=\fkn\fkl$, with $\fkl$ a prime, put $N'=N_{K/\Q}(\fkn')D_K$. As in \cite[\S{3.3}]{LLZ-K}, we consider the norm maps
\begin{equation}\begin{aligned}
    \mathcal{N}_{\fkn}^{\fkn'}:\mathcal{O}[H_{\fkn'}^{(p)}]&\otimes_{\mathbb{T}'(N')\otimes\Z_p,\phi_{\fkn'}}H^1_{\et}(Y_1(N')_{\overline{\Q}},\Z_p(1))\\ &\rightarrow \mathcal{O}[H_{\fkn}^{(p)}]\otimes_{\mathbb{T}'(N)\otimes\Z_p,\phi_{\fkn}}H^1_{\et}(Y_1(N)_{\overline{\Q}},\Z_p(1))
\end{aligned}
\nonumber
\end{equation}
defined by the following formulae: 
\begin{itemize}
    \item If $\fkl\mid\fkn$ then 
    \begin{equation*}
        \mathcal{N}_{\fkn}^{\fkn'}=1\otimes \text{pr}_{1*};
    \end{equation*}
    \item If $\fkl\nmid\fkn$ is split or ramified in $K$, then 
    \begin{equation*}
        \mathcal{N}_{\fkn}^{\fkn'}=1\otimes \text{pr}_{1*}-\frac{\psi(\mathfrak{l)}[\fkl]}{\ell}\otimes \text{pr}_{\ell*};
    \end{equation*}
    \item If $\fkl\nmid\fkn$ is inert in $K$, say $\fkl=(\ell)$, then 
    \begin{equation*}
        \mathcal{N}_{\fkn}^{\fkn'}=1\otimes \text{pr}_{1*}-\frac{\psi(\fkl)[\fkl]}{\ell^2}\otimes \text{pr}_{\ell\ell*},
    \end{equation*}
    where ${\mathrm{pr}}_{\ell\ell}:Y(N')\rightarrow Y(N)$ denotes the degeneracy map induced by $z\mapsto\ell^2z$, for $z$ on the complex upper half plane $\mathbf{H}$.
\end{itemize}
By composition, the definition of $\mathcal{N}_{\fkn}^{\fkn'}$ is extended to any pair of ideals $\fkn\mid\fkn'$.

\subsubsection{Patching}

Note that since $p$ splits in $K$, the restriction $V_{\theta_\psi}\vert_{G_{\bQ_p}}$ is reducible.

\begin{defn}\label{def:psi-dist}
We say that $\psi$ is \emph{$p$-distinguished} if the residual representation $\bar{\rho}_{\theta_\psi}$ is such that $\bar{\rho}_{\theta_\psi}\vert_{G_{\Q_p}}$ has non-scalar semi-simplification. And we say that $(\psi,\fkf)$ satisfies \emph{Condition~$\spadesuit$} if the following hold:
\begin{itemize}
\item $\psi$ is $p$-distinguished;
\item if $\pp\mid\fkf$ then $\ppbar\nmid\fkf$.
\end{itemize}
\end{defn}



Recall the map $\phi_{\fkn}$ from Proposition~\ref{phi_n}. It follows from the proof of \cite[Prop.~5.1.2]{LLZ-K} that if $(\psi,\fkf)$ satisfies Condition~$\spadesuit$, then for any ideal $\fkn\subset\cO_K$ divisible by $\mathfrak{f}$ and $\pp$ and with $(\fkn,\ppbar)=1$, the maximal ideal of $\mathbb{T}'(N)$ defined by the kernel of the composite map
\[
\mathbb{T}'(N)\overset{\phi_\fkn}\longrightarrow\cO[H_\fkn]\overset{{\mathrm{aug}}}\longrightarrow\cO\rightarrow\cO/\mathfrak{P},
\]
is non-Eisenstein, $\mathfrak{P}$-ordinary, and $p$-distinguished in the sense of Definitions~4.1.2 and 4.3.3 of \cite{LLZ-K}. Through its use in [\emph{op.\,cit.}, Prop.~5.2.5], this condition is to allow $\fkf$ and $\fkm$ to be possibly divisible by $\pp$ in the following result.



\begin{thm}\label{norm1} 
Let $\mathcal{A}^{(\ppbar)}$ be the set of ideals $\fkm\subset\cO_K$  with $\ppbar\nmid\fkm$, and put 
\[
\mathcal{A}:=\begin{cases}
\mathcal{A}^{(\ppbar)}&\textrm{if $(\psi,\fkf)$ satisfies Condition~$\spadesuit$,}\\[0.3em]
\{\fkm\in\mathcal{A}^{(\ppbar)}\,\colon\;\pp\nmid\fkm\}&\textrm{otherwise},
\end{cases}
\]
and $\mathcal{A}_\fkf=\{\fkf\fkm\;\colon\;\fkm\in\mathcal{A}\}$. 
Then there is a family of $G_{\Q}$-equivariant isomorphisms of $\mathcal{O}[H_{\fkn}^{(p)}]$-modules 
\begin{equation*}
    \nu_{\fkn}:\mathcal{O}[H_{\fkn}^{(p)}]\otimes_{\mathbb{T}'(N)\otimes\Z_p,\phi_{\fkn}}H^1_{\et}(Y_1(N)_{\overline{\Q}},\Z_p(1))\overset{\cong}{\longrightarrow}{\mathrm{Ind}}^{\Q}_{K(\fkn)}\mathcal{O}(\psi_{\mathfrak{P}}^{-1})
\end{equation*}
indexed by $\fkn\in \mathcal{A}_{\mathfrak{f}}$, such that for any $\fkn,\fkn'\in \mathcal{A}_\mathfrak{f}$ with $\fkn\mid\fkn'$ the following diagram commutes:
\begin{equation*}
    \xymatrix{
\mathcal{O}[H_{\fkn'}^{(p)}]\otimes_{\mathbb{T}'(N')\otimes\Z_p,\phi_{\fkn'}}H^1_{\et}(Y_1(N')_{\overline{\Q}},\Z_p(1)) \ar[d]_{\mathcal{N}_{\fkn}^{\fkn'}} \ar[r]^(.7){\nu_{\fkn'}}_(.7){\cong} & {\mathrm{Ind}}^{\Q}_{K(\fkn')}\mathcal{O}(\psi_{\mathfrak{P}}^{-1}) \ar[d]_{\mathrm{Norm}_{\fkn}^{\fkn'}} \\
\mathcal{O}[H_{\fkn}^{(p)}]\otimes_{\mathbb{T}'(N)\otimes\Z_p,\phi_{\fkn}}H^1_{\et}(Y_1(N)_{\overline{\Q}},\Z_p(1)) \ar[r]^(.7){ \nu_{\fkn}}_(.7){\cong} & {\mathrm{Ind}}^{\Q}_{K(\fkn)}\mathcal{O}(\psi_{\mathfrak{P}}^{-1}), }
\end{equation*}
where $\mathrm{Norm}_{\fkn}^{\fkn'}$ is the natural norm map.
\end{thm}

\begin{proof}
This is a reformulation of  Corollary 5.2.6 in \cite{LLZ-K} in the case where $p$ splits in $K$. 
\end{proof}
 

\subsection{Diagonal classes}\label{BSVcons}
We sketch the construction of the diagonal classes in the triple product of modular curves $Y_1(N)$ using classical invariant theory, following Section $3$ in \cite{BSV}.

With the notations of $\S\ref{relativeTate}$, we put $Y_1(N)=Y_1(N)_{\Q}$, and let $E_1(N)=E_1(N)_{\Q}$ denote the universal elliptic curve over $Y_1(N)$, together with the structural map $v: E_1(N)\rightarrow Y_1(N)$. Let $\mathscr{T}=R^1v_{*}\Z_p(1)$ be the relative Tate  module of the universal elliptic curve, and let $\mathscr{T}^{*}=\text{Hom}_{\Z_p}(\mathscr{T},\Z_p)$ be its dual. The cup product pairing combined with the relative trace
\begin{equation*}
    \mathscr{T}\otimes_{\Z_p}\mathscr{T}\rightarrow R^2v_{*}\Z_p(2)\cong \Z_p(1)
\end{equation*}
gives a perfect relative Weil pairing
\begin{equation*}
\langle-,-\rangle_{E_1(N)_{p^{\infty}}}:\mathscr{T}\otimes_{\Z_p}\mathscr{T}\rightarrow \Z_p(1),
\end{equation*}
which allows $\mathscr{T}(-1)$ to be identified  with $\mathscr{T}^{*}$. 
%
%

For a fixed geometric point $\eta:\text{Spec}(\overline{\Q})\rightarrow Y_1(N)$, denote by $\mathcal{G}_{\eta}=\pi_1^{\et}(Y_1(N),\eta)$ the fundamental group of $Y_1(N)$ with base point $\eta$. The stalk of $\mathscr{T}$ at $\eta$, denoted $\mathscr{T}_{\eta}$, is a free $\Z_p$-module of rank $2$, equipped with a continuous action of $\mathcal{G}_{\eta}$. Fix a choice of $\Z_p$-module isomorphism $\zeta:\mathscr{T}_{\eta}\cong \Z_p\oplus \Z_p$ such that $\langle x,y\rangle_{E_1(N)_{p^{\infty}}}=\zeta(x)\wedge \zeta(y)$ (where we identify $\bigwedge^2 \Z_p^2$ with $\Z_p$ via $(1,0)\wedge (0,1) = 1$). One then obtains a continuous group homomorphism:
\begin{equation*}
    \rho_{\eta}:\mathcal{G}_{\eta}\rightarrow \text{Aut}_{\Z_p}(\mathscr{T}_{\eta})\cong {\mathrm{GL}}_2(\Z_p).
\end{equation*}
By \cite[Prop.~A I.8]{FK}, the category of locally constant $p$-adic sheaves on $Y_1(N)_{\et}$ is equivalent to the category of $p$-adic representations of $\mathcal{G}_{\eta}$ via  $\mathscr{F}\mapsto \mathscr{F}_{\eta}$. Using $\rho_{\eta}$, one can associate with every continuous representation of ${\mathrm{GL}}_2(\Z_p)$ on a free finite $\Z_p$-module $M$ a smooth sheaf $M^{\et}$ on $Y_1(N)$ such that $M^{\et}_{\eta}=M$.

Let $S_i(A)$ be the set of $2$-variable homogeneous polynomials of degree $i$ in $A[x_1,x_2]$ equipped with the action of ${\mathrm{GL}}_2(\Z_p)$ by $gP(x_1,x_2)=P((x_1,x_2)\cdot g)$ for all $g\in{\mathrm{GL}}_2(\Z_p)$ and $P\in S_i(A)$. Its $A$-linear dual $L_i(A)$ is also equipped with a ${\mathrm{GL}}_2(\Z_p)$-action by \[g\tau(P(x_1,x_2))=\tau(g^{-1}P(x_1,x_2))\] for all $g\in{\mathrm{GL}}_2(\Z_p)$, $P\in S_i(A)$, and $\tau\in L_i(A)$. As sheaves on $Y_1(N)_{\Q}$, one has
\begin{equation}
    L_i(A)^{\et}=\Ll_i(A) \quad \text{and} \quad S_i(A)^{\et}=\Ss_i(A).\nonumber
\end{equation}
Hence $\mathscr{T}_{\eta}\cong L_1(\Z_p)$ and $\Z_p(1)_{\eta}\cong\bigwedge^2\mathscr{T}_{\eta}\cong \det^{-1}$. This implies that for any $j\in \Z$, and any $p$-adic representation $M$ of ${\mathrm{GL}}_2(\Z_p)$:
\begin{equation}\label{equation33}
    H^0({\mathrm{GL}}_2(\Z_p),M\otimes \text{det}^{-j})\hookrightarrow H^0(\mathcal{G}_{\eta},M\otimes \text{det}^{-j})\cong H^0_{\et}(Y_1(N),M^{\et}(j)).
\end{equation}

\begin{assumption}\label{assumption2.2}
Let $\mathbf{r}=(r_1,r_2,r_3)$ be such that $r_i\in\Z_{\ge 0}$, $(r_1+r_2+r_3)/2=r\in\Z_{\ge 0}$, and $r_i+r_j\ge r_k$ for all permutations $(i,j,k)$ of $(1,2,3)$. We call this a \emph{balanced triple}. 
\end{assumption}

Under the Assumption~\ref{assumption2.2}, let 
\[S_{\mathbf{r}}=S_{r_1}(\Z_p)\otimes_{\Z_p}S_{r_2}(\Z_p)\otimes_{\Z_p}S_{r_3}(\Z_p)\]
with its natural ${\mathrm{GL}}_2(\Z_p)$-representation from above, and let 
\[\Ss_{\mathbf{r}}=S_{\mathbf{r}}^{\et}=\Ss_{r_1}(\Z_p)\otimes_{\Z_p}\Ss_{r_2}(\Z_p)\otimes_{\Z_p}\Ss_{r_3}(\Z_p).\]
We identify $S_{\mathbf{r}}$ with the module of $6$-variable polynomials $\Z_p[x_1,x_2,y_1,y_2,z_1,z_2]$ which are homogeneous of degree $r_1$, $r_2$, and $r_3$ in the variables $(x_1,x_2)$, $(y_1,y_2)$, and $(z_1,z_2)$ respectively. By the Clebsch--Gordan decomposition of classical invariant theory, the following is a ${\mathrm{GL}}_2(\Z_p)$-invariant of $S_{\mathbf{r}}\otimes \det^{-r}$: 
\begin{equation*}
\text{Det}_N^{\mathbf{r}}:=\det\begin{pmatrix}x_1&x_2 \\ y_1&y_2
    \end{pmatrix}^{r-r_3}\det\begin{pmatrix}x_1&x_2 \\ z_1&z_2
    \end{pmatrix}^{r-r_2}\det\begin{pmatrix}y_1&y_2 \\ z_1&z_2
    \end{pmatrix}^{r-r_1},
\end{equation*}
i.e. $\text{Det}_N^{\mathbf{r}}\in H^0({\mathrm{GL}}_2(\Z_p),S_{\mathbf{r}}\otimes\det^{-r})$; we denote its image under (\ref{equation33}) as
\begin{equation}
    \mathtt{Det}^{\mathbf{r}}_N\in H^0_{\et}(Y_1(N),\Ss_{\mathbf{r}}(r)).
\end{equation}

Let $p_j:Y_1(N)^3\rightarrow Y_1(N)$ for $j\in\{1,2,3\}$ be the natural projections and denote 
\begin{gather*}
\Ss_{[\mathbf{r}]}:=p_1^{*}\Ss_{r_1}(\Z_p)\otimes_{\Z_p}p_2^{*}\Ss_{r_2}(\Z_p)\otimes_{\Z_p}p_3^{*}\Ss_{r_3}(\Z_p),
\end{gather*}
and
\begin{gather*}
\mathtt{W}_{N,\mathbf{r}}:=H^3_{\et}(Y_1(N)^3_{\overline{\Q}},\Ss_{[\mathbf{r}]}(r+2)),\quad W_{N,\mathbf{r}}=\mathtt{W}_{N,\mathbf{r}}\otimes\Q_p.
\end{gather*}
Because $Y_1(N)_{\overline{\Q}}$ is a smooth affine curve over $\overline{\Q}$, we have $H^4_{\et}(Y_1(N)^3_{\overline{\Q}},\Ss_{[\mathbf{r}]}(r+2))=0$. Hence by the Hochschild--Serre spectral sequence, 
\[H^p(\Q,H^q_{\et}(Y_1(N)^3_{\overline{\Q}},\Ss_{[\mathbf{r}]}(r+2)))\Longrightarrow H^{p+q}_{\et}(Y_1(N)^3_{\overline{\Q}},\Ss_{[\mathbf{r}]}(r+2))\]
one obtains
\begin{equation*}
    \mathtt{HS}:H^4_{\et}(Y_1(N)^3,\Ss_{[\mathbf{r}]}(r+2))\rightarrow H^1(\Q,\mathtt{W}_{N,\mathbf{r}}).
\end{equation*}
If we let $d:Y_1(N)\rightarrow Y_1(N)^3$ be the diagonal embedding, then there is a natural isomorphism $d^{*}\Ss_{[\mathbf{r}]}\cong \Ss_{\mathbf{r}}$ of smooth sheaves on $Y_1(N)_{\et}$. As $d$ is an embedding of codimension $2$, there is a pushforward map
\begin{equation*}
    d_{*}:H^0_{\et}(Y_1(N),\Ss_{\mathbf{r}}(r))\rightarrow H^4_{\et}(Y_1(N)^3,\Ss_{\mathbf{r}}(r+2)),
\end{equation*}
and we can consider the class
\begin{equation*}
    (\mathtt{HS}\circ d_{*})(\mathtt{Det}^{\mathbf{r}}_N)\in H^1(\Q,\mathtt{W}_{N,\mathbf{r}}).
\end{equation*}
Dually, by the bilinear form $\det^{*}:L_i(\Z_p)\otimes_{\Z_p}L_i(\Z_p)\rightarrow \Z_p\otimes \det^{-i}$ defined by $\det^{*}(\tau\otimes \sigma)=\tau\otimes\sigma((x_1y_2-x_2y_1)^i)$ that becomes perfect after inverting $p$, we obtain an isomorphism of ${\mathrm{GL}}_2(\Z_p)$-modules
\begin{equation}\label{srlr}
\mathbf{s}_i:S_i(\Q_p)\cong L_i(\Q_p)\otimes\text{det}^i, 
\end{equation}
and so $\mathbf{s}_i:\Ss_i(\Q_p)\cong \Ll_i(\Q_p)\otimes_{\Z_p}\Z_p(-i)$ by the above equivalence of categories. We similarly define the sheaves $\Ll_{\mathbf{r}}$ on $Y_1(N)$ and $\Ll_{[\mathbf{r}]}$ on $Y_1(N)^3$, and set
\begin{equation}
    \mathtt{V}_{N,\mathbf{r}}:=H^3_{\et}(Y_1(N)^3_{\overline{\Q}},\Ll_{[\mathbf{r}]}(2-r)),\quad V_{N,\mathbf{r}}=
    \mathtt{V}_{N,\mathbf{r}}\otimes \Q_p.\nonumber
\end{equation}

Let $\mathbf{s}_{\mathbf{r}}=\mathbf{s}_{\mathbf{r}_1}\otimes \mathbf{s}_{\mathbf{r}_2}\otimes \mathbf{s}_{\mathbf{r}_3}$, which gives an isomorphism $W_{N,\mathbf{r}}\rightarrow V_{N,\mathbf{r}}$, and finally as in \cite{BSV} put
\begin{equation}\label{BSVconstructclass}
    \kappa_{N,\mathbf{r}}:=(\mathbf{s}_{\mathbf{r}*}\circ\mathtt{HS}\circ d_{*})(\mathtt{Det}^{\mathbf{r}}_N)\in H^1(\Q,V_{N,\mathbf{r}}).
\end{equation}
As explained in detail in  [\emph{loc.\,cit.}, $\S${3.2}], the class  $\kappa_{N,\mathbf{r}}$ is closely related to the $p$-adic \'etale Abel--Jacobi image of the generalized Gross--Kudla--Schoen diagonal cycles on Kuga--Sato varieties studied in \cite{DR1}.


\begin{prop}\label{table}
For a prime number $\ell$ and a positive integer $m$, if $(m\ell,N)=1$ then
\begin{equation*}
({\mathrm{pr}}_{i*},{\mathrm{pr}}_{j*},{\mathrm{pr}}_{k*})\kappa_{Nm\ell,\textbf{r}}= (\bigstar)\kappa_{Nm,\textbf{r}}
\end{equation*}
where
\begin{center}
\begin{tabular}{|c|c| } 
 \hline
 $(i,j,k)$ & $\bigstar$  \\ 
 \hline
  $(\ell,1,1)$& $(\ell-1)(T_\ell,1,1)$  \\ 
  $(1,\ell,1)$&  $(\ell-1)(1,T_\ell,1)$     \\
  $(1,1,\ell)$& $(\ell-1)(1,1,T_\ell)$ \\
  $(1,\ell,\ell)$& $\ell^{r-r_1}(\ell-1)(T_\ell',1,1)$\\
  $(\ell,1,\ell)$& $\ell^{r-r_2}(\ell-1)(1,T_\ell',1)$ \\
  $(\ell,\ell,1)$& $\ell^{r-r_3}(\ell-1)(1,1,T_\ell')$ \\
 \hline
\end{tabular}
\end{center}
If $(\ell,m)=1$ then we also have
\begin{center}
\begin{tabular}{|c|c| } 
 \hline
 $(i,j,k)$ & $\bigstar$  \\ 
 \hline
  $(1,1,1)$& $(\ell^2-1)$  \\ 
  $(\ell,\ell,\ell)$&  $(\ell^2-1)\ell^r$     \\
 
 \hline
\end{tabular}
\end{center}
\end{prop}

\begin{proof}
See equations (174) and (176) in \cite{BSV}. (The proof of the above relations in \emph{op.\,cit.} is given for $\ell=p$, but the same argument applies for any prime $\ell$ as above.)
\end{proof}



\section{Main theorems}\label{sec:main-thms}


In this section, for a newform $f\in S_{2r}(\Gamma_0(N_f))$ of weight $2r\geq 2$ and a family of anticyclotomic Hecke characters $\chi$ of $K$, 
we construct a family of cohomology classes for 
the conjugate self-dual representation 
\[
V_{f,\chi}:=V_f^\vee(1-r)\otimes\chi^{-1}
\]
defined over ring class extensions of $K$, and prove that they  satisfy the Euler system norm relations. 

The construction builds on the results from \cite{BSV,LLZ-K} recalled in the preceding section, and is done in two steps. We consider (suitable modifications of) diagonal classes attached to triples $(f,\theta_{\psi_1},\theta_{\psi_2})$ consisting of our fixed newform $f$ and a pair of theta series $\theta_{\psi_1},\theta_{\psi_2}$ attached to Hecke characters $\psi_1,\psi_2$ of $K$ satisfying the self-duality condition
\[
\chi_{\psi_1}\chi_{\psi_2}=1,
\]
and first give the construction in the case where $(f,\theta_{\psi_1},\theta_{\psi_2})$ have weights $(2,2,2)$, resulting in the construction of classes 
\begin{equation}\label{eq:wt2-classes}
z_{f,\psi_1,\psi_2,m}\in H^1(K[m],T_{f,\psi_1\psi_2\mathbf{N}^{-1}}),\quad\quad
{}^{\cc}z_{f,\psi_1,\psi_2,m}\in H^1(K[m],T_{f,\psi_1\psi_2^\cc\mathbf{N}^{-1}}),\nonumber
\end{equation}
where $\psi_2^\cc$ is the composition of $\psi_2$ with the action of the non-trivial automorphism of $K/\Q$, for which we prove the tame norm relations (see Theorem~\ref{maintheorem1} and Corollary~\ref{maincor1}). 

Replacing $(\theta_{\psi_1},\theta_{\psi_2})$ by a pair of CM Hida families $(\boldsymbol{\theta}_{\xi_1}(Z_1),\boldsymbol{\theta}_{\xi_2}(Z_2))$ attached to ray class characters $\xi_1,\xi_2$ of $K$ satisfying $\chi_{\xi_1}\chi_{\xi_2}=1$, and considering (suitable modifications of) their associated `big' diagonal classes, we extend the construction 
to all weights $2r\geq 2$ and more general anticyclotomic Hecke characters, and deduce the proof of the wild norm relations (see Theorem~\ref{maintheorem2}).

Throughout we use the notations introduced in the preceding sections, so in particular $K$ is an imaginary quadratic field such that
\begin{equation}
\textrm{$p=\pp\ppbar$ splits in $K$},\tag{spl}
\end{equation}
with $\pp$ the prime of $K$ above $p$ induced by $i_p$. Further, we assume that $p\nmid 6N_f$ and that
\begin{equation}\label{eq:p-nmid-h}
\textrm{$p\nmid h_K$,}\tag{cn}
\end{equation}
where $h_K=\vert H_1\vert$ is the class number of $K$

\subsection{Construction for weight (2,2,2)}
\label{subsec:tame}

Suppose $f$ has weight $2$, and $\psi_1,\psi_2$ are Hecke characters of $K$ of infinity type $(-1,0)$ and modulus $\fkf_1,\fkf_2\subset\cO_K$ with 
\[
(p,\fkf_i)=1
\]
for $i=1,2$, and satisfying $\chi_{\psi_1}\chi_{\psi_2}=1$. Let $N=\text{lcm}(N_f,N_{\psi_1},N_{\psi_2})$ and for every positive integer $m$ put 
\begin{equation*}
    {Y(m):=Y(1,Nm)=Y_1(Nm)}.
\end{equation*} 
When $m=1$, we drop it from the notation, so $Y:=Y_1(N)$. We begin with the cohomology class
\begin{equation}\label{kappaBSV}
\tilde{\kappa}_m^{(1)}:=\kappa_{Nm,\mathbf{r}}\in H^1\bigl(\Q,H^3_{\et}(Y(m)^3_{\overline{\Q}},\Z_p(2)\bigr)
\end{equation}
of \eqref{BSVconstructclass}, where $\mathbf{r}=(0,0,0)$, and put  
\begin{equation}
    \tilde{\kappa}_m^{(2)}=({\mathrm{pr}}_{m*},1,1)\tilde{\kappa}_{m}^{(1)}\in H^1\bigl(\Q,H^3_{\et}(Y_{\overline{\Q}}\times Y(m)^2_{\overline{\Q}},\Z_p(2)\bigr),\nonumber
\end{equation}
where, writing $m=\prod_i\ell_i$ as a product of primes $\ell_i$, $\text{pr}_{m*}$ is defined as the composition of the pushforward by the degeneracy maps $\text{pr}_{\ell_i}$. Applying the K\"unneth decomposition theorem \cite[Thm.~22.4]{milne-lecture}, 
the class $\tilde{\kappa}_{m}^{(2)}$ is projected to
\begin{equation}\label{kappa3}
\tilde{\kappa}_{m}^{(3)}\in H^1\bigl(\Q,H^1_{\et}(Y_{\overline{\Q}},\Z_p(1))\otimes H^1_{\et}(Y(m)_{\overline{\Q}},\Z_p(1))\otimes H^1_{\et}(Y(m)_{\overline{\Q}},\Z_p(1))(-1)\bigr).
\end{equation}

Now we restrict to squarefree integers $m>0$ divisible only by primes split in $K$ with $(m,pN)=1$, and write 
\[
m=\fkm\fkmbar
\]
according to a \emph{fixed} choice of splitting $\ell=\mathfrak{l}\bar{\mathfrak{l}}$ for each prime $\ell\mid m$. 
We also fix a triple of level-$N$ test vectors
\begin{equation}\label{eq:triple-222}
(\breve{f},\breve{\theta}_{\psi_1},\breve{\theta}_{\psi_2})\in S_2(\Gamma_0(N))[f]\times S_2(\Gamma_1(N))[\theta_{\psi_1}]\times S_2(\Gamma_1(N))[\theta_{\psi_2}].\nonumber
\end{equation}
(Even though it will not be reflected in the notation, our construction will depend on this choice; in later applications we shall specify the choice of test vectors when needed.) 

Since the maps used in the construction $\tilde{\kappa}^{(3)}_m$ are compatible with correspondences, after tensoring with $\cO$ and taking $(f,\theta_{\psi_1},\theta_{\psi_2})$-isotypic components, the above process and the choice of test vectors 
give rise to a class
\begin{align*}
\tilde{\kappa}_{f,\psi_1,\psi_2,m}^{(4)}\in H^1\bigl(\Q,T_{f}^{\vee}&\otimes H^1_{\et}(Y_1(N_{\psi_1}m)_{\overline{\Q}},\Z_p(1))\otimes_{\mathbb{T}'(N_{\psi_1}m)}\mathcal{O}[H_{\mathfrak{f}_1\fkm}^{(p)}]\\
 &\quad\otimes H^1_{\et}(Y_1(N_{\psi_2}m)_{\overline{\Q}},\Z_p(1))\otimes_{\mathbb{T}'(N_{\psi_2}m)}\mathcal{O}[H_{\mathfrak{f}_2\fkmbar}^{(p)}](-1)\bigr),
\end{align*}
where the labeled tensor products are with respect to the Hecke algebra homomorphisms
\begin{equation}\label{eq:rayclass-str}
\phi_{\mathfrak{f}_1\fkm}:\mathbb{T}'(N_{\psi_1}m)\rightarrow\mathcal{O}[H_{\mathfrak{f}_1\fkm}^{(p)}],\quad 
\phi_{\mathfrak{f}_2\fkmbar}:\mathbb{T}'(N_{\psi_2}m)\rightarrow\mathcal{O}[H_{\mathfrak{f}_2\fkmbar}^{(p)}]
\end{equation}
of Proposition~\ref{phi_n}, and we used $\breve{f}$ to take the image under the projection \[H^1_{\et}(Y_{\overline{\Q}},\Z_p(1))[f]\rightarrow T_f^\vee\] in the first factor, and similarly for \[H_{\et}^1(Y(m)_{\overline{\Q}},\Z_p(1))[\theta_{\psi_i}]\rightarrow H^1_{\et}(Y_1(N_{\psi_i}m)_{\overline{\Q}},\Z_p(1))\] using $\breve{\theta}_{\psi_i}$, $i=1,2$.
%

By the isomorphisms from Proposition~\ref{norm1}:
\begin{equation}\label{eq:CM-patching}
\begin{aligned}
\nu_{\fkf_1\fkm}:H^1_{\et}(Y_1(N_{\psi_1}m)_{\overline{\Q}},\Z_p(1))\otimes_{\mathbb{T}'(N_{\psi_1}m),\phi_{\fkf_1\fkm}}\mathcal{O}[H_{\mathfrak{f}_1\fkm}^{(p)}]&\xrightarrow{\sim}{\mathrm{Ind}}^{\Q}_{K(\fkf_1\fkm)}\mathcal{O}(\psi_{1}^{-1}),\\
\nu_{\fkf_2\fkmbar}:H^1_{\et}(Y_1(N_{\psi_2}m)_{\overline{\Q}},\Z_p(1))\otimes_{\mathbb{T}'(N_{\psi_2}m),\phi_{\fkf_2\fkmbar}}\mathcal{O}[H_{\mathfrak{f}_2\fkmbar}^{(p)}]&\xrightarrow{\sim}{\mathrm{Ind}}^{\Q}_{K(\fkf_2\fkmbar)}\mathcal{O}(\psi_{2}^{-1}),
\end{aligned}
\end{equation}
the class $\tilde{\kappa}_{f,\psi_1,\psi_2,m}^{(4)}$ defines an element in 
\begin{align*}
&H^1\bigl(\Q,T_f^{\vee}\otimes_{\mathcal{O}}\text{Ind}^{\Q}_{K(\fkf_1\fkm)}\mathcal{O}(\psi_1^{-1})\otimes_{\mathcal{O}}\text{Ind}^{\Q}_{K(\fkf_2\fkmbar)}\mathcal{O}(\psi_2^{-1})(-1)\bigr)\\
&\cong H^1\bigl(\Q,T_f^{\vee}\otimes_{\mathcal{O}}\text{Ind}^{\Q}_K\mathcal{O}_{\psi_1^{-1}}[H_{\fkf_1\fkm}^{(p)}]\otimes_{\mathcal{O}}\text{Ind}^{\Q}_{K}\mathcal{O}_{\psi_2^{-1}}[H_{\fkf_2\fkmbar}^{(p)}](-1)\bigr),
\end{align*}
where for the second factor in the triple product we use the $G_K$-module isomorphism ${\mathrm{Ind}}_{K(\fkf_1\fkm)}^K\cO(\psi_1^{-1})\cong\cO_{\psi_1^{-1}}[H_{\fkf_1\fkm}^{(p)}]$, with $\cO_{\psi_1^{-1}}$ standing for the free $\cO$-module of rank one on which $G_K$ acts via $\psi_1^{-1}$ and with $H_{\fkf_1\fkm}$ being equipped with the $G_K$-action arising from the projection $G_K\twoheadrightarrow H_{\fkf_1\fkm}^{(p)}$; and likewise for the last factor in the triple product. Taking the image under the maps induced by $
H_{\fkf_1\fkm}\twoheadrightarrow H_{\fkm}$ and $H_{\fkf_2\fkmbar}\twoheadrightarrow H_{\fkmbar}$, from  $\tilde{\kappa}_{f,\psi_1,\psi_2,m}^{(4)}$ we obtain 
\begin{equation}\label{eq:wrongclass-5}
\tilde{\kappa}_{f,\psi_1,\psi_2,m}^{(5)}\in H^1\bigl(\Q,T_f^{\vee}\otimes_{\mathcal{O}}{\mathrm{Ind}}^{\Q}_{K}\mathcal{O}_{\psi_1^{-1}}[H_{\fkm}^{(p)}]\otimes_{\mathcal{O}}{\mathrm{Ind}}^{\Q}_{K}\mathcal{O}_{\psi_2^{-1}}[H_{\fkmbar}^{(p)}](-1)\bigr).
\end{equation}

\subsubsection{Projection to ring class groups}

Directly from the definitions of the class groups involved, we deduce the commutative diagram with exact rows
\[
\xymatrix{
\cO_K^\times\times\cO_K^\times\ar[r]&(\cO_K/\fkm)^\times\times(\cO_K/\overline{\fkm})^\times\ar[r]&H_{\fkm}\times H_{\overline{\fkm}}\ar[r]&H_1\times H_1\ar[r]&1\\
\cO_K^\times\ar[u]^-{\Delta}\ar[r]&(\cO_K/m\cO_K)^\times\ar[u]^{\simeq}\ar[r]&H_{m}\ar[u]\ar[r]&H_1\ar[u]^-{\Delta}\ar[r]&1,
}
\]
where the unlabeled vertical arrow is given by the restriction map 
\[\sigma\mapsto(\sigma\vert_{K_\fkm},\sigma\vert_{K_{\fkmbar}}).
\] 
Since we assume $p\nmid 6h_K$, taking $p$-primary parts this map induces an isomorphism
\begin{equation}\label{eq:H-m-decompose}
H_m^{(p)}\xrightarrow{\simeq} H_{\fkm}^{(p)}\times H_{\fkmbar}^{(p)}.
\end{equation}

Given an integer $n>0$, let $H[n]$ be the ring class group of $K$ of conductor $n$, so 
$H[n]\simeq{\mathrm{Pic}}(\cO_n)$ under the Artin reciprocity map, where $\cO_n=\Z+n\cO_K$ is the order of $K$ of conductor $n$. Let $H[n]^{(p)}$ be the maximal $p$-power quotient of $H[n]$, and denote by $K[n]$ be the maximal $p$-extension inside the ring class field of $K$ of conductor $n$, so $H[n]^{(p)}={\mathrm{Gal}}(K[n]/K)$.

\begin{prop}\label{ringclass} Suppose $p\nmid 6h_K$ and $m>0$ is an integer divisible only by primes split in $K$. Then writing $m=\fkm\fkmbar$ and identifying $H_m^{(p)}$ with $H_\fkm^{(p)}\times H_{\fkmbar}^{(p)}$ as in (\ref{eq:H-m-decompose}), we have an exact sequence
\[
 1\longrightarrow (\Z/m\Z)^{\times, (p)}\overset{\Delta}\longrightarrow H_{\fkm}^{(p)}\times H_{\fkmbar}^{(p)} \overset{\pi_\Delta}\longrightarrow H[m]^{(p)}\longrightarrow 1,
 \]
where the map $\Delta$ sends $a\mapsto([a],[a])$ for every integer $a$ coprime to $m$. Moreover, if $\ell\nmid m$ is a prime that splits in $K$, the projection $\pi_\Delta$ sends 
\begin{align*}
[\lambda]\times[\lambda]&\mapsto \mathrm{Frob}_{\lambda} 
\end{align*}
for every prime $\lambda$ of $K$ above $\ell$, where $\mathrm{Frob}_{\lambda}$ is the geometric Frobenius element of $\lambda$ in $H[m]^{(p)}$. 
\end{prop}

\begin{proof}
The first part is clear from the above discussion  together with the commutative diagram with exact rows
\[
\xymatrix{
\cO_K^\times\ar[r]\ar[d]&(\cO_K/m\cO_K)^\times\ar[r]\ar[d]&H_m\ar[r]\ar[d]&H_1\ar[r]\ar[d]&1\\
\cO_K^\times/\Z^\times\ar[r]&(\cO_K/m\cO_K)^\times/(\Z/m\Z)^\times\ar[r]&H[m]\ar[r]&H_1\ar[r]&1,
}
\]
where the vertical arrows are given by the natural projections. The second part follows from the functoriality properties of Frobenii.
\end{proof}

Now we can consider the image of $\tilde{\kappa}_{f,\psi_1,\psi_2,m}^{(5)}$ under the composite map
\begin{equation}\label{ind2}
\begin{aligned}
     \xymatrix{
\text{Ind}^{\Q}_{K}\mathcal{O}_{\psi_{1}^{-1}}[H_{\fkm}^{(p)}]\otimes_{\mathcal{O}}\text{Ind}^{\Q}_{K}\mathcal{O}_{\psi_{2}^{-1}}[H_{\fkmbar}^{(p)}] \ar@{..>}[dr]_{\xi_\Delta} \ar[r]^(.56){\xi} & \text{Ind}^{\Q}_{K}\mathcal{O}_{\psi_{1}^{-1}\psi_{2}^{-1}}[H_{\fkm}^{(p)}\times H_{\fkmbar}^{(p)}] \ar[d]^{\pi_\Delta} \\
 & \text{Ind}^{\Q}_{K}\mathcal{O}_{\psi_{1}^{-1}\psi_{2}^{-1}}[H[m]^{(p)}], }
 \end{aligned}
\end{equation}
where the horizontal arrow is the map determined by $\phi_1\otimes \phi_2\mapsto\xi(\phi_1\otimes \phi_2)$ with $\xi(\phi_1\otimes\phi_2)(g)=\phi_1(g_1)\otimes \phi_2(g_2)$ if $g=(g_1,g_2)\in H_{\fkm}^{(p)}\times H_{\fkmbar}^{(p)}$, resulting in the class 
\[
\tilde{\kappa}_{f,\psi_1,\psi_2,m}^{(6)}\in H^1\bigl(\Q,T_f^{\vee}\otimes_{\mathcal{O}}\text{Ind}^{\Q}_{K}\mathcal{O}_{\psi_{1}^{-1}\psi_{2}^{-1}}[H[m]^{(p)}](-1)\bigr).
\]

\begin{defn}\label{def:tame-classes}
For $m>0$ any squarefree integer divisible only by primes $\ell\nmid pN$ split in $K$, we define
\[
\tilde{\kappa}_{f,\psi_1,\psi_2,m}\in H^1\bigl(K[m],T_f^{\vee}(\psi_{1}^{-1}\psi_{2}^{-1})(-1)\bigr)
\]
to be the image of $\tilde{\kappa}_{f,\psi_1,\psi_2,m}^{(6)}$ under the isomorphism 
\[H^1\bigl(\Q,T_f^{\vee}\otimes_{\mathcal{O}}\text{Ind}^{\Q}_{K}\mathcal{O}_{\psi_{1}^{-1}\psi_{2}^{-1}}[H[m]^{(p)}](-1)\bigr)\simeq H^1\bigl(K[m],T_f^{\vee}(\psi_{1}^{-1}\psi_{2}^{-1})(-1)\bigr)
\]
given by Shapiro's lemma.
\end{defn}

We finish this section by recording the following observation for our later use.

\begin{lem}\label{keydiagram}
For any integer $m=\fkm\fkmbar$ divisible only by primes split in $K$, and any prime $\ell=\fkl\fklbar$ split in $K$, the following diagram is commutative:
\[
\xymatrix{
 {\mathrm{Ind}}^{\Q}_{K}\mathcal{O}_{\psi_{1}^{-1}}[H_{\fkm\fkl}^{(p)}]\otimes_{\mathcal{O}}{\mathrm{Ind}}^{\Q}_{K}\mathcal{O}_{\psi_{2}^{-1}}[H_{\fkmbar\fklbar}^{(p)}] \ar[d]^{\mathrm{Norm}^{\fkm\fkl}_{\fkm}\otimes\mathrm{Norm}^{\fkmbar\fklbar}_{\fkmbar}} \ar[r] &{\mathrm{Ind}}^{\Q}_{K}\mathcal{O}_{\psi_{1}^{-1}\psi_{2}^{-1}}[H[m\ell]^{(p)}] \ar[d]^{\mathrm{Norm}^{m\ell}_{m}} \\
{\mathrm{Ind}}^{\Q}_{K}\mathcal{O}_{\psi_{1}^{-1}}[H_{\fkm}^{(p)}]\otimes_{\mathcal{O}}{\mathrm{Ind}}^{\Q}_{K}\mathcal{O}_{\psi_{2}^{-1}}[H_{\fkmbar}^{(p)}] \ar[r] & {\mathrm{Ind}}^{\Q}_{K}\mathcal{O}_{\psi_{1}^{-1}\psi_{2}^{-1}}[H[m]^{(p)}],}
\]
where the horizonal arrows are given by the composition $\xi_\Delta$ in  (\ref{ind2}).
\end{lem}

\begin{proof}
This is clear from the explicit description of the maps involved.
\end{proof}

\subsection{Proof of the tame norm relations}

Let $m>0$ be any integer for which we have the class $\tilde{\kappa}_{f,\psi_1,\psi_2,m}$ of Definition \ref{def:tame-classes}. 



\begin{prop}\label{wrongnorm}
Let $\ell=\fkl\fklbar$ be a prime split in $K$ and coprime to $mpN$. Then
\begin{align*}
\Norm_{K[m]}^{K[m\ell]}(\tilde{\kappa}_{f,\psi_1,\psi_2,m\ell})=(\ell-1)\bigg(&a_\ell(f)-\frac{\psi_1(\fkl)\psi_2(\fkl)}{\ell}([\fkl]\times [\fkl])\\-\frac{\psi_1(\fklbar)\psi_2(\fklbar)}{\ell}([\fklbar]\times [\fklbar])
   &+(1-\ell)\frac{\psi_1(\fkl)\psi_2(\fklbar)}{\ell^2}([\fkl]\times[\fklbar])\bigg)
    (\tilde{\kappa}_{f,\psi_1,\psi_2,m}).
\end{align*}
\end{prop}

\begin{proof}
In the notations of Theorem \ref{norm1}, for any $\fkn=\fkf\fkm\in\mathcal{A}_\fkf$ put
\[
H^1(\psi,\fkf\fkm):=H^1_{\et}(Y_1(N_{\psi}m)_{\overline{\Q}},\Z_p(1))\otimes_{\mathbb{T}'(N_{\psi}m)}\mathcal{O}[H_{\mathfrak{f}\fkm}^{(p)}].
\]
Then from Lemma \ref{keydiagram} we have the following commutative diagram:
\begin{equation}\label{norm5}
\begin{aligned}
    \xymatrix{
H^1\bigl(\Q,T_f^\vee\otimes H^1(\psi_1,\fkf_1\fkm\fkl)\otimes H^1(\psi_2,\fkf_2\fkmbar\fklbar)(-1)\bigr) \ar[d]^{1\otimes \mathcal{N}^{\fkm\fkl}_{\fkm}\otimes\mathcal{N}^{\fkmbar\fklbar}_{\fkmbar}} \ar[r] &H^1\bigl(K[m\ell],T_f^{\vee}(\psi_{1}^{-1}\psi_{2}^{-1})(-1)\bigr) \ar[d]^{\mathrm{Norm}^{m\ell}_{m}} \\
 H^1\bigl(\Q,T_f^\vee\otimes H^1(\psi_1,\fkf_1\fkm)\otimes H^1(\psi_2,\fkf_2\fkmbar)(-1)\bigr)
\ar[r] & H^1\bigl(K[m],T_f^{\vee}(\psi_{1}^{-1}\psi_{2}^{-1})(-1)\bigr),}
\end{aligned}
\end{equation}
where, for $\fkm'\in\{\fkm\fkl,\fkm\}$, the horizontal arrows are given by the maps induced by the projections
\begin{align*}
H^1(\psi_1,\fkf_1\fkm')&\xrightarrow{\nu_{\fkf_1\fkm'}}{\mathrm{Ind}}_{K(\fkf_1\fkm')}^\Q\cO(\psi_1^{-1})\simeq {\mathrm{Ind}}_K^\Q\cO_{\psi_1^{-1}}[H_{\fkf_1\fkm'}^{(p)}]\longrightarrow 
{\mathrm{Ind}}_K^\Q\cO_{\psi_1^{-1}}[H_{\fkm'}^{(p)}]\\
H^1(\psi_2,\fkfbar_1\fkmbar')&\xrightarrow{\nu_{\fkf_2\fkmbar'}}{\mathrm{Ind}}_{K(\fkf_2\fkmbar')}^\Q\cO(\psi_2^{-1})\simeq {\mathrm{Ind}}_K^\Q\cO_{\psi_2^{-1}}[H_{\fkf_2\fkmbar'}^{(p)}]\longrightarrow {\mathrm{Ind}}_K^\Q\cO_{\psi_2^{-1}}[H_{\fkmbar'}^{(p)}]
\end{align*}
from Theorem~\ref{norm1} together with the map $\xi_\Delta$ in \eqref{ind2}. Now, tracing through the definitions we compute:
\begin{align*}
&(1\otimes\mathcal{N}_{\fkm}^{\mathfrak{ml}}\otimes\mathcal{N}_{\fkmbar}^{\fkmbar\fklbar})(\tilde{\kappa}^{(2)}_{m\ell})\\
&=(1\otimes\mathcal{N}_{\fkm}^{\mathfrak{ml}}\otimes\mathcal{N}_{\fkmbar}^{\fkmbar\fklbar})(\tpr_{m\ell*},1,1)(\tilde{\kappa}^{(1)}_{m\ell}) \\
&=(\tpr_{m*},1,1)(\text{pr}_{\ell*}\otimes \mathcal{N}_{\fkm}^{\mathfrak{ml}}\otimes\mathcal{N}_{\fkmbar}^{\fkmbar\fklbar})(\tilde{\kappa}^{(1)}_{m\ell})\\
&=(\tpr_{m*},1,1)\biggl(\tpr_{\ell*}\times(1\otimes \text{pr}_{1*}-\frac{\psi_1(\fkl)[\fkl]}{\ell}\otimes \text{pr}_{\ell*})\times(1\otimes \text{pr}_{1*}-\frac{\psi_2(\bar{\mathfrak{l}})[\bar{\mathfrak{l}}]}{\ell}\otimes \text{pr}_{\ell*})\biggr)(\tilde{\kappa}^{(1)}_{m\ell})\\
&=(\tpr_{m*},1,1)\biggl((\text{pr}_{\ell*},{\mathrm{pr}}_{1*},{\mathrm{pr}}_{1*})-\frac{\psi_1(\fkl)[\fkl]}{\ell}(\text{pr}_{\ell*},\text{pr}_{\ell*},{\mathrm{pr}}_{1*})-\frac{\psi_2(\fklbar)[\fklbar]}{\ell}(\text{pr}_{\ell*},{\mathrm{pr}}_{1*},\text{pr}_{\ell*})\\
&\quad+\frac{\psi_1(\fkl)\psi_2(\fklbar)}{\ell^2}([\fkl]\times[\fklbar])(\text{pr}_{\ell*},\text{pr}_{\ell*},\text{pr}_{\ell*})\biggr)(\tilde{\kappa}^{(1)}_{m\ell}).
\end{align*}
Together with Proposition \ref{table}, we thus obtain
\begin{align*}
(1\otimes\mathcal{N}_{\fkm}^{\mathfrak{ml}}\otimes\mathcal{N}_{\fkmbar}^{\fkmbar\fklbar})(\tilde{\kappa}^{(2)}_{m\ell})
&=(\ell-1)(\tpr_{m*},1,1)\bigg((T_\ell,1,1)-\frac{\psi_1(\fkl)[\fkl]}{\ell}(1,1,T_\ell') \\ &-(1,T_\ell',1)\frac{\psi_2(\fklbar)[\fklbar]}{\ell}\quad+\frac{\psi_1(\fkl)\psi_2(\fklbar)}{\ell^2}([\fkl]\times[\fklbar])(\ell+1)\bigg)(\tilde{\kappa}^{(1)}_{m})\\
&=(\ell-1)\bigg((T_\ell,1,1)-\frac{\psi_1(\fkl)[\fkl]}{\ell}(1,1,T_\ell')\\ &-(1,T_\ell',1)\frac{\psi_2(\fklbar)[\fklbar]}{\ell}\quad+\frac{\psi_1(\fkl)\psi_2(\fklbar)}{\ell^2}([\fkl]\times[\fklbar])(\ell+1)\bigg)(\tilde{\kappa}^{(2)}_{m}),
\end{align*}
and from this it follows that
\begin{align*}
&(1\otimes\mathcal{N}_{\fkm}^{\mathfrak{ml}}\otimes\mathcal{N}_{\fkmbar}^{\fkmbar\fklbar})(\tilde{\kappa}^{(5)}_{f,\psi_1,\psi_2,m\ell})\\
     &=(\ell-1)\bigg(a_\ell(f)-\frac{\psi_1(\fkl)[\fkl]}{\ell}(\psi_2(\fkl)[\fkl]+\psi_2(\fklbar)[\fklbar])-(\psi_1(\fkl)[\fkl]+\psi_1(\fklbar)[\fklbar])\frac{\psi_2(\fklbar)[\fklbar]}{\ell}\\
     &\quad+\frac{\psi_1(\fkl)\psi_2(\fklbar)}{\ell^2}([\fkl]\times[\fklbar])(\ell+1)\bigg)(\tilde{\kappa}^{(5)}_{f,\psi_1,\psi_2,m})\\
     &=(\ell-1)\bigg(a_\ell(f)-\frac{\psi_1(\fkl)\psi_2(\fkl)}{\ell}([\fkl]\times [\fkl])-\frac{\psi_1(\fklbar)\psi_2(\fklbar)}{\ell}([\fklbar]\times [\fklbar])\\
     &\quad+(1-\ell)\frac{\psi_1(\fkl)\psi_2(\fklbar)}{\ell^2}([\fkl]\times[\fklbar])\bigg)(\tilde{\kappa}^{(5)}_{f,\psi_1,\psi_2,m}).
\end{align*} 
In light of the commutative diagram (\ref{norm5}), this yields the result.
\end{proof}

\begin{rmk}
The appearance of the factor $(\ell-1)$ in Proposition~\ref{wrongnorm} can be traced back to the relations $\text{deg}(\mu_\ell) T_\ell=\text{pr}_{\ell*}\circ \text{pr}_1^{*}$ and $\text{deg}(\mu_\ell)T_\ell'=\text{pr}_{1*}\circ\text{pr}_\ell^{*}$, 
i.e., it is caused by the degeneracy map $\mu_\ell$.  In the next subsection we shall get rid of this extra factor.
\end{rmk}

\begin{rmk} We want to emphasize that Proposition~\ref{wrongnorm} is the key result for the construction of our anticyclotomic Euler system for $T_f^\vee(\psi_1^{-1}\psi_2^{-1})(-1)$.
In fact, with the factor $(\ell-1)$ stripped out, the term on the right-hand side of Proposition~\ref{wrongnorm} can be massaged to agree with the local Euler factor at $\fkl$ of the Galois representation $[T_f^\vee(\psi_1^{-1}\psi_2^{-1})(-1)]^\vee(1)=T_f(\psi_1\psi_2)(2)$, giving the correct norm relations. 
\end{rmk}

\subsubsection{Removing the extra factor $(\ell-1)$}\label{thefix} 

Adapting an idea from \cite[$\S{1.4}$]{DR2}, we now introduce a modification of the classes $\tilde{\kappa}_{f,\psi_1,\psi_2,m}$ for which we can prove an analogue of Proposition \ref{wrongnorm} without the extra factor $(\ell-1)$. 

We begin by noting that for any prime $\ell\nmid N$ the degeneracy maps ${\mathrm{pr}}_1,{\mathrm{pr}}_\ell:Y_1(N\ell)\rightarrow Y_1(N)$ can be factored as
\[
\xymatrix{
Y_1(N\ell) \ar[dr]^{\text{pr}_1} \ar[d]_{\mu_\ell}  \\
Y(1,N(\ell))\ar[r]_{\pi_1} & Y_1(N) } \quad\quad\quad
\xymatrix{
Y_1(N\ell) \ar[dr]^{\text{pr}_\ell} \ar[d]_{\mu_\ell}  \\
Y(1,N(\ell))\ar[r]_{\pi_\ell} & Y_1(N), }
\]
where 
$\pi_1$ and $\pi_\ell$ are a non-Galois coverings of degree $\ell+1$, and we recall that $\mu_\ell$ is a cyclic Galois covering of degree $\ell-1$.

Denote by 
\begin{equation}\label{eq:diamond-Dm}
D_m=\{(\langle d\rangle,\langle d\rangle): d\in (\Z/Nm\Z)^{\times},\; d\equiv 1\pmod{N}\}
\end{equation}
the set of diamond operators acting diagonally and freely on $Y_1(Nm)^2$, and set 
\[
W_1(Nm)=(Y_1(Nm)\times Y_1(Nm))/D_m.
\]
Let $\tilde{\kappa}_m^{(1)}$ be as in \eqref{kappaBSV}, and let 
\[
\kappa_{m}^{(1)}\in H^1\bigl(\Q,H^3_{\et}(Y(1,N(m))_{\overline{\Q}}\times W_1(Nm)_{\overline{\Q}},\Z_p)(2)\bigr)
\]
be the image of $(\mu_{m*},1,1)(\tilde{\kappa}_m^{(1)})$ under the
natural map induced by the projection \[Y_1(Nm)^2\rightarrow W_1(Nm),\] which is an \'etale morphism of degree $\phi(m)=\vert(\Z/m\Z)^\times\vert$. 
Thus, the class $\kappa_m^{(1)}$ is defined by the relation
\begin{equation}\label{def-relation}
(\mu_{m*},d_{m*})\tilde{\kappa}^{(1)}_{m}=\phi(m)\kappa^{(1)}_m.
\end{equation}

\begin{prop}\label{correcttable}
For any prime number $\ell$ and positive integer $m$ such that $(m,\ell)=1$ and $(m\ell,N)=1$, we have
\begin{equation*}
    (\pi_{i*},{\mathrm{pr}}_{j*},{\mathrm{pr}}_{k*})\kappa_{m\ell}^{(1)}= (\bigstar)\kappa_{m}^{(1)},
\end{equation*}
where
\begin{center}
\begin{tabular}{|c|c|c|c| } 
 \hline
 $(i,j,k)$ & $\bigstar$  & $(i,j,k)$ & $\bigstar$ \\ 
 \hline
  $(\ell,1,1)$& $(T_\ell,1,1)$ &  $(\ell,1,\ell)$& $(1,T_\ell',1)$  \\ 
  $(1,\ell,1)$&  $(1,T_\ell,1)$  & $(\ell,\ell,1)$& $(1,1,T_\ell')$   \\
  $(1,1,\ell)$& $(1,1,T_\ell)$ & $(1,1,1)$& $(\ell+1)$ \\
  $(1,\ell,\ell)$& $(T_\ell',1,1)$ &  $(\ell,\ell,\ell)$&  $(\ell+1)$  \\
 \hline
\end{tabular}
\end{center}
\end{prop}

\begin{proof}
Directly from the definitions we find
\begin{align*}
(\mu_{m*},d_{m*})({\mathrm{pr}}_{\ell*},{\mathrm{pr}}_{1*},{\mathrm{pr}}_{1*})\tilde{\kappa}_{m\ell}^{(1)}&
=(\pi_{\ell*},{\mathrm{pr}}_{1*},{\mathrm{pr}}_{1*})(\mu_{m\ell*},d_{m\ell*})\tilde{\kappa}_{m\ell}^{(1)}\\
&=\phi(m\ell)(\pi_{\ell*},{\mathrm{pr}}_{1*},{\mathrm{pr}}_{1*})\kappa_{m\ell}^{(1)},
\end{align*}
while on the other hand, from Proposition \ref{table} and (\ref{def-relation}) we have
\begin{align*}
(\mu_{m*},d_{m*})({\mathrm{pr}}_{\ell*},{\mathrm{pr}}_{1*},{\mathrm{pr}}_{1*})\tilde{\kappa}_{m\ell}^{(1)}&=(\mu_{m*},d_{m*})(\ell-1)(T_\ell,1,1)\tilde{\kappa}_m^{(1)}\\
&=\phi(m)(\ell-1)(T_\ell,1,1)\kappa_m^{(1)}.
\end{align*}

Since $\phi(m\ell)=(\ell-1)\phi(m)$ according to our assumptions, this shows the result in the case $(i,j,k)=(\ell,1,1)$ and the other cases are shown in the same manner.
\end{proof}

Now we want to proceed as above to obtain from the new $\kappa_m^{(1)}$ a construction of classes satisfying the correct norm relations (i.e., without the factor $\ell-1$). This requires a careful study of the \'etale cohomology of the quotient $Y(1,N(m))\times W_1(Nm)$.

We begin with the Hochschild--Serre spectral sequence:
\begin{equation*}\begin{aligned}
     E_2^{p,q}=H^p\bigl(D_m,H^q_{\et,c}(Y(1,N(m))_{\overline{\Q}}\times Y_1(Nm)^2_{\overline{\Q}},\Z_p)\bigr)\\ \Rightarrow H^{p+q}_{\et,c}\bigl(Y(1,N(m))_{\overline{\Q}}\times W_1(Nm)_{\overline{\Q}},\Z_p\bigr).
\end{aligned}
\end{equation*}
This yields the exact sequence
\begin{equation}\label{eq:HS}
    E\longrightarrow H^3_{\et,c}\bigl(Y(1,N(m))_{\overline{\Q}}\times W_1(Nm)_{\overline{\Q}},\Z_p\bigr)\xrightarrow{(1,d_m^{*})} E_2^{0,3}\xrightarrow{d_2^{0,3}} E_2^{2,2},
\end{equation}
where $E$ is naturally identified with a subquotient of $E_2^{1,2}\oplus E_2^{2,1}$. Thus, we see that the difference between the two middle pieces are classes coming from $H^q_{\et,c}(Y(1,N(m))_{\overline{\Q}}\times Y_1(Nm)^2_{\overline{\Q}},\Z_p)$ with $0\leq q\le 2$. From the K\"unneth decomposition \cite[Thm.~22.4]{milne-lecture}, each of these classes will have a factor from either $H^0_{\et,c}(Y(1,N(m))_{\overline{\Q}},\Z_p)$ or $H^0_{\et,c}(Y_1(Nm)_{\overline{\Q}},\Z_p)$. These vanish because {\'e}tale cohomology of affine smooth curves with compact support vanishes in degree zero.
Hence, we obtain an isomorphism
\begin{equation}\label{dmpullback}
    H^3_{\et,c}\bigl(Y(1,N(m))_{\overline{\Q}}\times W_1(Nm)_{\overline{\Q}},\Z_p\bigr)\xrightarrow{(1,d_m^{*})}H^3_{\et,c}\bigl(Y(1,N(m))_{\overline{\Q}}\times Y_1(Nm)^2_{\overline{\Q}},\Z_p)\bigr)^{D_m}.\nonumber
\end{equation}

As in the proof of \cite[Lem.~1.8]{DR2}, Poincar\'e duality implies from (\ref{eq:HS}) that the following map
\begin{equation}\label{dmpushforward}
    H^3_{\et}\bigl(Y(1,N(m))_{\overline{\Q}}\times W_1(Nm)_{\overline{\Q}},\Z_p\bigr)\xleftarrow{(1,d_{m*})} H^3_{\et}\bigl(Y(1,N(m))_{\overline{\Q}}\times Y_1(Nm)^2_{\overline{\Q}},\Z_p)\bigr)_{D_m}
\end{equation}
is also an isomorphism.

Therefore from (\ref{dmpushforward}) and the K\"unneth decomposition, it follows that we get a natural map
\begin{multline}\label{eq:loc-I}
     H^3_{\et}(Y(1,N(m))_{\overline{\Q}}\times W_1(Nm)_{\overline{\Q}},\Z_p) \xrightarrow{(1,d_{m*}^{-1})} H^1_{\et}(Y(1,N(m))_{\overline{\Q}},\Z_p)\otimes \\ H^1_{\et}( Y_1(Nm)_{\overline{\Q}},\Z_p)\otimes_{D_m} H^1_{\et}(Y_1(Nm)_{\overline{\Q}},\Z_p)).
\end{multline}


Now we put $\kappa_m^{(2)}:=(\pi_{m*},1,1)\kappa_m^{(1)}$, and define
\[
\kappa_m^{(3)}\in H^1\bigl(\Q,H^1_{\et}(Y_1(N)_{\overline{\Q}},\Z_p)\otimes  H^1_{\et}( Y_1(Nm)_{\overline{\Q}},\Z_p)\otimes_{D_m} H^1_{\et}(Y_1(Nm)_{\overline{\Q}},\Z_p)(-1)\bigr),
\]
to be the image of $\kappa_m^{(2)}$ under the map (\ref{eq:loc-I}). 

Note that taking $D_m^{(p)}$-coinvariants (where $D_m^{(p)}$ denotes the $p$-part of $D_m$) is compatible with the map $\xi_\Delta$ in (\ref{ind2}), since 
by Theorem \ref{phi_n}  
for $(\langle d\rangle,\langle d\rangle)\in D_m^{(p)}$ we have $\phi_{\fkm}(\langle d\rangle ')\times\phi_{\fkmbar}(\langle d\rangle ')=[d]\times[d]\in H_{\fkm}^{(p)}\times H_{\fkmbar}^{(p)}$, and this is in the kernel of $\pi_\Delta$. 
Thus, applying to $\kappa_m^{(3)}$ the same process we used above to go from $\tilde{\kappa}_m^{(3)}$ to the class $\tilde{\kappa}_{f,\psi_1,\psi_2,m}$ of Definition \ref{def:tame-classes}, we obtain  
\[
{\kappa}_{f,\psi_1,\psi_2,m}\in H^1\bigl(K[m],T_f^{\vee}(\psi_{1}^{-1}\psi_2^{-1})(-1)\bigr).
\]



\begin{prop}\label{correctnorm}
Suppose $m$ is a positive squarefree integer divisible only by primes $q\nmid pN$ split in $K$. Let $\ell\nmid pmN$ be a prime split in $K$. Then
\begin{align*}
\Norm_{K[m]}^{K[m\ell]}({\kappa}_{f,\psi_1,\psi_2,m\ell})=\bigg(&a_\ell(f)-\frac{\psi_1(\fkl)\psi_2(\fkl)}{\ell}([\fkl]\times [\fkl])-\frac{\psi_1(\fklbar)\psi_2(\fklbar)}{\ell}([\fklbar]\times [\fklbar])\\
   &+(1-\ell)\frac{\psi_1(\fkl)\psi_2(\fklbar)}{\ell^2}([\fkl]\times[\fklbar])\bigg)
    ({\kappa}_{f,\psi_1,\psi_2,m}).
\end{align*}
\end{prop}

\begin{proof}
After the above discussion, the same calculation as in the proof of Proposition \ref{wrongnorm} applies, replacing the use of Proposition \ref{table} by Proposition \ref{correcttable}.
\end{proof}

Therefore, we arrive at the following theorem:

\begin{thm}\label{maintheorem1}
Suppose $p>3$ is a prime split in $K$ with $p\nmid h_K$. Let $m=\fkm\fkmbar$ run over the squarefree integers divisible only by primes split in $K$ and coprime to $pN$. 
Then there exists a collection of cohomology classes
\[
z_{f,\psi_1,\psi_2,m}\in H^1\bigl(K[m],T_f^{\vee}(\psi_{1}^{-1}\psi_{2}^{-1})(-1)\bigr)
\]
such that for every prime $\ell=\fkl\fklbar$ split in $K$ with $(\ell,mpN)=1$ we have the norm relation
\begin{equation*}
    \mathrm{Norm}_{K[m]}^{K[m\ell]}(z_{f,\psi_1,\psi_2,m\ell})=P_{\fkl}(\mathrm{Frob}_{\fkl})(z_{f,\psi_1,\psi_2,m}),
\end{equation*}
where $P_{\fkl}(X)=\det(1-X\cdot \mathrm{Frob}_{\fkl}\,|\,T_f(\psi_1\psi_2)(2))$.
\end{thm}

\begin{proof}
Denote by $Q_\fkl$ the factor appearing in the right-hand side of Proposition \ref{correctnorm}. Recalling that $[\fkl]\times[\fkl]$ corresponds to $\Frob_\fkl\in H[m]^{(p)}$ under the map $\pi_\Delta$ of Proposition \ref{ringclass}, we find the following congruences:
\begin{align*}
&-\psi_1\psi_2(\fkl)([\fkl]\times [\fkl])\cdot Q_{\fkl}\\
&\equiv
-a_\ell(f)\psi_1\psi_2(\fkl)([\fkl]\times[\fkl])+\frac{\psi_1\psi_2(\fkl)^2}{\ell}([\fkl]\times[\fkl])^2+\frac{\psi_1\psi_2((\ell))}{\ell}([\ell]\times[\ell])\\
&\equiv P_\fkl(\Frob_\fkl)\pmod{\ell-1},
\end{align*}
 as endomorphisms of $H^1(K[m],T_f^{\vee}(\psi_{1}^{-1}\psi_{2}^{-1})(-1))$. Here, we use the relation $\psi_1\psi_2((\ell))=\chi_{\psi_1}\chi_{\psi_2}(\ell)\ell^2=\ell^2$ and the fact that $[\ell]\times[\ell]$ is in the kernel of $\pi_\Delta$ for the second congruence. Therefore, by Lemmas 9.6.1 and 9.6.3 in \cite{Rubin-ES} (which will not alter the bottom class of our Euler system), the existence of classes $z_{f,\psi_1,\psi_2,m}$ with $z_{f,\psi_1,\psi_2,1}=\kappa_{f,\psi_1,\psi_2,1}$ and satisfying the stated norm relations follows from Proposition~\ref{correctnorm}. 
\end{proof}

\begin{rmk}
Similar to what we did for $\fkl$ a split prime of $K$, when $\fkl = (\ell)$ is inert in $K$, we also obtain such a norm relation like in Theorem \ref{maintheorem1}. Remember that in this case, we push forward from level $N\ell^2$ to level $N$. First, note that the norm map from Proposition \ref{phi_n} is then given by
 \begin{equation*}
\mathcal{N}_{\fkn}^{\fkn(\ell)}=1\otimes \text{pr}_{1*}-\frac{\psi(\ell)[(\ell)]}{\ell^2}\otimes \text{pr}_{\ell\ell*}
\end{equation*}
Second, to calculate $(1\otimes\mathcal{N}_{\fkm}^{\mathfrak{m\ell}}\otimes\mathcal{N}_{\fkmbar}^{\fkmbar\ell})(\kappa_{m\ell}^{(2)})$, just like in Proposition \ref{wrongnorm}, we use the table in Proposition \ref{table} together with
\begin{align*}
    (\text{pr}_{\ell*},\pr1,\pr1)(T_{\ell},1,1)\kappa^{(2)}_{m\ell}&=\{(T_{\ell}^2,1,1)-(\ell+1)(\langle \ell\rangle ,1,1)\}\kappa^{(2)}_{m},\\
        (\text{pr}_{\ell*},\text{pr}_{\ell*},\pr1)(1,1,T_{\ell}')\kappa^{(2)}_{m\ell}&=\{(1,1,T_{\ell}'^2)-(\ell+1)(1,1,\langle \ell\rangle ')\}\kappa^{(2)}_{m},
\end{align*}
and arrive at 
\begin{equation*}
\begin{aligned}
     \mathrm{Norm}_{K[m]}^{K[m\ell]}(\kappa^{(6)}_{f,\psi_1,\psi_2,m\ell})= &(\ell-1)\big(a_{\ell}(f)^2-(\ell+1)\\ &-\frac{2(\ell+1)}{\ell}[\ell]\times [\ell]+(\ell+1)[\ell]\times [\ell]\big)
    (\kappa^{(6)}_{f,\psi_1,\psi_2,m}).
\end{aligned}
\end{equation*}
Instead of Proposition \ref{ringclass}, we use the following exact sequence
\begin{equation*}
    1\longrightarrow H_m\overset{\Delta}\longrightarrow H_m\times H_m\longrightarrow H_m\longrightarrow 1, 
\end{equation*}
combining with the quotient $H_m\rightarrow H[m]$, which makes $[\ell]\times [\ell]$ acting trivially on the cohomology class. After removing the extra factor $(\ell-1)$ and multiplying with $-1$ on the RHS factor, we obtain the correct Euler factor modulo $\ell^2-1$:
\begin{equation*}
     P_{\fkl}(\mathrm{Frob}_{\fkl}) = 2+2\ell-a_{\ell}(f)^2.
\end{equation*}
Note that $(\ell+1)/\ell\equiv (\ell +1)\ell=\ell^2+\ell\equiv 1+\ell \pmod{\ell^2-1}$ and the twist $\psi_1(\ell)\psi_2(\ell)/\ell^2=1$.

\end{rmk}


\subsection{A variant}\label{variant}
With a slight modification of the construction in the preceding  sections, we can obtain a similar collection of cohomology classes for $T_f^\vee(\psi_1^{-1}\psi_2^{-\cc})(-1)$  satisfying the corresponding norm relations. Here  $\psi_2^\cc(\sigma)=\psi_2(\cc\sigma\cc^{-1})$ denotes the composition of $\psi_2$ with the action of the non-trivial automorphism of $K/\Q$. Indeed, following Section \ref{subsec:tame}, we replace the second map in \eqref{eq:rayclass-str} 
by 
 \[
\phi_{\fkf_2\fkm}:\mathbb{T}'(N_{\psi_2}m)\rightarrow\cO[H_{\fkf_2\fkm}^{(p)}]
\]
and the second map in \eqref{eq:CM-patching} 
by
\[\nu_{\fkf_2\fkm}:H^1_{\et}(Y_1(N_{\psi_2}m)_{\overline{\Q}},\Z_p(1))\otimes_{\mathbb{T}'(N_{\psi_2}m),\phi_{\fkf_2\fkm}}\mathcal{O}[H_{\mathfrak{f}_2\fkm}^{(p)}]\xrightarrow{\sim}{\mathrm{Ind}}^{\Q}_{K(\fkf_2\fkm)}\mathcal{O}(\psi_{2}^{-1}).\]
Thus, similarly as in \eqref{eq:wrongclass-5} we obtain 
\[
{}^\cc\tilde{\kappa}_{f,\psi_1,\psi_2,m}^{(5)}\in H^1\bigl(\Q,T_f^{\vee}\otimes_{\mathcal{O}}{\mathrm{Ind}}^{\Q}_{K}\mathcal{O}_{\psi_1^{-1}}[H_{\fkm}^{(p)}]\otimes_{\mathcal{O}}{\mathrm{Ind}}^{\Q}_{K}\mathcal{O}_{\psi_2^{-1}}[H_{\fkm}^{(p)}](-1)\bigr).\]
The diagram \eqref{ind2} is then replaced with
\begin{equation}\label{ind2-c}
\begin{aligned}
     \xymatrix{
\text{Ind}^{\Q}_{K}\mathcal{O}_{\psi_{1}^{-1}}[H_{\fkm}^{(p)}]\otimes_{\mathcal{O}}\text{Ind}^{\Q}_{K}\mathcal{O}_{\psi_{2}^{-1}}[H_{\fkm}^{(p)}] \ar@{..>}[ddr]_{\xi_{\Delta}^{\cc}} \ar[r]^{1\otimes\upsilon} & \text{Ind}^{\Q}_{K}\mathcal{O}_{\psi_{1}^{-1}}[H_{\fkm}^{(p)}]\otimes_{\mathcal{O}}\text{Ind}^{\Q}_{K}\mathcal{O}_{\psi_{2}^{-\cc}}[H_{\fkmbar}^{(p)}] \ar[d]^(.45){\xi} \\
& \text{Ind}^{\Q}_{K}\mathcal{O}_{\psi_{1}^{-1}\psi_{2}^{-\cc}}[H_{\fkm}^{(p)}\times H_{\fkmbar}^{(p)}] \ar[d]^{\pi_\Delta} \\
& \text{Ind}^{\Q}_{K}\mathcal{O}_{\psi_{1}^{-1}\psi_{2}^{-\cc}}[H[m]^{(p)}], }
\end{aligned}
\end{equation}
where the isomorphism $\upsilon: \text{Ind}^{\Q}_{K}\mathcal{O}_{\psi_{2}^{-1}}[H_{\fkm}^{(p)}] \rightarrow\text{Ind}^{\Q}_{K}\mathcal{O}_{\psi_{2}^{-\cc}}[H_{\fkmbar}^{(p)}]  $ comes from the action of complex conjugation on the inducing representation.

Consequently, Lemma \ref{keydiagram} now turns into
the following commutative diagram:
\begin{equation}\label{keydiagram_complex_conjugate}
   \begin{aligned}
        \xymatrix{
 {\mathrm{Ind}}^{\Q}_{K}\mathcal{O}_{\psi_{1}^{-1}}[H_{\fkm\fkl}^{(p)}]\otimes_{\mathcal{O}}{\mathrm{Ind}}^{\Q}_{K}\mathcal{O}_{\psi_{2}^{-1}}[H_{\fkm\fkl}^{(p)}] \ar[d]^{\mathrm{Norm}^{\fkm\fkl}_{\fkm}\otimes\mathrm{Norm}^{\fkm\fkl}_{\fkm}} \ar[r]^-{\xi_{\Delta}^{\cc}} &{\mathrm{Ind}}^{\Q}_{K}\mathcal{O}_{\psi_{1}^{-1}\psi_{2}^{-\cc}}[H[m\ell]^{(p)}] \ar[d]^{\mathrm{Norm}^{m\ell}_{m}} \\
{\mathrm{Ind}}^{\Q}_{K}\mathcal{O}_{\psi_{1}^{-1}}[H_{\fkm}^{(p)}]\otimes_{\mathcal{O}}{\mathrm{Ind}}^{\Q}_{K}\mathcal{O}_{\psi_{2}^{-1}}[H_{\fkm}^{(p)}] \ar[r]^-{\xi_{\Delta}^{\cc}} & {\mathrm{Ind}}^{\Q}_{K}\mathcal{O}_{\psi_{1}^{-1}\psi_{2}^{-\cc}}[H[m]^{(p)}].}
   \end{aligned}
\end{equation}
The same process as above then leads to the following `conjugate' variant of Theorem~\ref{maintheorem1}:

\begin{cor}\label{maincor1}
With notations as in Theorem~\ref{maintheorem1}, there exists a collection of cohomology classes
\[
{}^\cc{z}_{f,\psi_1,\psi_2,m}\in H^1\bigl(K[m],T_f^{\vee}(\psi_{1}^{-1}\psi_{2}^{-\cc})(-1)\bigr)
\]
such that for every prime $\ell=\fkl\fklbar$ split in $K$ with $(\ell,mpN)=1$ we have the norm relation
\begin{equation*}
    \mathrm{Norm}_{K[m]}^{K[m\ell]}({}^\cc{z}_{f,\psi_1,\psi_2,m\ell})=P_{\fkl}(\mathrm{Frob}_{\fkl})({}^\cc{z}_{f,\psi_1,\psi_2,m}),
\end{equation*}
where $P_{\fkl}(X)=\det(1-X\cdot \mathrm{Frob}_{\fkl}\,|\,T_f(\psi_1\psi_2^\cc)(2))$.
\end{cor}

\begin{proof}
We first follow Proposition \ref{wrongnorm} to have
\begin{align*}
(1\otimes\mathcal{N}_{\fkm}^{\mathfrak{ml}}\otimes&\mathcal{N}_{\fkm}^{\fkm\fkl})({}^\cc\tilde{\kappa}^{(5)}_{f,\psi_1,\psi_2,m\ell})     =(\ell-1)\bigg( a_\ell(f)-\frac{\psi_1(\fkl)\psi_2(\fklbar)}{\ell}([\fkl]\times [\fklbar])\\&-\frac{\psi_1(\fklbar)\psi_2(\fkl)}{\ell}([\fklbar]\times [\fkl])
     +(1-\ell)\frac{\psi_1(\fkl)\psi_2(\fkl)}{\ell^2}([\fkl]\times[\fkl])\bigg)({}^\cc\tilde{\kappa}^{(5)}_{f,\psi_1,\psi_2,m}).
\end{align*} 
Combining the new $\xi_{\Delta}^{\cc}$ (this maps the second factor $[\fkl]$ to $[\fklbar]$, and leads to counterpart ${}^\cc\tilde{\kappa}_{f,\psi_1,\psi_2,m}$ of the classes $\tilde{\kappa}_{f,\psi_1,\psi_2,m}$ of Definition~\ref{def:tame-classes}) and the diagram \eqref{keydiagram_complex_conjugate}, one obtains
\begin{align*}
\Norm_{K[m]}^{K[m\ell]}(&{}^\cc\tilde{\kappa}_{f,\psi_1,\psi_2,m\ell})     =(\ell-1)\bigg( a_\ell(f)-\frac{\psi_1(\fkl)\psi_2^{\cc}(\fkl)}{\ell}([\fkl]\times [\fkl])\\&-\frac{\psi_1(\fklbar)\psi_2^{\cc}(\fklbar)}{\ell}([\fklbar]\times [\fklbar])
     +(1-\ell)\frac{\psi_1(\fkl)\psi_2^{\cc}(\fklbar)}{\ell^2}([\fkl]\times[\fklbar])\bigg)({}^\cc\tilde{\kappa}_{f,\psi_1,\psi_2,m}).
\end{align*} 
This formula should be treated as a replacement for Proposition~\ref{wrongnorm}. 

Then we remove the extra factor $(\ell-1)$ as in $\S\ref{thefix}$ to obtain classes ${}^\cc\kappa_{f,\psi_1,\psi_2,m}$, modify the remaining factor through multiplication by $-\psi_1\psi_2^{\cc}(\fkl)([\fkl]\times [\fkl])$, and apply Lemmas 9.6.1 and 9.6.3 in \cite{Rubin-ES} to obtain the existence of classes ${}^\cc{z}_{f,\psi_1,\psi_2,m}$ with ${}^\cc{z}_{f,\psi_1,\psi_2,1}={}^{\cc}\kappa_{f,\psi_1,\psi_2,1}$ and the desired norm relations.
\end{proof}

\subsection{Construction for general weights and wild norm relations}\label{wildnormsetup}


We 
now extend the constructions of the preceding subsections to $f\in S_{2r}(\Gamma_0(N_f))$ of any weight $2r\geq 2$ and more general Hecke characters, assuming in addition that
\begin{equation}
\textrm{$f$ is ordinary at $\mathfrak{P}$},\tag{ord}
\end{equation}
which we shall often refer to as $f$ being `$p$-ordinary'; and prove that the resulting classes 
also satisfy the wild norm relations, i.e., they are universal norms in the anticyclotomic $\Z_p$-extension of $K$.


\subsubsection{CM Hida families} 
\label{subsec:CM}

We shall replace the weight $2$ theta series $\theta_{\psi_1},\theta_{\psi_2}$ by $p$-adic families, so we begin by recalling the explicit construction of certain CM Hida families, following the exposition in \cite[\S{8.1}]{hsieh-triple}. 

Let $\Gamma_\infty={\mathrm{Gal}}(K_\infty/K)$ be the Galois group of the $\Z_p^2$-extension of $K$, which under our hypotheses can be written as 
\[
\Gamma_\infty\simeq\Gamma_{\pp^\infty}\times\Gamma_{\ppbar^\infty},
\]
with $\Gamma_{\pp^\infty}={\mathrm{Gal}}(K_{\pp^\infty}^{\circ}/K)$ (resp. $\Gamma_{\ppbar^\infty}={\mathrm{Gal}}(K^\circ_{\ppbar^\infty}/K)$) representing the Galois group of the unique $\Z_p$-extension of $K$ inside $K_\infty$ unramified outside $\pp$ (resp. $\ppbar$). Recall that for every ideal $\mathfrak{c}\subset\cO_K$ we denote by $K_\mathfrak{c}$ the ray class field of $K$ of conductor $\mathfrak{c}$ (so in particular $K_{\pp^\infty}^\circ$ is the maximal $\Z_p$-extension of $K$ inside $K_{\pp^\infty}$). For $\mathfrak{q}\in\{\pp,\ppbar\}$, denote by ${\mathrm{Art}}_{\qq}:K_{\qq}^\times\rightarrow G_K^{\mathrm{ab}}$ the restriction of the Artin reciprocity map to $K_{\qq}^\times$, with geometric normalizations. With the identification $\Z_p^\times\simeq\cO_{K_\qq}^\times$, the map ${\mathrm{Art}}_\qq$ induces an isomorphism $1+p\Z_p\xrightarrow{\sim}\Gamma_{\qq^\infty}$ (note that this uses our hypothesis \eqref{eq:p-nmid-h}). Let $\mathbf{u}=1+p$ and let $\gamma_{\mathfrak{q}}\in\Gamma_{\qq^\infty}$ be the topological generator $\gamma_\qq={\mathrm{Art}}_\qq(\mathbf{u})\vert_{K^\circ_{\qq^\infty}}$ with $\cc\gamma_\pp\cc^{-1}=\gamma_{\ppbar}$, where ${\mathrm{Gal}}(K/\Q)=\{1,\cc\}$. 

For each variable $Z$, let $\Psi_Z:\Gamma_\infty\rightarrow\Z_p\dBr{Z}^\times$ be the universal character 
given by
\[
\Psi_Z(\sigma)=(1+Z)^{l(\sigma)},
\]
where $l(\sigma)\in\Z_p$ is such that $\sigma\vert_{K_{\pp^\infty}^\circ}=\gamma_\pp^{l(\sigma)}$. Note that the specialization $\psi_0$ of $\Psi_Z$ to $Z=\mathbf{u}-1$ descends to an isomorphism 
\[
\psi_0=\Psi_{\mathbf{u}-1}:\Gamma_{\pp^\infty}\xrightarrow{\sim}1+p\Z_p
\]
(namely, the inverse of the above isomorphism induced by ${\mathrm{Art}}_\pp$), and may be seen as the $p$-adic avatar of a Hecke character -- still denoted $\psi_0$ -- of $K$ of infinity type $(1,0)$ and conductor  $\pp$.

For $\mathfrak{c}$ coprime to $p$, and for any finite order character $\ch:G_K\rightarrow\cO^\times$ of conductor dividing $\mathfrak{c}$ put
\begin{equation}\label{eq:CM-explicit}
\boldsymbol{\theta}_\ch(Z)(q)=\sum_{(\fa,\pp\mathfrak{c})=1}\ch\psi_0^{-1}(\sigma_{\fa})\Psi^{-1}_{Z}(\sigma_\fa)q^{N_{K/\Q}(\fa)}\in\cO\dBr{Z}\dBr{q},
\end{equation}
where $\sigma_\fa\in{\mathrm{Gal}}(K(\mathfrak{c}\pp^\infty)/K)$ is the Artin symbol of $\fa$. With  conventions as  in $\S\ref{subsubsec:hida}$ below (which differ slightly from those in \cite[\S{3.1}]{hsieh-triple}; our weight map is centered at weight $2$ rather than $0$),  
$\boldsymbol{\theta}_\ch(Z)(q)$ is a Hida family defined over $\cO\dBr{Z}$ of tame level $N_{K/\Q}(\mathfrak{c})D_K$ and tame character $(\ch\circ\mathscr{V})\epsilon_{K}\omega^{-1}$, where $\mathscr{V}:G_\Q^{\mathrm{ab}}\rightarrow G_K^{\mathrm{ab}}$ is the transfer map, $\epsilon_{K}$ is the quadratic character corresponding to $K/\Q$, and $\omega:(\Z/p\Z)^\times\rightarrow\Z_p^\times$ is the Teichm\"uller character. 


\subsubsection{The construction}


Let $\xi_1,\xi_2$ be ray class characters of $K$ of conductor prime-to-$p$ with
\begin{equation}\label{eq:sd-xi}
\chi_{\xi_1}\chi_{\xi_2}=1.\tag{sd}
\end{equation}
Let $\bff$ be the Hida family passing through (the $p$-ordinary $p$-stabilization of) $f$, and let 
\[
(\bfg,\bfh)=(\boldsymbol{\theta}_{\ch_1}(Z_1), \boldsymbol{\theta}_{\xi_2}(Z_2))
\]
be the CM Hida families attached to $\xi_1$ and $\xi_2$. The tame characters of $(\bff,\bfg,\bfh)$ are given by 
\[
(\chi_f,\chi_g,\chi_h)=
(\omega^{2r-2},\chi_{\xi_1}\epsilon_K\omega^{-1},\chi_{\xi_2}\epsilon_K\omega^{-1}),
\]
and so \eqref{eq:sd-xi} implies the self-duality condition in \eqref{eq:a} below. 

Let $\Lambda=\Z_p\dBr{1+p\Z_p}$ and let $\kappa:\Z_p^\times\rightarrow\Lambda^\times$ be a continuous character. Set $\mathrm{T}=\Z_p^\times\times\Z_p$,  $\mathrm{T}'=p\Z_p\times\Z_p^\times$, 
and consider the $\Lambda$-modules
\begin{align*}
\mathcal{A}_{\kappa}&=\left\{ f:\mathrm{T}\rightarrow\Lambda\mid f(1,z)\in \text{Cont}(\Z_p,\Lambda),\, f(a\cdot t)=\kappa(a)\cdot f(t),\, \forall \,a\in\Z_p^\times,\, t\in \mathrm{T}\right\},\\
\mathcal{A}_{\kappa}'&=\left\{ f:\mathrm{T}'\rightarrow\Lambda\mid f(pz,1)\in \text{Cont}(\Z_p,\Lambda), \, f(a\cdot t)=\kappa(a)\cdot f(t) ,\, \forall \, a\in\Z_p^\times,\, t\in \mathrm{T}'\right\}
\end{align*}
equipped with the $\fkm_\Lambda$-adic topology, for $\fkm_\Lambda$ the maximal ideal of $\Lambda$. We also consider
\[
\mathcal{D}_{\kappa}=\mathrm{Hom}_{\text{cont},\Lambda}(\mathcal{A}_{\kappa},\Lambda),\quad\mathcal{D}_{\kappa}'=\mathrm{Hom}_{\text{cont},\Lambda}(\mathcal{A}_{\kappa}',\Lambda)
\]
equipped with the weak-$\ast$ topology. 

For each $i\in\Z/(p-1)\Z$, let $\kappa_i:\Z_p^\times\rightarrow\Lambda^\times$ be the character  
\[
z\mapsto\omega^i(z)[\langle z\rangle],
\]
where $\langle z\rangle=z\omega^{-1}(z)\in 1+p\Z_p$ and $[\cdot]:1+p\Z_p\hookrightarrow\Lambda^\times$ is the inclusion as group-like elements. Put $\mathcal{A}_i^\cdot,\mathcal{D}_i^\cdot$ to denote $\mathcal{A}_{\kappa_i}^\cdot,\mathcal{D}_{\kappa_i}^\cdot$, and note that, by composing with the map $\rho_{k-2}:\Lambda^\times\rightarrow\Z_p^\times$ defined by \[\mathbf{u}\mapsto\mathbf{u}^{k-2},\] the $\Lambda$-adic character $\kappa_i$ interpolates the power maps $z\mapsto z^{k-2}$ on $\Z_p^\times$ for $k-2\equiv i\pmod{p-1}$. Replacing $\Lambda$ by $\Z_p^\times$ and $\kappa_i$ by the character $z\mapsto z^i$ for $z\in\Z_p^\times$ and $i\geq 0$, we define the $\Z_p$-modules $A_i^\cdot,D_i^\cdot$ in the same manner, see also \cite[\S{5.4}]{ACR}. 

To ease notation, set $Y(m,p)=Y(1,Nm(p))$ and denote by $\Gamma(m,p)$ the associated congruence subgroup. As in \cite[Eq.\,(81) \emph{et seq.}]{BSV}, the evaluation $\mathcal{A}_\kappa^\cdot\otimes_\Lambda\mathcal{D}_\kappa^\cdot\rightarrow\Lambda$ gives rise to a $\Lambda$-module homomorphism
\[
\xi_i:H^1(\Gamma(m,p),\mathcal{A}_i)\rightarrow{\mathrm{Hom}}_\Lambda(H^1_c(\Gamma(m,p),\mathcal{D}_i),\Lambda).
\]
Similarly, the determinant map $\det:\mathrm{T}\times\mathrm{T}'\rightarrow\Z_p^\times$ defined by $\det((x_1,x_2),(y_1,y_2))=x_1y_2-x_2y_1$, composed with $\kappa_i:\Z_p^\times\rightarrow\Lambda^\times$ gives rise to 
\[
\zeta_i:\mathrm{Hom}_{\Lambda}(H^1_c(\Gamma(m,p),\mathcal{D}_i),\Lambda)\rightarrow H^1(\Gamma(m,p),\mathcal{D}_i')(-\kappa_i).
\]
Then for any weight $k\geq 2$ with $k-2\equiv i\pmod{p-1}$ we have specialization maps
\[\begin{aligned}
   &\rho_{k-2}:H^1(\Gamma(m,p),\mathcal{A}_i)\rightarrow H^1_{\et}(Y(m,p)_{\overline{\Q}},\mathscr{S}_{k-2}),\\
   & \rho_{k-2}:H^1(\Gamma(m,p),\mathcal{D}_i')\rightarrow H^1_{\et}(Y(m,p)_{\overline{\Q}},\mathscr{L}_{k-2})
\end{aligned}
\]
fitting into the commutative diagram
\[
\xymatrix{
H^1(\Gamma(m,p),\mathcal{A}_i)\ar[d]^{\rho_{k-2}}\ar[rr]^-{\zeta_i\circ\xi_i}&&H^1(\Gamma(m,p),\mathcal{D}_i')(-\kappa_i)\ar[d]^{\rho_{k-2}}\\
H^1_{\et}(Y(m,p)_{\overline{\Q}},\mathscr{S}_{k-2})\ar[rr]^-{\mathtt{s}_{k-2}}&&H^1_{\et}(Y(m,p)_{\overline{\Q}},\mathscr{L}_{k-2})(2-k),
}
\]
where $\mathtt{s}_{k-2}$ is induced by \eqref{srlr}.

Adopting the notations from \cite{BSV} (but working with the above modules of continuous functions $\mathcal{A}_i^\cdot,A_i^\cdot$ and their duals $\mathcal{D}_i^\cdot,D_i^\cdot$, rather than the analogous spaces of locally analytic functions considered in \emph{op.\,cit.}, and letting $\boldsymbol{\mathcal{A}}_i^\cdot,\boldsymbol{A}_i^\cdot$ denote the (big)  \'etale sheaves on $Y(m,p)$ associated with $\mathcal{A}_i^\cdot,A_i^\cdot$ as in \cite[\S{5.3},\S5.6]{ACR}), we let
\begin{equation}\label{eq:wrong-big-1}
\tilde{\boldsymbol{\kappa}}_m^{(1)}\in H^1\bigl(\Q,H^1(\Gamma(m,p),D'_{2r-2})\hat\otimes_\cO H^1(\Gamma(m,p),\mathcal{D}'_{-1})\hat\otimes_\cO H^1(\Gamma(m,p),\mathcal{D}'_{-1})(2-\boldsymbol{\kappa}_{f\bfg\bfh}^*)\bigr)
\end{equation}
be the image of the element 
\[
\textbf{Det}_{Nm,p}^{f\bfg\bfh}\in H^0_{\et}(Y(m,p),\boldsymbol{A}_{2r-2}'\otimes\boldsymbol{\mathcal{A}}_{-1}\otimes\boldsymbol{\mathcal{A}}_{-1}(-\kappa_{f\bfg\bfh}^*))
\]
defined in \cite[\S{8.1}]{BSV} (which have specialized via $\rho_{2r-2}:\boldsymbol{\mathcal{A}}_{2r-2}'\rightarrow\boldsymbol{A}_{2r-2}'$) under the composition
\begin{align*}
&H^0_{\et}\bigl(Y(m,p),\boldsymbol{A}_{2r-2}'\otimes\boldsymbol{\cA}_{-1}\otimes\boldsymbol{\cA}_{-1}(-\kappa_{f\bfg\bfh}^*)\bigr)\\
&\overset{d_*}\longrightarrow H_{\et}^4\bigl(Y(m,p)^3,\boldsymbol{A}'_{2r-2}\boxtimes\boldsymbol{\cA}_{-1}\boxtimes\boldsymbol{\cA}_{-1}(-\boldsymbol{\kappa}_{f\bfg\bfh}^*)\otimes_{}\Z_p(2)\bigr)\\
&\overset{\mathtt{HS}}\longrightarrow H^1(\Q,H^3_{\et}(Y(m,p)_{\overline{\Q}}^3,\boldsymbol{A}_{2r-2}'\boxtimes\boldsymbol{\cA}_{-1}\boxtimes\boldsymbol{\cA}_{-1})(2+\kappa_{f\bfg\bfh}^*))\\
&\overset{\mathtt{K}}\rightarrow H^1\bigl(\Q,H^1(\Gamma(m,p),A_{2r-2}')\hat\otimes_{\Z_p} H^1(\Gamma(m,p),\cA_{-1})\hat\otimes_{\Z_p} H^1(\Gamma(m,p),\cA_{-1})(2+\boldsymbol{\kappa}_{f\bfg\bfh}^*)\bigr)\\
&\xrightarrow{(w_p\otimes 1\otimes 1)_*}
H^1\bigl(\Q,H^1(\Gamma(m,p),A_{2r-2})\hat\otimes_{\Z_p} H^1(\Gamma(m,p),\cA_{-1})\\ &\qquad\qquad\qquad\qquad\qquad\qquad\qquad\qquad\qquad\hat\otimes_{\Z_p} H^1(\Gamma(m,p),\cA_{-1})(2+\boldsymbol{\kappa}_{f\bfg\bfh}^*)\bigr)\\
&\xrightarrow{\mathtt{s}_{\bff\bfg\bfh}}
H^1\bigl(\Q,H^1(\Gamma(m,p),D'_{2r-2})\hat\otimes_{\Z_p} H^1(\Gamma(m,p),\mathcal{D}'_{-1})\\
&\qquad\qquad\qquad\qquad\qquad\qquad \qquad\qquad\hat\otimes_{\Z_p} H^1(\Gamma(m,p),\mathcal{D}'_{-1})(2-\boldsymbol{\kappa}_{f\bfg\bfh}^*)\bigr),
\end{align*}
where $\mathtt{s}_{\bff\bfg\bfh}=\mathtt{s}_{2r-2}\otimes(\zeta_{-1}\circ\xi_{-1})\otimes(\zeta_{-1}\circ\xi_{-1})$,  $\kappa_{f\bfg\bfh}^*:\Z_p^\times\rightarrow\Lambda^\times$ is the square-root of the product of the characters 
\begin{equation}\label{eq:char-kappa}
\kappa_f(z)=z^{2r-2},\quad\quad\kappa_{\bfg}(z)=\kappa_{-1}(z),\quad\quad\kappa_{\bfg}(z)=\kappa_{-1}(z),
\end{equation}
and $\boldsymbol{\kappa}^*_{f\bfg\bfh}:G_\Q\rightarrow\Lambda^\times$ is the composition of $\kappa_{f\bfg\bfh}^*$ with the $p$-adic cyclotomic character $\varepsilon_{\mathrm{cyc}}:G_\Q\rightarrow\Z_p^\times$. 

Let $\tilde{\Gamma}(m,p)=\Gamma(1,N(mp))$. Similarly as in $\S$\ref{thefix}, replacing the second and third copies of $Y(m,p)$ in the above construction by the quotient $Y(m,p)^2/D_m$, where $D_m$ is the group of diamond operators as in \eqref{eq:diamond-Dm} acting diagonally on $Y(m,p)^2$, we obtain the class 
\begin{equation}\label{eq:good-big-1}
\begin{aligned}
    \boldsymbol{\kappa}_m^{(1)}\in H^1\bigl(\Q,H^1(\tilde{\Gamma}(m,p),D'_{2r-2})\hat\otimes_{\Z_p} &H^1(\Gamma(m,p),\mathcal{D}'_{-1})\\ &\hat\otimes_{\Z_p[D_m]} H^1(\Gamma(m,p),\mathcal{D}'_{-1})(2-\boldsymbol{\kappa}_{f\bfg\bfh}^*)\bigr)
\end{aligned}
\end{equation}
determined by the relation $\phi(m)\boldsymbol{\kappa}_m^{(1)}=(\mu_{m*},d_{m*})\tilde{\boldsymbol{\kappa}}_m^{(1)}$, and we put
\begin{equation}\label{eq:good-big-2}
\boldsymbol{\kappa}_m^{(2)}:=(\pi_{m*},1,1)\boldsymbol{\kappa}_m^{(1)}.
\end{equation}


\begin{prop}\label{table2}
For a prime number $\ell$ and a positive integer $m$ with $(m\ell,pN)=1$ we have
\begin{equation*} (\pi_{i*},{\mathrm{pr}}_{j*},{\mathrm{pr}}_{k*})\boldsymbol{\kappa}_{m\ell}^{(1)}= (\bigstar)\boldsymbol{\kappa}_{m}^{(1)},
\end{equation*}
where
\begin{center}
\begin{tabular}{|c|c| } 
 \hline
 $(i,j,k)$ & $\bigstar$  \\ 
 \hline
  $(\ell,1,1)$& $(T_\ell,1,1)$  \\ 
  $(1,\ell,1)$&  $(1,T_\ell,1)$     \\
  $(1,1,\ell)$& $(1,1,T_\ell)$ \\
  $(1,\ell,\ell)$& $\boldsymbol{\kappa}_{f\bfg\bfh}^*(\ell)\boldsymbol{\kappa}_f(\ell)^{-1}(T_\ell',1,1)$\\
  $(\ell,1,\ell)$& $\boldsymbol{\kappa}_{f\bfg\bfh}^*(\ell)\boldsymbol{\kappa}_\bfg(\ell)^{-1}(1,T_\ell',1)$ \\
  $(\ell,\ell,1)$& $\boldsymbol{\kappa}_{f\bfg\bfh}^*(\ell)\boldsymbol{\kappa}_\bfh(\ell)^{-1}(1,1,T_\ell')$ \\
 \hline
\end{tabular}
\end{center}
and $\boldsymbol{\kappa}_f,\boldsymbol{\kappa}_{\bfg},\boldsymbol{\kappa}_{\bfh}:G_\Q\rightarrow\Lambda^\times$ denote the composition of the characters \eqref{eq:char-kappa} with $\varepsilon_{\mathrm{cyc}}$.
If we also have that $(\ell,m)=1$ then
\begin{center}
\begin{tabular}{|c|c| } 
 \hline
 $(i,j,k)$ & $\bigstar$  \\ 
 \hline
  $(1,1,1)$& $(\ell+1)$  \\ 
  $(\ell,\ell,\ell)$&  $(\ell+1)\boldsymbol{\kappa}_{f\bfg\bfh}^*(\ell)$     \\
 
 \hline
\end{tabular}
\end{center}
\end{prop}

\begin{proof}
With $\pi_i$ replaced by ${\mathrm{pr}}_i$ and the classes $\boldsymbol{\kappa}_m^{(1)}$ replaced by $\tilde{\boldsymbol{\kappa}}_m^{(1)}$, the stated relations with an extra factor of $\ell-1$ follow immediately from equations (174) and (176) in \cite{BSV} (adding the prime $\ell$ to the level, rather than $p$). The stated relations for $\boldsymbol{\kappa}_m^{(1)}$ then follow in the same way as in Proposition~\ref{correcttable}.
\end{proof}

Assume that 
\begin{equation}\label{eq:dist}
\xi_i(\pp)\not\equiv\xi_i(\ppbar)\pmod{\mathfrak{P}}, 
\tag{dist}
\end{equation}
for $i=1,2$, so the Galois representation associated to the CM family   $\boldsymbol{\theta}_{\xi_i}(Z_i)$ are $p$-distinguished. Note that the specialization of $\boldsymbol{\theta}_{\xi_i}(Z_i)$ to $Z_i=0$ gives the ordinary $p$-stabilization (with $U_p$-eigenvalue $\xi_i\psi_0^{-1}(\ppbar)$) of the weight $2$ theta series associated to $\xi_i\psi_0^{-1}$, which may be seen as the $p$-adic avatar of a Hecke character of $K$ of infinity type $(-1,0)$. 

Let $\fkf_i\subset\cO_K$ with $(p,\fkf_i)=1$ be a modulus for $\xi_i$, and put $N_{\xi_i}=N_{K/\Q}(\fkf_i)D_K$. Then for every $r\geq 1$ and every integer $m=\fkm\fkmbar$ coprime to $p$ and divisible only by primes split in $K$, we have Hecke algebra homomorphisms
\begin{equation}\label{eq:hecke-wild}
\phi_{\fkf_1\fkm\pp^r}:\mathbb{T}(1,N_{\xi_1}mp^r)'\rightarrow\cO[H_{\fkf_1\fkm\pp^r}],\quad
\phi_{\fkf_2\fkmbar\pp^r}:\mathbb{T}(1,N_{\xi_2}mp^r)'\rightarrow\cO[H_{\fkf_2\fkmbar\pp^r}]
\end{equation}
associated to $\ch_1\psi_0^{-1},\ch_2\psi_0^{-1}$, respectively, and by Theorem~\ref{norm1} these induce  isomorphisms
\begin{equation}\label{eq:iso-r}
\begin{aligned}
\nu_{\fkf_1\fkm\pp^r}:H^1_{\et}(Y_1(N_{\xi_1}mp^r)_{\overline{\Q}},\Z_p(1))\otimes_{\Z_p}\cO[H_{\fkf_1\fkm\pp^r}^{(p)}]&\xrightarrow{\simeq}{\mathrm{Ind}}_{K(\fkf_1\fkm\pp^r)}^\Q\cO(\ch_1^{-1}\psi_0)\\
\nu_{\fkf_2\fkmbar\pp^r}:H^1_{\et}(Y_1(N_{\xi_2}mp^r)_{\overline{\Q}},\Z_p(1))\otimes_{\Z_p}\cO[H_{\fkf_2\fkmbar\pp^r}^{(p)}]&\xrightarrow{\simeq}{\mathrm{Ind}}_{K(\fkf_2\fkmbar\pp^r)}^\Q\cO(\ch_2^{-1}\psi_0)
\end{aligned}
\nonumber
\end{equation}
satisfying the natural compatibility as $r$ varies. On the other hand, as 
explained in \cite[\S{5.6}]{ACR} 
we have $G_\Q$-module isomorphisms
\begin{equation}\label{eq:GS}
H^1(\Gamma(1,N_{\xi_i}m(p)),\mathcal{D}_{j}'(1))\simeq e_{j}\varprojlim_rH^1_{\et}(Y_1(N_{\xi_i}mp^r)_{\overline{\Q}},\Z_p(1)),
\end{equation}
where $e_j=\frac{1}{p-1}\sum_{a\in(\Z/p\Z)^\times}\omega^{-j}(a)[a]$ 
is the projector onto the $\omega^j$-isotypic component of $\bZ_p[\![\Z_p^\times]\!]$. 
Therefore, combining (\ref{eq:GS}) with the inverse limit $\varprojlim_r\nu_{\fkf_i\fkm\pp^r}$ and using the decompositions $H_{\fkf_1\fkm\pp^\infty}^{(p)}\simeq H_{\fkf_1\fkm}^{(p)}\times\Gamma_\pp$ and $H_{\fkf_2\fkmbar\pp^\infty}^{(p)}\simeq H_{\fkf_2\fkmbar}^{(p)}\times\Gamma_\pp$, we obtain the $G_\Q$-equivariant isomorphisms
\begin{equation}\label{eq:nu-infty}
\begin{aligned}
\nu_{\fkf_1\fkm\pp^\infty}:H^1(\Gamma(1,N_{\xi_1}m(p)),\mathcal{D}_{-1}'(1))\hat\otimes_{\Z_p}\cO[\![H_{\fkf_1\fkm\pp^\infty}^{(p)}]\!]\xrightarrow{\simeq}{\mathrm{Ind}}_{K(\fkf_1\fkm)}^\Q\Lambda_\pp(\ch_1^{-1}\psi_0),\\
\nu_{\fkf_2\fkmbar\pp^\infty}:H^1(\Gamma(1,N_{\xi_2}m(p)),\mathcal{D}_{-1}'(1))\hat\otimes_{\Z_p}\cO[\![H_{\fkf_2\fkmbar\pp^\infty}^{(p)}]\!]\xrightarrow{\simeq}{\mathrm{Ind}}_{K(\fkf_2\fkmbar)}^\Q\Lambda_\pp(\ch_2^{-1}\psi_0),
\end{aligned}
\end{equation}
where $\Lambda_\pp=\cO\dBr{\Gamma_{\pp^\infty}}$ with the $G_K$-action given by the tautological character $G_K\twoheadrightarrow\Gamma_{\pp^\infty}\hookrightarrow\Lambda_\pp^\times$.

Continuing with the construction in this section, as in $\S$\ref{subsec:tame}, the maps used to arrive at   $\boldsymbol{\kappa}_m^{(2)}$ in \eqref{eq:good-big-2} are compatible under correspondences. Therefore, after tensoring with $\cO[H_{\fkf_1\fkm\pp^r}^{(p)}]$ and $\cO[H_{\fkf_2\fkmbar\pp^r}^{(p)}]$ via $\phi_{\fkf_1\fkm\pp^r}$ and $\phi_{\fkf_2\fkmbar\pp^r}$, respectively, and letting $r\rightarrow\infty$, the same construction gives rise  to a class
\begin{align*}\label{eq:good-big-3}
\boldsymbol{\kappa}_{m}^{(3)}\in H^1\bigl(\Q,H^1(&\Gamma(1,N(p)),D'_{2r-2})\hat\otimes_{\cO} (H^1(\Gamma(m,p),\mathcal{D}'_{-1})\hat\otimes_{\Z_p}\cO[\![H_{\fkf_1\fkm\pp^\infty}^{(p)}]\!])\\
&\quad\hat\otimes_{\cO[D_m]} (H^1(\Gamma(m,p),\mathcal{D}'_{-1})\hat\otimes_{\Z_p}\cO[\![H_{\fkf_2\fkmbar\pp^\infty}^{(p)}]\!])(2-\boldsymbol{\kappa}_{f\bfg\bfh}^*)\bigr).
\end{align*}

Now let $(\breve{f},\breve{\bfg},\breve{\bfh})$ be a triple of level-$N$ test vectors for $(f,\bfg,\bfh)$. Then we obtain a $G_\Q$-equivariant map
\begin{equation}\label{eq:proj-f}
\varpi_{\breve{f}}:H^1(\Gamma(1,N(p)),D_{2r-2}'(1))[f]\rightarrow T_f^\vee.
\end{equation}
Composing with the isomorphisms \eqref{eq:nu-infty} and the natural projections \[H_{\fkf_1\fkm}^{(p)}\rightarrow H_{\fkm}^{(p)},\qquad H_{\fkf_2\fkmbar}^{(p)}\rightarrow H_{\fkmbar}^{(p)},\] we also obtain the $G_\bQ$-equivariant maps
\begin{equation}\label{eq:proj-gh}
\begin{aligned}
\varpi_{\breve{\bfg}}:H^1(\Gamma(m,p),\mathcal{D}'_{-1}(1))\hat\otimes_{\Z_p}\cO[\![H_{\fkf_1\fkm\pp^\infty}^{(p)}]\!]&\rightarrow{\mathrm{Ind}}_{K}^\Q\cO_{\ch_1^{-1}\psi_0}[H_\fkm^{(p)}][\![\Gamma_\pp]\!],\\
\varpi_{\breve{\bfh}}:H^1(\Gamma(m,p),\mathcal{D}'_{-1}(1))\hat\otimes_{\Z_p}\cO[\![H_{\fkf_2\fkmbar\pp^\infty}^{(p)}]\!]&\rightarrow{\mathrm{Ind}}_{K}^\Q\cO_{\ch_2^{-1}\psi_0}[H_{\fkmbar}^{(p)}][\![\Gamma_\pp]\!].
\end{aligned}
\end{equation}
Taking the image of $\boldsymbol{\kappa}_{m}^{(3)}$ under the natural maps induced by (\ref{eq:proj-f}) and (\ref{eq:proj-gh})  we thus obtain 
\[\begin{aligned}
\boldsymbol{\kappa}_{f,\xi_1,\xi_2,m}^{(4)}\in H^1\bigl(\Q,T_f^\vee\hat\otimes_{\cO}&({\mathrm{Ind}}_K^\Q\cO_{\ch_1^{-1}\psi_0}[H_{\fkm}^{(p)}][\![\Gamma_\pp]\!])\\ &\hat\otimes_{\cO}({\mathrm{Ind}}_K^\Q\cO_{\ch_2^{-1}\psi_0}[H_{\fkmbar}^{(p)}][\![\Gamma_\pp]\!])(-1-\boldsymbol{\kappa}_{f\bfg\bfh}^*)\bigr).
\end{aligned}
\]
using that by \eqref{eq:sd-xi} the above  $\otimes_{\Z_p[D_m]}$ can be replaced by $\hat\otimes_{\Z_p}$.

Next, we note that, with the identifications $\Lambda_\pp\cong\cO\dBr{Z_1}, \Lambda_\pp\cong\cO\dBr{Z_2}$ defined by our choice of topological generator $\gamma_\pp\in\Gamma_\pp$, we have the following equalities 
\[
\cO_{\xi_1^{-1}\psi_0}[\![\Gamma_\pp]\!]=\xi_1^{-1}\psi_0\Psi_{Z_1},\quad\quad\cO_{\xi_2^{-1}\psi_0}[\![\Gamma_\pp]\!]=\xi_2^{-1}\psi_0\Psi_{Z_2},
\]
as $G_K$-representations, where the terms on the right-hand side of these isomorphisms denote the free $\cO\dBr{Z_i}$-module of rank one with $G_K$-action via $\xi_i^{-1}\psi_0\Psi_{Z_i}$; and the  character $-1-\boldsymbol{\kappa}_{f\bfg\bfh}^*$ (written additively as in \cite[\S{8.1}]{BSV}) is the same as $\varepsilon_{\mathrm{cyc}}^{1-r}(\psi_0^{-1}\Psi_{Z_1}^{-1/2}\Psi_{Z_2}^{-1/2}\circ\mathscr{V})$. Thus, we may equivalently write
\begin{equation}\label{eq:big-4}
\begin{aligned}
    \boldsymbol{\kappa}_{f,\xi_1,\xi_2,m}^{(4)}\in H^1\bigl(\Q,T_f^\vee(1-r)\hat\otimes_{\cO}({\mathrm{Ind}}_K^\Q\xi_1^{-1}\psi_0\Psi_{Z_1}[H_{\fkm}^{(p)}])\hat\otimes_{\cO}({\mathrm{Ind}}_K^\Q\xi_2^{-1}\psi_0\Psi_{Z_2}[H_{\fkmbar}^{(p)}])\\ \otimes(\psi_0^{-1}\Psi_{Z_1}^{-1/2}\Psi_{Z_2}^{-1/2}\circ\mathscr{V})\bigr),
\end{aligned}
\end{equation}
which applying the diagonal map $\xi_\Delta$ in \eqref{ind2}  gives rise to the class
\begin{equation}\label{eq:big-5}
\begin{aligned}
\boldsymbol{\kappa}_{f,\xi_1,\xi_2,m}^{(5)}&\in 
H^1\bigl(\Q,T_f^\vee(1-r)\hat\otimes_\cO{\mathrm{Ind}}_{K}^\Q\xi_1^{-1}\xi_2^{-1}\psi_0^{1-\cc}\Psi_{Z_1}^{(1-\cc)/2}\Psi_{Z_2}^{(1-\cc)/2}[H[m]^{(p)}]\bigr),
\end{aligned}
\end{equation}
where $\Psi_{Z_i}^\cc:\Gamma_\infty\rightarrow\cO\dBr{Z_i}^\times$ denotes the character given by
\[
\Psi_{Z_i}^\cc(\sigma)=(1+Z_i)^{l(\sigma)},
\]
with $l(\sigma)\in\Z_p$ determined by   $\sigma\vert_{K^\circ_{\ppbar^\infty}}=\gamma_{\ppbar}^{l(\sigma)}$.

\subsubsection{Anticyclotomic Iwasawa cohomology classes} 
\label{subsubsec:Iw-classes}

The action of the complex conjugation $\cc\in{\mathrm{Gal}}(K/\Q)$ (or more precisely, any lift of $\cc$ to ${\mathrm{Gal}}(K_\infty/K))$  on $\Gamma_\infty$ yields the subgroup decomposition
\[
\Gamma_\infty\simeq \Gamma^+\times\Gamma^-,
\]
with $\Gamma^-$ representing the Galois group of the anticyclotomic $\Z_p$-extension $K_\infty^-/K$. We can then identify $\Gamma_{\pp^\infty}\cong \Gamma^-$  by mapping the topological generator $\gamma_\pp$ to $\gamma_\pp^{1/2}\gamma_{\ppbar}^{-1/2}=:\gamma_-\in\Gamma^-$. Puting $V_1:=(1+Z_1)^{1/2}-1$, we have $\cO\dBr{\Gamma^-}\cong\cO\dBr{V_1}$. The character
\[
\Psi_{V_1}^{1-\cc}=\Psi_{Z_1}^{(1-\cc)/2}:\Gamma_\infty\rightarrow\cO\dBr{V_1}^\times
\]
factors through $\Gamma^-$ and is identified with the tautological character $\Gamma^-\hookrightarrow\cO\dBr{\Gamma^-}^\times$. Thus setting $Z_2=0$ in $\eqref{eq:big-5}$, by Shapiro's lemma the class $\boldsymbol{\kappa}_{f,\xi_1,\xi_2,m}^{(5)}$ finally gives rise to
\begin{equation}\label{eq:big-6}
\begin{aligned}
\boldsymbol{\kappa}_{f,\xi_1,\xi_2,m}^{(6)}&\in
H^1_{\mathrm{Iw}}(K[mp^\infty],T_f^\vee(1-r)\otimes\xi_1^{-1}\xi_2^{-1}\psi_0^{1-\cc}),
\end{aligned}
\end{equation}
where $H^1_{\mathrm{Iw}}(K[mp^\infty],T)$ denotes the limit $\varprojlim_sH^1(K[mp^s],T)$ with respect to corestriction. Thus, we arrive at the following key result.

\begin{thm}\label{maintheorem2}
Let $f\in S_{2r}(\Gamma_0(N_f))$ be a $p$-ordinary newform of weight $2r\geq 2$, let $K$ be an imaginary quadratic field satisfying \eqref{eq:spl} and \eqref{eq:p-nmid-h}, let $\xi_1,\xi_2$ be ray class characters of $K$ satisfying \eqref{eq:sd-xi} and \eqref{eq:dist} with moduli $\fkf_1,\fkf_2\subset\cO_K$, and suppose $(p,N_f\fkf_1\fkf_2)=1$.  Let $\phi:\Gamma^-\rightarrow\cO^\times$ be the $p$-adic avatar of an anticyclotomic Hecke character of $K$ of infinity type $(-j,j)$ with $j\in\Z$, and consider the $G_K$-representations
\[
T_{f,\xi_1\xi_2\phi}=T_f^\vee(1-r)\otimes\xi_1^{-1}\xi_2^{-1}\phi^{-1},\quad\quad T_{f,\xi_1\xi_2^\cc\phi}=T_f^\vee(1-r)\otimes\xi_1^{-1}\xi_2^{-\cc}\phi^{-1}.
\]
Let $m=\fkm\fkmbar$ run over the squarefree integers divisible only by primes split in $K$ and coprime to $pN$, where $N={\mathrm{lcm}}(N_f,N_{\xi_1},N_{\xi_2})$.  
Then there exists collections of Iwasawa cohomology classes
\[
\mathbf{z}_{f,\xi_1,\xi_2,\phi,m}\in H_{\mathrm{Iw}}^1\bigl(K[mp^\infty],T_{f,\xi_1\xi_2\phi}\bigr),\quad{}^\cc\mathbf{z}_{f,\xi_1,\xi_2,\phi,m}\in H_{\mathrm{Iw}}^1\bigl(K[mp^\infty],T_{f,\xi_1\xi_2^\cc\phi}\bigr)
\]
such that for every prime $\ell=\fkl\fklbar$ split in $K$ with $(\ell,mpN)=1$ we have the norm relations
\[
\mathrm{Norm}_{K[m]}^{K[m\ell]}(\mathbf{z}_{f,\xi_1,\xi_2,\phi,m\ell})=P_{\fkl}(\mathrm{Frob}_{\fkl})(\mathbf{z}_{f,\xi_1,\xi_2,\phi,m}),
\]
where $P_{\fkl}(X)=\det(1-X\cdot\mathrm{Frob}_{\fkl}\,|\,T_{f,\xi_1\xi_2\phi}^\vee(1))$, and
\[
\mathrm{Norm}_{K[m]}^{K[m\ell]}({}^\cc\mathbf{z}_{f,\xi_1,\xi_2,\phi,m\ell})=P^\cc_{\fkl}(\mathrm{Frob}_{\fkl})({}^\cc\mathbf{z}_{f,\xi_1,\xi_2,\phi,m}),
\]
where $P^\cc_{\fkl}(X)=\det(1-X\cdot \mathrm{Frob}_{\fkl}\,|\,T_{f,\xi_1\xi_2^\cc\phi}^\vee(1))$.
\end{thm}

\begin{proof}
For $\phi=\psi_0^{\cc-1}$ (which factor through $\Gamma^-$ and corresponds to the $p$-adic avatar of a Hecke character of $K$ of infinity type $(-1,1)$), the construction of the classes $\mathbf{z}_{f,\xi_1,\xi_2,\phi,m}$ satisfying the stated norm relations follows from a direct adaptation of the proof of Theorem~\ref{maintheorem1} applied to the classes $\boldsymbol{\kappa}_{f,\xi_1,\xi_2,m}^{(6)}$ of \eqref{eq:big-6} using Proposition~\ref{table2}; the construction of $\mathbf{z}_{f,\xi_1,\xi_2,\phi,m}$ for general $\phi$ then following by twisting by $\phi^{-1}\psi_0^{\cc-1}$ using \cite[Thm.~6.3.5]{Rubin-ES}.

The construction of the `conjugate' variant classes  ${}^\cc\mathbf{z}_{f,\xi_1,\xi_2,\phi,m}$ follows from an adaptation of the construction described in Section~\ref{variant}. Indeed, replacing the homomorphism $\phi_{\fkf_2\fkmbar\pp^r}$ in \eqref{eq:hecke-wild} by
\[
\phi_{\fkf_2\fkm\pp^r}:\mathbb{T}(1,N_{\xi_2}mp^r)'\rightarrow\cO[H_{\fkf_2\fkm\pp^r}],
\]
we arrive at the $G_\Q$-equivariant isomorphism
\[
\nu_{\fkf_2\fkm\pp^\infty}:H^1(\Gamma(1,N_{\xi_2}m(p)),\mathcal{D}_{-1}'(1))\hat\otimes_{\Z_p}\cO[\![H_{\fkf_2\fkm\pp^\infty}^{(p)}]\!]\xrightarrow{\simeq}{\mathrm{Ind}}_{K(\fkf_2\fkm)}^\Q\Lambda_\pp(\ch_2^{-1}\psi_0),
\]
and in the same manner as above from the classes $\boldsymbol{\kappa}_m^{(2)}$ in \eqref{eq:good-big-2} we obtain classes
\begin{equation}\label{eq:4-c}
\begin{aligned}
{}^\cc\boldsymbol{\kappa}_{f,\xi_1,\xi_2,m}^{(4)}\in H^1\bigl(\Q,T_f^\vee(1-r)\otimes({\mathrm{Ind}}_K^\Q\xi_1^{-1}\psi_0\Psi_{Z_1}[H_{\fkm}^{(p)}])\hat\otimes({\mathrm{Ind}}_K^\Q\xi_2^{-1}\psi_0\Psi_{Z_2}[H_{\fkm}^{(p)}])\\
\otimes(\psi_0^{-1}\Psi_{Z_1}^{-1/2}\Psi_{Z_2}^{-1/2}\circ\mathscr{V})\bigr).
\end{aligned}
\end{equation}
Applying to these the `conjugate' map $\xi_\Delta^\cc$ in \eqref{ind2-c} we obtain
\[
{}^\cc\boldsymbol{\kappa}_{f,\xi_1,\xi_2,m}^{(5)}\in 
H^1\bigl(\Q,T_f^\vee(1-r)\otimes_\cO{\mathrm{Ind}}_{K}^\Q\xi_1^{-1}\xi_2^{-\cc}\Psi_{Z_1}^{(1-\cc)/2}\Psi_{Z_2}^{(\cc-1)/2}[H[m]^{(p)}]\bigr),
\]
which after setting $Z_2=0$ result in classes
\[
{}^{\cc}\boldsymbol{\kappa}_{f,\xi_1,\xi_2,m}\in H^1_{\mathrm{Iw}}(K[mp^\infty],T_f^\vee(1-r)\otimes\xi_1^{-1}\xi_2^{-\cc}).
\]
Applying the argument in the proof of Theorem~\ref{maintheorem1} to these classes yields the construction of ${}^\cc\mathbf{z}_{f,\xi_1,\xi_2,\phi,m}$ for $\phi=1$, and the construction for general $\phi$ then follows again by twisting (by $\phi^{-1}$, in this case).
\end{proof}

\begin{rmk}\label{rem:diag-components}
Denote by $h_2$ the weight $2$ specialization of $\bfh=\boldsymbol{\theta}_{\xi_2}(Z_2)$ obtained by setting $Z_2=0$, put
\begin{equation}\label{eq:dec-components}
\begin{aligned}
\mathbb{V}^\dagger&=T_f^\vee(1-r)\hat\otimes_{\cO}({\mathrm{Ind}}_K^\Q\xi_1^{-1}\psi_0\Psi_{Z_1})\hat\otimes_{\cO}({\mathrm{Ind}}_K^\Q\xi_2^{-1}\psi_0)\otimes(\psi_0^{-1}\Psi_{Z_1}^{-1/2}\circ\mathscr{V})\\
&\cong\bigl(T_f^\vee(1-r)\otimes\xi_1^{-1}\xi_2^{-1}\psi_0^{1-\cc}\Psi_{V_1}^{1-\cc}\bigr)\oplus\bigl(T_f^\vee(1-r)\otimes\xi_1^{-1}\xi_2^{-{\cc}}\Psi_{V_1}^{1-\cc}\bigr)
\end{aligned}
\end{equation}
and let $\kappa(\breve{f},\breve{\bfg},\breve{h}_2)\in H^1(\Q,\mathbb{V}^\dagger)$ be the projection associated to the level-$N$ test vectors $(\breve{f},\breve{\bfh},\breve{h}_2)$ of the corresponding specialization of the $(\bff,\bfg,\bfh)$-isotypic component of the class $\boldsymbol{\kappa}_m^{(1)}$ in \eqref{eq:wrong-big-1}  for $m=1$. Then, writing
\begin{equation}\label{eq:comp-base}
\kappa(\breve{f},\breve{\bfg},\breve{h}_2)=(\kappa_1(\breve{f},\breve{\bfg},\breve{h}_2),\kappa_2(\breve{f},\breve{\bfg},\breve{h}_2))
\end{equation}
according to the decomposition
\begin{align*}
&H^1(\Q,\mathbb{V}^\dagger)\\
&\cong H^1(\Q,T_f^\vee(1-r)\hat\otimes_\cO\xi_1^{-1}\xi_2^{-1}\psi_0^{1-\cc}\Psi_{V_1}^{1-\cc})\oplus H^1(\Q,T_f^\vee(1-r)\hat\otimes_{\cO}{\mathrm{Ind}}_K^\Q\xi_1^{-1}\xi_2^{-\cc}\Psi_{V_1}^{1-\cc})\\
&\cong H^1_{\mathrm{Iw}}(K[p^\infty],T_{f,\xi_1\xi_2\psi_0^{\cc-1}})\oplus H^1_{\mathrm{Iw}}(K[p^\infty],T_{f,\xi_1\xi_2^{\cc}})
\end{align*}
from \eqref{eq:dec-components} and Shapiro's lemma, we see directly from the proof of Theorem~\ref{maintheorem2} that
\begin{equation}\label{eq:comp=base}
(\kappa_1(\breve{f},\breve{\bfg},\breve{h}_2),\kappa_2(\breve{f},\breve{\bfg},\breve{h}_2))=(\mathbf{z}_{f,\xi_1,\xi_2,\psi_0^{\cc-1},1},{}^\cc\mathbf{z}_{f,\xi_1,\xi_2,1,1}).
\end{equation}
\end{rmk}

\section{Anticyclotomic Euler systems}\label{anticycltomicES}

In this section we show that the systems of Iwasawa cohomology classes constructed in Theorem~\ref{maintheorem2}, which form anticyclotomic Euler systems in the sense of Jetchev--Nekov\'a{\v r}--Skinner \cite{JNS}, land in certain Selmer groups defined in the style of Greenberg \cite{Greenberg55}. We then record the bounds on these Selmer groups that follow by applying the machinery of \cite{JNS}  to our construction.


\subsection{Selmer groups}

Let $f\in S_{2r}(\Gamma_0(N_f))$ be a $p$-ordinary newform of weight $2r\geq 2$ with $p\nmid N_f$, and $K$ be an imaginary quadratic field in which $p=\pp\ppbar$ splits. Let $\chi$ be an anticyclotomic Hecke character of $K$ of infinity type $(-j,j)$, and consider the conjugate self-dual $G_K$-representation 
\[
V_{f,\chi}:=V_f^\vee(1-r)\otimes\chi^{-1}.
\]

Given a prime $v\mid p$ of $K$ and a $G_{K_v}$-stable subspace $\mathscr{F}_v^+(V_{f,\chi})\subset V_{f,\chi}$, we put 
\[\mathscr{F}_v^-(V_{f,\chi})=V_{f,\chi}/\mathscr{F}_v^+(V_{f,\chi}).\]

\begin{defn}
Let $L$ be a finite extension of $K$, and fix $\mathscr{F}=\{\mathscr{F}_v^+(V_{f,\chi})\}_{v\vert p}$. The associated \emph{Greenberg Selmer group} $\Sel_{\mathscr{F}}(L,V_{f,\chi})$ is defined by 
\begin{equation*}
    \Sel_{\mathscr{F}}(L,V_{f,\chi}):={{\mathrm{ker}}}\biggl\{H^1(L,V_{f,\chi})\rightarrow\prod_{w}\frac{H^1(L_w,V_{f,\chi})}{H^1_{\mathscr{F}}(L_w,V_{f,\chi})}
    \biggr\},
\end{equation*}
where $w$ runs over the finite primes of $L$, and the local conditions are given by
\[
H^1_{\mathscr{F}}(L_w,V_{f,\chi})=\begin{cases}
{{\mathrm{ker}}}\bigl\{H^1(L_w,V_{f,\chi})\rightarrow H^1(L_{w}^{\ur},V_{f,\chi})\bigr\} & \textrm {if $w\nmid p$},  \\[0.3em]
{{\mathrm{ker}}}\bigl\{H^1(L_w,V_{f,\chi})\rightarrow H^1(L_{w},\mathscr{F}_{v}^-(V_{f,\chi}))\bigr\} & \textrm{if $w\mid v\mid p$}.   
\end{cases}
\]

For any lattice $T_{f,\chi}\subset V_{f,\chi}$, we let $H^1_{\mathscr{F}}(L_w,T_{f,\chi})$ be the inverse image of $H^1_{\mathscr{F}}(L_w,V_{f,\chi})$ under the natural map $H^1(L_w,T_{f,\chi})\rightarrow H^1(L_w,V_{f,\chi})$, and define $\Sel_{\mathscr{F}}(L,T_{f,\chi})$ in the same manner; and given any $\Z_p$-extension $L_\infty=\bigcup_nL_n$ of $L$, we put
\[
\Sel_{\mathscr{F}}(L_\infty,T_{f,\chi}):=\varprojlim_n\Sel_{\mathscr{F}}(L_n,T_{f,\chi}),
\]
with limit with respect to corestriction, and also put 
\[\Sel_{\mathscr{F}}(L_\infty,V_{f,\chi}):=
\Sel_{\mathscr{F}}(L_\infty,T_{f,\chi})\otimes_{\Z_p}\Q_p\] (which is independent of the chosen $T_{f,\chi}$).
\end{defn}

We shall be particularly interested in the following two instances of these definitions:
\begin{itemize}
    \item The \emph{relaxed-strict Selmer group} $\Sel_{\relstr}(L,V_{f,\chi})$ obtained by taking
    \[
    \mathscr{F}_v^+(V_{f,\chi})=\begin{cases}
    V_{f,\chi}&\textrm{if $v=\pp$,}\\[0.3em]
    0&\textrm{if $v=\ppbar$.}
    \end{cases}
    \]
    \item The \emph{ordinary Selmer group} $\Sel_{\ord}(L,V_{f,\chi})$. Since $f$ is $p$-ordinary, upon restriction to $G_{\Q_p}\subset G_{\Q}$ the Galois representation $V_f^\vee$ fits into a short exact sequence
    \[
    0\rightarrow V_f^{\vee,+}\rightarrow V_f^\vee\rightarrow V_f^{\vee,-}\rightarrow 0
    \]
    with $V_{f}^{\vee,\pm}$ one-dimensional, and with the $G_{\Q_p}$-action on $V_{f}^{\vee,-}$ being unramified (see $\S$\ref{subsec:p-ord}). Then $\Sel_{\ord}(L,V_{f,\chi})$ is the Greenberg Selmer group defined by
    \begin{equation}\label{eq:fil+}
    \mathscr{F}_v^+(V_{f,\chi})=V_{f,\chi}^+:=V_f^{\vee,+}(1-r)\otimes\chi^{-1}
    \end{equation}
    for all $v\mid p$. 
\end{itemize}


Following \cite{BK}, we also define the Bloch--Kato Selmer group $\Sel_{\mathrm{BK}}(L,V_{f,\chi})$ by
\[
\Sel_{\mathrm{BK}}(L,V_{f,\chi}):={{\mathrm{ker}}}\biggl\{H^1(L,V_{f,\chi})\rightarrow\prod_w\frac{H^1(L_w,V_{f,\chi})}{H^1_f(L_w,V_{f,\chi})}\biggr\},
\]
where as before $w$ runs over the finite primes of $L$, and the local conditions are given by
\[
H^1_{f}(L_w,V_{f,\chi})=\begin{cases}
{{\mathrm{ker}}}\bigl\{H^1(L_w,V_{f,\chi})\rightarrow H^1(L_{w}^{\ur},V_{f,\chi})\bigr\} & \textrm {if $w\nmid p$},  \\[0.3em]
{{\mathrm{ker}}}\bigl\{H^1(L_w,V_{f,\chi})\rightarrow H^1(L_{w},V_{f,\chi}\otimes\mathbf{B}_{\mathrm{cris}})\bigr\} & \textrm{if $w\mid p$},   
\end{cases}
\]
with $\mathbf{B}_{\mathrm{cris}}$ being Fontaine's crystalline period ring. The local conditions $H^1_f(L_w,T_{f,\chi})\subset H^1(L_w,T_{f,\chi})$ are then defined by propagation. 

For our later convenience, we now recall the well-known relation between these different Selmer groups. Here we shall adopt the convention that the $p$-adic cyclotomic character has Hodge--Tate weight $-1$. Thus, since $\chi$ has infinity type $(-j,j)$ (see $\S$\ref{theta} for our convention regarding infinity types), the $p$-adic avatar of $\chi$ has Hodge--Tate weight $j$ at $\pp$ and $-j$ at $\ppbar$. Since it will suffice for our applications, we suppose $j\geq 0$.

\begin{lem}\label{lem:BK-Gr}
Suppose $f$ has weight $2r\geq 2$. Then for any finite extension $L$ of $K$ we have
\begin{equation}
\Sel_{\mathrm{BK}}(L,V_{f,\chi})=\begin{cases}
\Sel_{\relstr}(L,V_{f,\chi}) &\textrm{if $j\geq r$},\\[0.3em]
\Sel_{\ord}(L,V_{f,\chi}) &\textrm{if $0\leq j<r$}.
\end{cases}\nonumber
\end{equation}
\end{lem}

\begin{proof}
Combining the results of \cite[(3.1)-(3.2)]{nekovarCRM} and \cite[Lem.\,2, p.\,125]{flach-CT}, for every prime $w\mid v\mid p$ of $L/K/\Q$  we have 
\[
H^1_f(L_w,V_{f,\chi})={\mathrm{im}}\bigl\{H^1(L_w,{\mathrm{Fil}}_v^1(V_{f,\chi}))\rightarrow H^1(L_w,V_{f,\chi})\bigr\},
\]
where ${\mathrm{Fil}}_v^1(V_{f,\chi})\subset V_{f,\chi}$  is a $G_{K_v}$-stable subspace (assuming it exists) such that the Hodge--Tate weights of ${\mathrm{Fil}}_v^1(V_{f,\chi})$ (resp. $V_{f,\chi}/{\mathrm{Fil}}_v^1(V_{f,\chi})$) are all $<0$  (resp. $\geq 0$).

Now, the Hodge--Tate weights of $V_{f,\chi}^+$ 
and $V_{f,\chi}^-:=V_{f,\chi}/V_{f,\chi}^{+}$ at the primes of $K$ above $p$ are given by:
\begin{center}
\begin{tabular}{c|c|c|}
\cline{2-3}
  & $V_{f,\chi}^+$ & $V_{f,\chi}^-$ \\
 \hline
\multicolumn{1}{|c|}{HT weight at $\pp$} & $-j-r$ & $-j-1+r$  \\
 \hline
 \multicolumn{1}{|c|}{HT weight at $\ppbar$} & $j-r$ & $j-1+r$  \\
 \hline
\end{tabular}
\end{center}
Hence, ${\mathrm{Fil}}_\pp^1(V_{f,\chi})=V_{f,\chi}$ and ${\mathrm{Fil}}_{\ppbar}^1(V_{f,\chi})=0$ when $j\geq r$, and ${\mathrm{Fil}}_\pp^1(V_{f,\chi})={\mathrm{Fil}}_{\ppbar}^1(V_{f,\chi})=V_{f,\chi}^+$ when $0\leq j<r$, which yield the equalities in the lemma.
\end{proof}

For  
 a choice of Galois stable subspaces $\mathscr{F}=\{\mathscr{F}_v^+(V_{f,\chi})\}_{v\vert p}$ and 
 \[A_{f,\chi}:={\mathrm{Hom}}_{\Z_p}(T_{f,\chi},\mu_{p^\infty}),\]
 we define the associated \emph{dual Selmer group} $\Sel_{\mathscr{F}^*}(L,A_{f,\chi})$ by
\[
\Sel_{\mathscr{F}^*}(L,A_{f,\chi}):={{\mathrm{ker}}}\biggl\{H^1(L,A_{f,\chi})\rightarrow\prod_{w}\frac{H^1(L_w,A_{f,\chi})}{H^1_{\mathscr{F}^*}(L_w,A_{f,\chi})}\biggr\},
\]
where $H^1_{\mathscr{F}^*}(L_w,A_{f,\chi})$ is the orthogonal complement of $H^1_{\mathscr{F}}(L_w,T_{f,\chi})$ under local Tate duality
\[
H^1(L_w,T_{f,\chi})\times H^1(L_w,A_{f,\chi})\rightarrow\Q_p/\Z_p.
\]
In particular, we find that:
\begin{itemize}
\item The dual Selmer group of $\Sel_{\relstr}(L,T_{f,\chi})$ consists of classes that are unramified outside $p$ and have the strict (resp. relaxed) condition at the primes $w\vert\pp$ (resp. $w\vert\ppbar)$; we shall denote this by $\Sel_{\strrel}(L,A_{f,\chi})$.
\item The dual Selmer group of $\Sel_{\ord}(L,T_{f,\chi})$ consists of classes that are unramified outside $p$, and land in the image of the natural map
\[
H^1(L_w,\mathscr{F}_v^+(A_{f,\chi}))\rightarrow H^1(L_w,A_{f,\chi}),\quad\textrm{$\mathscr{F}_v^+(A_{f,\chi}):={\mathrm{Hom}}_{\Z_p}(\mathscr{F}_v^-(T_{f,\chi}),\mu_{p^\infty})$,}
\]
for $w\mid v\mid p$; 
we shall denote this by $\Sel_{\ord}(L,A_{f,\chi})$.
\end{itemize}

\subsection{Local conditions} 

We now determine the Selmer groups where the classes $\mathbf{z}_{f,\xi_1,\xi_2,\phi,m}$ and  ${}^\cc\mathbf{z}_{f,\xi_1,\xi_2,\phi,m}$ of Theorem~\ref{maintheorem2} live.

\begin{thm}
\label{mainthm2-localcond}
For any $\xi_1,\xi_2,\phi$, and $m$ as in Theorem~\ref{maintheorem2}, we have the inclusion
\[
\mathbf{z}_{f,\xi_1,\xi_2,\phi,m}\in\Sel_{\relstr}(K[mp^\infty],T_{f,\xi_1\xi_2\phi}),
\]
\[
{}^{\cc}\mathbf{z}_{f,\xi_1,\xi_2,\phi,m}\in\Sel_{\ord}(K[mp^\infty],T_{f,\xi_1\xi_2^\cc\phi}).
\]
%
\end{thm}

\begin{proof}
With notations as in the proof of Theorem~\ref{maintheorem2}, it suffices to check the result for $\mathbf{z}_{f,\xi_1,\xi_2,\psi_0^{\cc-1},m}$ and ${}^\cc\mathbf{z}_{f,\xi_1,\xi_2,1,m}$; the result for arbitrary $\phi$ then follows by twisting. 

We begin by explaining the case $m=1$. 
Letting $\mathbb{V}^\dagger$ and $\kappa(\breve{f},\breve{\boldsymbol{\bfg}},\breve{h}_2)\in H^1(\Q,\mathbb{V}^\dagger)$ be as in Remark~\ref{rem:diag-components}, by \eqref{eq:comp=base} we need to show the inclusions
\begin{equation}\label{eq:incl-m=1}
\kappa_1(\breve{f},\breve{\boldsymbol{\bfg}},\breve{h}_2)\in\Sel_{\relstr}(K[p^\infty],T_{f,\xi_1\xi_2\psi_0^{\cc-1}}),\quad\quad
\kappa_2(\breve{f},\breve{\boldsymbol{\bfg}},\breve{h}_2)\in\Sel_{\ord}(K[p^\infty],T_{f,\xi_1\xi_2}).
\end{equation}

It follows from \cite[Cor.~8.2]{BSV} that the class  $\kappa(\breve{f},\breve{\boldsymbol{\bfg}},\breve{h}_2)$ lands in the balanced Selmer group $\Sel^{\bal}(\Q,\mathbb{V}^\dagger)$ of Definition~\ref{def:Sel-bu} below. Restricted to $G_K$, the $G_\Q$-representation $\mathbb{V}^\dagger$ decomposes as
\begin{equation}\label{eq:split-V}
\mathbb{V}^\dagger\vert_{G_K}=\bigl(T_f^\vee(1-r)\otimes\ch_1^{-1}\ch_2^{-1}\psi_0^{1-\cc}\Psi_{V_1}^{1-\cc}\bigr)\oplus\bigl(T_f^\vee(1-r)\otimes\ch_1^{-1}\ch_2^{-\cc}\Psi_{V_1}^{1-\cc}\bigr),
\end{equation}
where $V_1=(1+Z_1)^{1/2}-1$, and 
from 
we readily find that the local condition $\mathscr{F}_p^{\bal}(\mathbb{V}^\dagger)$ at $p$ cutting out the balanced Selmer group corresponds to
\begin{equation}\label{eq:bal-pp}
\begin{split}
\mathscr{F}_\pp^{\bal}(\mathbb{V}^\dagger\vert_{G_K})&=\bigl(T_f^\vee(1-r)\otimes\ch_1^{-1}\ch_2^{-1}\psi_0^{1-\cc}\Psi_{V_1}^{1-\cc}\bigr)\\
\mathscr{F}_{\ppbar}^{\bal}(\mathbb{V}^\dagger\vert_{G_K})&=\{0\}
\end{split}
\begin{split}
&\oplus\bigl(T_f^{\vee,+}(1-r)\otimes\ch_1^{-1}\ch_2^{-\cc}\Psi_{V_1}^{1-\cc}\bigr),\\
&\oplus\bigl(T_f^{\vee,+}(1-r)\otimes\ch_1^{-1}\ch_2^{-\cc}\Psi_{V_1}^{1-\cc}\bigr),
\end{split}
\end{equation}
showing that the classes in \eqref{eq:incl-m=1} satisfy the right local conditions at the primes above $p$. 

For the finite  primes $w\nmid p$, we can give an argument that applies to all $m$. Since $V_{f,\xi_1\xi_2\psi_0^{\cc-1}}$ is conjugate self-dual and pure of weight $-1$, we see that 
\[
H^0(K[mp^s]_w,V_{f,\xi_1\xi_2\psi_0^{\cc-1}})=H^2(K[mp^s]_w,V_{f,\xi_1\xi_2\psi_0^{\cc-1}})=0
\]
for all $s\geq 0$, and therefore $H^1(K[mp^s]_w,V_{f,\xi_1\xi_2\psi_0^{\cc-1}})=0$ by Tate's local Euler characteristic formula. This shows that  $H^1(K[mp^s]_w,T_{f,\xi_1\xi_2\psi_0^{\cc-1}})$ is torsion, and as a result the inclusion
\[
{\mathrm{res}}_w(\mathbf{z}_{f,\xi_1,\xi_2,\psi_0^{\cc-1},m})\in \varprojlim_s H^1_f(K[mp^s]_w,T_{f,\xi_1\xi_2\psi_0^{\cc-1}})
\]
follows automatically. A similar argument shows that the classes  ${}^\cc\mathbf{z}_{f,\xi_1,\xi_2,m}$ are unramified at the primes outside $p$, thereby concluding the proof of the inclusions \eqref{eq:incl-m=1} and hence the result for $m=1$. 

It remains to show that for general $m$, the classes in the statement satisfy the claimed local condition at the primes above $p$. Specializing the class $\boldsymbol{\kappa}_{f,\xi_1,\xi_2,m}^{(4)}$ in \eqref{eq:big-4} to $Z_2=0$, it suffices to show the result for $m=1$ with $\xi_1$ replaced by $\xi_1\eta$ and $\xi_2$ replaced by $\xi_2\eta'$ for any two characters  $\eta:H_{\fkm}^{(p)}\rightarrow\mu_{p^\infty}, \eta':H_{\fkmbar}^{(p)}\rightarrow\mu_{p^\infty}$ with inverse characters. We obtain classes $\kappa_{f,\xi_1\eta,\xi_2\eta',m}\in H^1(\Q,\mathbb{V}^\dagger_{\eta,\eta'})$, where
\[
\mathbb{V}^\dagger_{\eta,\eta'}:=T_f^\vee(1-r)\otimes({\mathrm{Ind}}_K^\Q\xi_1^{-1}\eta\psi_0\Psi_{Z_1})\otimes({\mathrm{Ind}}_K^\Q\xi_2^{-1}\eta'\psi_0)\otimes(\psi_0^{-1}\Psi_{Z_1}^{-1/2}\circ\mathscr{V}),
\]
landing in $\Sel^{\bal}(\Q,\mathbb{V}_{\eta,\eta'}^\dagger)$ as a consequence of \cite[Cor.~8.2]{BSV}. Since the map $\xi_\Delta$ in \eqref{ind2} has the effect of projecting onto the first direct summand in the decomposition 
\begin{equation}\label{eq:split-V-eta}
\mathbb{V}_{\eta,\eta'}^\dagger\vert_{G_K}=\bigl(T_f^\vee(1-r)\otimes\ch_1^{-1}\ch_2^{-1}\eta\eta'\psi_0^{1-\cc}\Psi_{V_1}^{1-\cc}\bigr)\oplus\bigl(T_f^\vee(1-r)\otimes\ch_1^{-1}\ch_2^{-\cc}\eta(\eta')^\cc\Psi_{V_1}^{1-\cc}\bigr),
\end{equation}
from the description of $\mathscr{F}_\pp^{\bal}(\mathbb{V}_{\eta,\eta'}^\dagger\vert_{G_K})$ and $\mathscr{F}^{\bal}_{\ppbar}(\mathbb{V}_{\eta,\eta'}^\dagger\vert_{G_K})$ analogous to \eqref{eq:bal-pp}, and letting  $\eta$ and $\eta'$ vary, the inclusions
\[
{\mathrm{res}}_w(\mathbf{z}_{f,\xi_1,\xi_2,\psi_0^{\cc-1},m})\in\begin{cases}
H^1_{\mathrm{Iw}}(K[mp^\infty]_w,T_{f,\xi_1\xi_2\psi_0^{\cc-1}})&\textrm{if $w\mid\pp$,}\\[0.2em]
\{0\}&\textrm{if $w\mid\ppbar$}
\end{cases}
\]
follow, concluding the proof of the first inclusion in the theorem. Finally, the inclusions
\[
{\mathrm{res}}_w({}^\cc\mathbf{z}_{f,\xi_1,\xi_2,m})\in H^1_{\mathrm{Iw}}(K[mp^\infty],\mathscr{F}_w^+(T_{f,\xi_1\xi_2}))
\]
for all $w\mid p$ can be shown in the same manner, now specialising the class ${}^\cc\boldsymbol{\kappa}_{f,\xi_1,\xi_2,m}^{(4)}$ in \eqref{eq:4-c} to $Z_2=0$ and characters $\nu,\nu':H_\fkm^{(p)}\rightarrow\mu_{p^\infty}$ with inverse central characters, and using the fact that the map $\xi_\Delta^\cc$ in \eqref{ind2-c} has the effect of projecting onto the second direct summand in the decomposition  \eqref{eq:split-V-eta}.
\end{proof}

\subsection{Applying the general machinery} 

In this section we give some direct arithmetic applications that follow by applying to the classes of Theorem~\ref{maintheorem2}  the general Euler system machinery of Jetchev--Nekov\'a{\v r}--Skinner  \cite{JNS}. Later in the paper, by exploiting the relation between the bottom class of our Euler systems and special values of complex and $p$-adic $L$-functions, we shall deduce from these results applications to the Bloch--Kato conjecture and the anticyclotomic Iwasawa main conjecture.

For $\xi_1,\xi_2,\phi$ and $m$ as in Theorem~\ref{maintheorem2}, denote by
\begin{align*}
z_{f,\xi_1,\xi_2,\phi,m}&\in\Sel_{\relstr}(K[m],T_{f,\xi_1\xi_2\phi}),\\
{}^\cc z_{f,\xi_1,\xi_2,\phi,m}&\in\Sel_{\ord}(K[m],T_{f,\xi_1\xi_2^\cc\phi}),
\end{align*}
the image of $\mathbf{z}_{f,\xi_1,\xi_2,\phi,m}, {}^\cc\mathbf{z}_{f,\xi_1,\xi_2,\phi,m}$ under the projections
\begin{align*}
\Sel_{\relstr}(K[mp^\infty],T_{f,\xi_1\xi_2\phi})&\rightarrow\Sel_{\relstr}(K[m],T_{f,\xi_1\xi_2\phi}),\\
\Sel_{\ord}(K[mp^\infty],T_{f,\xi_1\xi_2^\cc\phi})&\rightarrow\Sel_{\ord}(K[m],T_{f,\xi_1\xi_2^\cc\phi}),
\end{align*}
respectively, and put (recalling that we assume \eqref{eq:p-nmid-h}, so $K[1]=K$)
\begin{align*}
z_{f,\xi_1,\xi_2,\phi}:=z_{f,\xi_1,\xi_2,\phi,1}&\in\Sel_{\relstr}(K,T_{f,\xi_1\xi_2\phi}),\\
{}^\cc z_{f,\xi_1,\xi_2,\phi}:=
{}^\cc z_{f,\xi_1,\xi_2,\phi,1}&\in\Sel_{\ord}(K,T_{f,\xi_1\xi_2^\cc\phi}).
\end{align*}

\subsubsection{Rank one results}

\begin{thm}\label{thm:rank-1-general}
Let the hypotheses be as in Theorem~\ref{maintheorem2}. Assume also that $f$ is not of CM-type. Then the following hold:
\begin{itemize}
\item[(I)] If $z_{f,\xi_1,\xi_2,\phi}$ is non-torsion, then $\Sel_{\relstr}(K,V_{f,\xi_1\xi_2\phi})$ is one-dimensional.
\item[(II)] If ${}^\cc z_{f,\xi_1,\xi_2,\phi}$ is non-torsion, then $\Sel_{\ord}(K,V_{f,\xi_1\xi_2^\cc\phi})$ is one-dimensional.
\end{itemize}
\end{thm}

\begin{proof}
By Theorem~\ref{maintheorem2} and Theorem~\ref{mainthm2-localcond}, the system of classes 
\begin{equation}\label{eq:ES-rel-str}
\bigl\{z_{f,\xi_1,\xi_2,\phi,m}\in\Sel_{\relstr}(K[m],T_{f,\xi_1\xi_2\phi})\bigr\}_m
\end{equation}
forms an anticyclotomic Euler system in the sense of Jetchev--Nekov\'a{\v r}--Skinner \cite{JNS} for the relaxed-strict Greenberg Selmer group. 
Hence from their general results\footnote{See also \cite[\S{8.1}]{ACR} for an exposition of the relevant results from \cite{JNS}, which at the time of writing is not publicly available yet.} the one-dimensionality of $\Sel_{\relstr}(K,V_{f,\xi_1\xi_2\phi})$ is implied by the nonvanishing of $z_{f,\xi_1,\xi_2,\phi}\in\Sel_{\relstr}(K,V_{f,\xi_1\xi_2\phi})$ provided the $G_K$-representation $V:=V_{f,\xi_1\xi_2\phi}$ satisfies the following hypotheses:
\begin{itemize}
    \item[(i)] $V$ is absolutely irreducible;
    \item[(ii)] There is an element $\sigma\in G_K$ fixing $K(\mu_{p^\infty},(\cO_K^\times)^{1/p^\infty})$ such that $V/(\sigma-1)V$ is one-dimensional;
    \item[(iii)] There is an element $\gamma\in G_K$ fixing $K(\mu_{p^\infty},(\cO_K^\times)^{1/p^\infty})$ such that $V^{\gamma=1}=0$.
\end{itemize}

Since $f$ is not of CM-type, hypotheses (i)--(iii) follow easily from Momose's big image results \cite{Momo} as in \cite[Prop.~7.1.4]{LLZ-K}, whence part (I) of the theorem holds; the proof of part (II) is the same.
\end{proof}




\subsubsection{Iwasawa-theoretic results}


Recall that $K_\infty^-$ denotes the anticyclotomic $\Z_p$-extension of $K$, and put $\Lambda_K^-=\cO\dBr{{\mathrm{Gal}}(K_\infty^-/K)}$. Let
\begin{align*}
\mathbf{z}_{f,\xi_1,\xi_2,\phi}:=\mathbf{z}_{f,\xi_1,\xi_2,\phi,1}&\in\Sel_{\relstr}(K_\infty^-,T_{f,\xi_1\xi_2\phi}),\\
{}^{\cc}\mathbf{z}_{f,\xi_1,\xi_2,\phi}:={}^\cc\mathbf{z}_{f,\xi_1,\xi_2,\phi,1}&\in\Sel_{\ord}(K_\infty^-,T_{f,\xi_1\xi_2^\cc\phi})
\end{align*}
be the bottom classes of the systems $\{\mathbf{z}_{f,\xi_1,\xi_2,\phi,m}\}_m$ and $\{{}^\cc\mathbf{z}_{f,\xi_1,\xi_2,\phi,m}\}_m$ from Theorem~\ref{maintheorem2}, where the inclusions follow from Theorem~\ref{mainthm2-localcond}.

\begin{notation}\label{notation-bigimage}
As in \cite[\S{7.1}]{LLZ-K}, we shall say that $f$ \emph{has big image} if the image of $G_\Q$ in ${\mathrm{Aut}}_{\cO}(T_f^\vee)$ contains a conjugate of $\mathrm{SL}_2(\Z_p)$.
\end{notation}

We also note that, by a theorem of Ribet \cite{Ribet-glasgow}, if $f$ is not of CM-type, then it has big image at $\mathfrak{P}$ all but finitely many primes $\mathfrak{P}$ of $L$.

Put 
\[
X_{\strrel}(K_\infty^-,A_{f,\xi_1\xi_2\phi})={\mathrm{Hom}}_{\Z_p}\biggl(\varinjlim_n\Sel_{\strrel}(K_n^-,A_{f,\xi_1\xi_2\phi}),\Q_p/\Z_p\biggr),
\]
where $K_n^-$ denotes the subextension of $K_\infty^-$ of with $[K_n^-:K]=p^n$, and put
\[
X_{\ord}(K_\infty^-,A_{f,\xi_1\xi_2\phi})={\mathrm{Hom}}_{\Z_p}\biggl(\varinjlim_n\Sel_{\ord}(K_n^-,A_{f,\xi_1\xi_2\phi}),\Q_p/\Z_p\biggr).
\]

The next result can be seen as a divisibility towards an anticyclotomic Iwasawa main conjecture `without $L$-functions'.

\begin{thm}\label{thm:IMC-general}
Let the hypotheses be as in Theorem~\ref{maintheorem2}, and assume in addition that $f$ has big image. Then the following hold:
\begin{itemize}
\item[(I)] If $\mathbf{z}_{f,\xi_1,\xi_2,\phi}$ is non-torsion, then $X_{\strrel}(K_\infty^-,A_{f,\xi_1\xi_2\phi})$ and $\Sel_{\relstr}(K_\infty^-,T_{f,\xi_1\xi_2\phi})$ both have $\Lambda_K^-$-rank one, and we have the divisibility
\[
{\mathrm{char}}_{\Lambda_K^-}(X_{\strrel}(K_\infty^-,A_{f,\xi_1\xi_2\phi})_{\mathrm{tors}})\supset{\mathrm{char}}_{\Lambda_K^-}\biggl(\frac{\Sel_{\relstr}(K_\infty^-,T_{f,\xi_1\xi_2\phi})}{\Lambda_K^-\cdot\mathbf{z}_{f,\xi_1,\xi_2,\phi}}\biggr)^2
\]
in $\Lambda_K^-$. 
\item[(II)] If ${}^{\cc}\mathbf{z}_{f,\xi_1,\xi_2,\phi}$ is non-torsion, then $X_{\ord}(K_\infty^-,A_{f,\xi_1\xi_2^\cc\phi})$ and $\Sel_{\ord}(K_\infty^-,T_{f,\xi_1\xi_2^\cc\phi})$ both have $\Lambda_K^-$-rank one, and we have the divisibility
\[
{\mathrm{char}}_{\Lambda_K^-}(X_{\ord}(K_\infty^-,A_{f,\xi_1\xi_2^\cc\phi})_{\mathrm{tors}})\supset{\mathrm{char}}_{\Lambda_K^-}\biggl(\frac{\Sel_{\ord}(K_\infty^-,T_{f,\xi_1\xi_2^\cc\phi})}{\Lambda_K^-\cdot{}^{\cc}\mathbf{z}_{f,\xi_1,\xi_2,\phi}}\biggr)^2
\]
in $\Lambda_K^-$. 
\end{itemize}
Here, in both $({\mathrm{I}})$ and $({\mathrm{II}})$, the subscript ${\mathrm{tors}}$ denotes the $\Lambda_K^-$-torsion submodule.
\end{thm}

\begin{proof}
By Theorem~\ref{maintheorem2} and Theorem~\ref{mainthm2-localcond}, the system of classes
\begin{equation}\label{eq:ES-Lambda}
\bigl\{\mathbf{z}_{f,\xi_1,\xi_2,\phi,m}\in\Sel_{\relstr}(K[mp^\infty],T_{f,\xi_1\xi_2\phi})\bigr\}_m
\end{equation}
forms a $\Lambda_K^-$-adic anticyclotomic Euler system in the sense of \cite{JNS} for the relaxed-strict Selmer group. Thus by the general results of \emph{op.\,cit.} (see also \cite[\S{8.1}]{ACR} for a summary), the non-torsionness of $\mathbf{z}_{f,\xi_1,\xi_2,\phi}$ implies the conclusions in part (I) provided the $G_K$-module $T=T_{f,\xi_1\xi_2\phi}$ satisfies the following hypotheses:
\begin{itemize}
    \item[(i)] $\bar{T}:=T/\mathfrak{P}T$ is absolutely irreducible;
    \item[(ii)] There is an element $\sigma\in G_K$ fixing $K(\mu_{p^\infty},(\cO_K^\times)^{1/p^\infty})$ such that $T/(\sigma-1)T$ is free of rank $1$ over $\cO$;
    \item[(iii)] There is an element $\gamma\in G_K$ fixing $K(\mu_{p^\infty},(\cO_K^\times)^{1/p^\infty})$ and acting as multiplication by a scalar $a_\gamma\neq 1$ on $\bar{T}$.
\end{itemize}
These are easily checked under our assumption that $f$ has big image  (see \cite[Prop.~7.1.6]{LLZ-K}). This shows part (I) of the theorem, and part (II) follows in the same manner.
\end{proof}

\newpage
\part{Applications}
\label{part:App}

\section{Preliminaries}\label{sec:prelim-2}


In this section, we briefly review the unbalanced triple product $p$-adic $L$-function constructed in \cite{hsieh-triple} and  their associated Selmer groups. 
We also recall from \cite{BSV} the explicit reciprocity law for  diagonal classes. 


\subsection{Triple product \texorpdfstring{$p$}{p}-adic \texorpdfstring{$L$}{L}-function}\label{subsec:triple}

\subsubsection{Hida families} 
\label{subsubsec:hida}

Let $\cR$ be a normal domain, finite flat over 
\[
\Lambda:=\cO\dBr{1+p\Z_p},
\] 
where $\cO$ is the ring of integers of a finite extension of $L_\mathfrak{P}$ of $\Q_p$. (Here, as in $\S\ref{sec:main-thms}$,  $L_\mathfrak{P}$ denotes the completion of a number field $L$ at a prime $\mathfrak{P}$ above $p$ induced by our fixed embedding $i_p:\overline{\bQ}\hookrightarrow\overline{\bQ}_p$.) For an  integer $N>0$ with $p\nmid N$, and a Dirichlet character $\chi:(\Z/Np\Z)^\times\rightarrow\cO^\times$, we denote by $S^o(N,\chi,\cR)\subset\cR\dBr{q}$ the space of ordinary $\cR$-adic cusp forms of tame level $N$ and branch character $\chi$ as defined in \cite[\S{3.1}]{hsieh-triple}. 

Denote by $\mathfrak{X}_\cR^+\subset{\mathrm{Spec}}\,\cR(\overline{\Q}_p)$ the set of \emph{arithmetic points} of $\cR$, consisting of the ring homomorphisms $Q:\cR\rightarrow\overline{\Q}_p$ such that $Q\vert_{1+p\Z_p}$ is given by $z\mapsto z^{k_Q-2}\epsilon_Q(z)$ for some $k_Q\in\Z_{\geq 2}$ called the \emph{weight of $Q$} and $\epsilon_Q(z)\in\mu_{p^\infty}$. (Note that here we center the weight map at weight $2$, rather than weight $0$ as done in \emph{loc.\,cit.}). As in  \cite[\S{3.1}]{hsieh-triple}, we say that $\boldsymbol{f}=\sum_{n=1}^\infty a_n(\bfff)q^n\in S^o(N,\chi,\cR)$ is a \emph{primitive Hida family} if for every $Q\in\mathfrak{X}_\cR^+$ the specialization $\boldsymbol{f}_Q$ gives the $q$-expansion of a $\mathfrak{P}$-ordinary $p$-stabilized newform of weight  $k_Q$ and tame conductor $N$. Attached to such $\bfff$ we let $\mathfrak{X}_{\cR}^{\mathrm{cls}}$ be the set of ring homomorphisms $Q$ as above with $k_Q\in\Z_{\geq 1}$ such that $\bfff_Q$ is the $q$-expansion of a classical modular form (thus $\mathfrak{X}_\cR^{\mathrm{cls}}$ contains $\mathfrak{X}_\cR^+$).

For $\bfff$ a primitive Hida family of tame level $N$, we let 
\begin{equation}\label{eq:big-Gal-rep}
\rho_{\bfff}:G_\Q\rightarrow{\mathrm{Aut}}_{\cR}(V_\bfff)\simeq{\mathrm{GL}}_2(\cR)
\end{equation}
denote the associated Galois representation, which here we take to be the \emph{dual} of that in \cite[\S{3.2}]{hsieh-triple}; in particular, the determinant of $\rho_\bfff$ is $\chi_\cR\cdot\varepsilon_{\mathrm{cyc}}$ in the notations of \emph{loc.\,cit.}, where $\varepsilon_{\mathrm{cyc}}$ is the $p$-adic cyclotomic character. A priori, $\rho_{\bff}$ is just realized over in the fraction field ${\mathrm{Frac}}(\cR)$, but we shall always assume that its associated residual representation $\bar{\rho}_{\bff}:G_\Q\rightarrow{\mathrm{GL}}_2(\kappa_{\cR})$, where $\kappa_{\cR}$ denotes the residue field of $\cR$, is absolutely irreducible, in which case an integral model as in \eqref{eq:big-Gal-rep} can always be found.

Restricting to $G_{\Q_p}$, the Galois representation $V_\bfff$ fits into a short exact sequence
\[
0\rightarrow V_\bfff^+\rightarrow V_\bfff\rightarrow V_\bfff^-\rightarrow 0,
\]
where the quotient $V_\bfff^-$ is free of rank one over $\cR$, with the $G_{\Q_p}$-action given by the unramified character sending an arithmetic Frobenius $\Frob_p^{-1}$ to $a_p(\bfff)$, see \cite[Thm.~2.2.2]{wiles88}. 

Denote by $\bT(N,\cR)$ the Hecke algebra acting on $\oplus_\chi S^o(N,\chi,\cR)$, with $\chi$ running over the Dirichlet characters modulo $Np$. Associated with $\bfff$ there is a $\cR$-algebra homomorphism 
\[
\lambda_{\bfff}:\bT(N,\cR)\rightarrow\cR
\] 
factoring through a local component $\bT_{\fkm}$. Following \cite{hida-AJM}, we define the \emph{congruence ideal} $C(\bfff)$ of $\bfff$ by
\[
C(\bfff):=\lambda_{\bfff}({\mathrm{Ann}}_{\bT_\fkm}({\mathrm{ker}}\,\lambda_\bfff))\subset\cR.
\] 
If $\bar{\rho}_{\bfff}$ is absolutely irreducible and $p$-distinguished, it follows from the results of \cite{Wiles} and \cite{hida-AJM} that $C(\bfff)$ is generated by a nonzero element $\eta_{\bfff}^{cong}\in\cR$.

\subsubsection{Triple products of Hida families}\label{subsubsec:triple-hida}

Let
\[
(\bff,\bfg,\bfh)\in S^o(N_f,\chi_f,\cR_f)\times S^o(N_g,\chi_g,\cR_g)\times S^o(N_h,\chi_h,\cR_h)
\]
be a triple of primitive Hida families with 
\begin{equation}\label{eq:a}
\textrm{$\chi_f\chi_g\chi_h=\omega^{2a}$ for some $a\in\Z$,}\tag{sd-triple}
\end{equation} 
where $\omega$ is the Teichm\"uller character. Put 
\[
\mathcal{R}=\cR_f\hat\otimes_{\cO}\cR_g\hat\otimes_{\cO}\cR_h,
\] 
which is a finite extension of the three-variable Iwasawa algebra $\Lambda\hat\otimes_{\cO}\Lambda\hat\otimes_{\cO}\Lambda$. 

Let $\mathfrak{X}_{\mathcal{R}}^+\subset{\mathrm{Spec}}\,\mathcal{R}(\overline{\Q}_p)$ be the weight space of $\mathcal{R}$ given by
\[
\mathfrak{X}_{\mathcal{R}}^+:=\left\{\underline{Q}=(Q_0,Q_1,Q_2)\in\mathfrak{X}_{\cR_f}^+\times\mathfrak{X}_{\cR_g}^{\mathrm{cls}}\times\mathfrak{X}_{\cR_h}^{\mathrm{cls}}\;:\;k_{Q_0}+k_{Q_1}+k_{Q_2}\equiv 0\pmod{2}\right\}.
\] 
This can be written as the disjoint union $\mathfrak{X}_{\mathcal{R}}^+=\mathfrak{X}_{\mathcal{R}}^{\bal}\sqcup\mathfrak{X}_{\mathcal{R}}^{\bff}\sqcup\mathfrak{X}_{\mathcal{R}}^{\bfg}\sqcup\mathfrak{X}_{\mathcal{R}}^{\bfh}$, where
\begin{align*}
\mathfrak{X}_{\mathcal{R}}^{\bal}&:=\left\{\underline{Q}\in\mathfrak{X}_{\mathcal{R}}^+\;:\;\textrm{$k_{Q_0}+k_{Q_1}+k_{Q_2}> 2k_{Q_i}$ for all $i=0,1,2$}\right\}
\end{align*}
is the set of \emph{balanced} weights, i.e. where each weight $k_{Q_i}$ is smaller than the sum of the other two, and
\begin{align*}
\mathfrak{X}_{\mathcal{R}}^\bff&:=\left\{\underline{Q}\in\mathfrak{X}_{\mathcal{R}}^+\;:\;\textrm{$k_{Q_0}\geq k_{Q_1}+k_{Q_2}$}\right\},\\
\mathfrak{X}_{\mathcal{R}}^\bfg&:=\left\{\underline{Q}\in\mathfrak{X}_{\mathcal{R}}^+\;:\;\textrm{$k_{Q_1}\geq k_{Q_0}+k_{Q_2}$}\right\},\\
\mathfrak{X}_{\mathcal{R}}^\bfh&:=\left\{\underline{Q}\in\mathfrak{X}_{\mathcal{R}}^+\;:\;\textrm{$k_{Q_2}\geq k_{Q_0}+k_{Q_1}$}\right\},
\end{align*}
are the sets of $\bff$-, $\bfg$-, and $\bfh$-\emph{unbalanced} weights, respectively.

Let $\mathbf{V}=V_\bff\hat\otimes_{\cO}V_{\bfg}\hat\otimes_{\cO}V_{\bfh}$ be the triple tensor product Galois representation attached to $(\bff,\bfg,\bfh)$. Write the determinant of $\mathbf{V}$ in the form $\det\mathbf{V}=\mathcal{X}^2\varepsilon_{\mathrm{cyc}}$ (note that this is possible by (\ref{eq:a}) and $p>2$), and put  
\begin{equation}\label{eq:crit-twist}
\Vdag:=\mathbf{V}\otimes\mathcal{X}^{-1},
\end{equation}
which is a self-dual twist of $\mathbf{V}$. 

\subsubsection{Unbalanced triple product $p$-adic $L$-function}

Define the rank four $G_{\Q_p}$-invariant subspace $\mathscr{F}_p^\bff(\Vdag)$ of $\Vdag$ by
\begin{equation}\label{eq:unb-intro}
\mathscr{F}_p^\bff(\Vdag):=V_{\bff}^+\hat\otimes_{\cO}V_\bfg\hat\otimes_{\cO}V_{\bfh}\otimes\mathcal{X}^{-1},
\end{equation}
and for any $\underline{Q}=(Q_0,Q_1,Q_2)\in\mathfrak{X}_{\mathcal{R}}^\bff$  denote by $\mathscr{F}_p^\bff(\VQdag)\subset\VQdag$  the corresponding specializations. 

For a rational prime $\ell$, let $\varepsilon_\ell(\VQdag)$ be the epsilon factor attached to the local representation $\VQdag\vert_{G_{\Q_\ell}}$ 
(see \cite[p.\,21]{tate-background}), and assume that 
\begin{equation}\label{eq:+1}
\textrm{for some $\underline{Q}\in\mathfrak{X}_{\mathcal{R}}^\bff$, we have $\varepsilon_\ell(\VQdag)=+1$ for all prime factors $\ell$ of $N_f N_g N_h$.}\tag{sgn}
\end{equation}
As explained in \cite[\S{1.2}]{hsieh-triple}, it is known that condition (\ref{eq:+1}) is independent of $\underline{Q}$, and it implies that the sign in the functional equation for the triple product $L$-function (with center at $s=0$) 
\[
L(\VQdag,s)
\] 
is $+1$ (resp. $-1$) for all $\underline{Q}\in\mathfrak{X}_{\mathcal{R}}^\bff\cup\mathfrak{X}_{\mathcal{R}}^\bfg\cup\mathfrak{X}_{\mathcal{R}}^\bfh$ (resp. $\underline{Q}\in\mathfrak{X}_{\mathcal{R}}^{\bal}$).


\begin{thm}\label{thm:hsieh-triple}
Let $(\bff,\bfg,\bfh)$ be a triple of primitive Hida families 
satisfying conditions \eqref{eq:a} and \eqref{eq:+1}. Assume in addition that
\begin{itemize}
\item $\gcd(N_f,N_g,N_h)$ is square-free;
\item $\bar{\rho}_{\bff}$ is absolutely irreducible and $p$-distinguished;
\end{itemize}
and fix a generator $\eta_{\bff}^{cong}$ of the congruence ideal of $\bff$. Then there exists a unique element 
\[
\mathscr{L}_p^{\bff}(\bff,\bfg,\bfh)\in\mathcal{R}
\]
such that for all $\underline{Q}=(Q_0,Q_1,Q_2)\in\mathfrak{X}_{\mathcal{R}}^{\bff}$ of weight $(k_0,k_1,k_2)$ with $\epsilon_{Q_0}=1$ we have
\[
\bigl(\mathscr{L}_p^{\bff}(\bff,\bfg,\bfh)(\underline{Q})\bigr)^2=\Gamma_{\VQdag}(0)\cdot\frac{L(\VQdag,0)}{(\sqrt{-1})^{2k_{0}}\cdot\Omega_{\bff_{Q_0}}^2}\cdot\mathcal{E}_p(\mathscr{F}_p^{\bff}(\VQdag))\cdot\prod_{q\in\Sigma_{\mathrm{exc}}}(1+q^{-1})^2,
\]
where:
\begin{itemize}
\item $\Gamma_{\VQdag}(0)=\Gamma_{\bC}(c_{\underline{Q}})\Gamma_\bC(c_{\underline{Q}}+2-k_1-k_2)\Gamma_\bC(c_{\underline{Q}}+1-k_1)\Gamma_\bC(c_{\underline{Q}}+1-k_2)$, with 
\[
c_{\underline{Q}}=(k_0+k_1+k_2-2)/2
\] 
and $\Gamma_\bC(s)=2(2\pi)^{-s}\Gamma(s)$;
\item $\Omega_{\bff_{Q_0}}$ is the canonical period
\[
\Omega_{\bff_{Q_0}}:=(-2\sqrt{-1})^{k_0+1}\cdot\frac{\Vert\bff_{Q_0}^\circ\Vert_{\Gamma_0(N_f)}^2}{\eta_{\bff_{Q_0}}^{cong}}\cdot\Bigl(1-\frac{\chi_{f}'(p)p^{k_0-1}}{\alpha_{Q_0}^2}\Bigr)\Bigl(1-\frac{\chi_{f}'(p)p^{k_0-2}}{\alpha_{Q_0}^2}\Bigr),
\]
with $\bff_{Q_0}^\circ\in S_{k_0}(\Gamma_0(N_{f}))$ the newform of conductor $N_f$ associated with $\bff_{Q_0}$, $\chi_f'$ the prime-to-$p$ part of $\chi_f$, and $\alpha_{Q_0}$ the specialization of $a_p(\bff)\in\cR_f^\times$ at $Q_0$;
\item $\mathcal{E}_p(\mathscr{F}_p^{\bff}(\VQdag))$ is the modified $p$-Euler factor
\[
\mathcal{E}_p(\mathscr{F}_p^{\bff}(\VQdag)):=\frac{L_p(\mathscr{F}_p^{\bff}(\VQdag),0)}{\varepsilon_p(\mathscr{F}_p^{\bff}(\VQdag))\cdot L_p(\VQdag/\mathscr{F}_p^{\bff}(\VQdag),0)}\cdot\frac{1}{L_p(\VQdag,0)},
\]
\end{itemize}
and $\Sigma_{\mathrm{exc}}$ is an explicitly defined subset of the prime factors of $N_f N_g N_h$, \cite[p.~416]{hsieh-triple}.
\end{thm}

\begin{proof}
This is Theorem~A in \cite{hsieh-triple}, which in fact proves a more general interpolation formula. 
\end{proof}


\begin{rmk}\label{rmk:harris-kudla}
The construction of the $p$-adic $L$-function
$\mathscr{L}_p^{\bff}(\bff,\bfg,\bfh)$ is based on Hida's $p$-adic Rankin--Selberg method \cite{hidaII}, and the proof of the above exact interpolation formula relies on a suitable choice of test vectors $(\breve{\bff}^\star,\breve{\bfg}^\star,\breve{\bfh}^\star)$ for  $(\bff,\bfg,\bfh)$ of level $N={\mathrm{lcm}}(N_f,N_g,N_h)$.  
In general (without any additional hypotheses on $\bar{\rho}_{\bfff}$), for any choice $(\breve{\bff},\breve{\bfg},\breve{\bfh})$ of level-$N$ test vectors, Hida's method produces an element 
\begin{equation}\label{eq:general-triple}
\mathscr{L}_p^{\bff}(\breve{\bff},\breve{\bfg},\breve{\bfh})\in{\mathrm{Frac}}(\cR_f)\hat\otimes_\cO\cR_g\hat\otimes_\cO\cR_h,
\end{equation}
and by virtue of the proof of Jacquet's conjecture by  Harris--Kudla \cite{harris-kudla} (see also \cite[Rem.~4.8]{DR1}), for any fixed $\underline{Q}_0\in\mathcal{X}_{\mathcal{R}}^\bff$, if the central $L$-value $L(\mathbf{V}_{\underline{Q}_0}^\dagger,0)$ is nonzero, one can find $(\breve{\bff},\breve{\bfg},\breve{\bfh})$ such that 
\[
\mathscr{L}_p^{\bff}(\breve{\bff},\breve{\bfg},\breve{\bfh})(\underline{Q}_0)\neq 0.
\]
In terms of \eqref{eq:general-triple}, under the hypotheses on $\bar{\rho}_f$ in Theorem~\ref{thm:hsieh-triple}, the $p$-adic $L$-function $\mathscr{L}_p^{\bff}(\bff,\bfg,\bfh)$  is given by  $\eta_\bff^{cong}\cdot\mathscr{L}_p^\bff(\breve{\bff}^\star,\breve{\bfg}^\star,\breve{\bfh}^\star)$.
\end{rmk}

\subsection{Triple product Selmer groups}
\label{subsec:tripleSelmer}

Let $\Vdag=\mathbf{V}\otimes\mathcal{X}^{-1}$ be the self-dual twist of the Galois representation associated to a triple of primitive Hida families $(\bff,\bfg,\bfh)$ 
satisfying (\ref{eq:a}).

\begin{defn}\label{def:local-p}
Put
\[
\mathscr{F}^{\bal}_p(\Vdag)
:=\bigl(V_{\bff}\otimes V_{\bfg}^+\otimes V_{\bfh}^++V_{\bff}^+\otimes V_{\bfg}\otimes V_{\bfh}^++V_{\bff}^+\otimes V_{\bfg}^+\otimes V_{\bfh}\bigr)\otimes\mathcal{X}^{-1},
\]
and define the \emph{balanced local condition} $\rH^1_{\bal}(\Q_p,\Vdag)$ by 
\[
\rH^1_{\bal}(\Q_p,\Vdag):={\mathrm{im}}\bigl(\rH^1(\Q_p,\mathscr{F}_p^{\bal}(\Vdag))\rightarrow\rH^1(\Q_p,\Vdag)\bigr).
\]
As in \eqref{eq:unb-intro}, put $\mathscr{F}_p^{\unb}(\Vdag):=\bigl(V_{\bff}^+\otimes V_{\bfg}\otimes V_{\bfh}\bigr)\otimes\mathcal{X}^{-1}$, and define the \emph{$\bff$-unbalanced local condition} $\rH^1_{\unb}(\Q_p,\Vdag)$ by
\[
\rH^1_{\bff}(\Q_p,\Vdag):={\mathrm{im}}\bigl(\rH^1(\Q_p,\mathscr{F}_p^{\bff}(\Vdag))\rightarrow\rH^1(\Q_p,\Vdag)\bigr).
\]
\end{defn}

It is easy to see that the maps appearing in these definitions are injective, and in the following we shall use this to identify $\rH^1_{\any}(\Q_p,\Vdag)$ with $\rH^1(\Q_p,\mathscr{F}_p^\any(\Vdag))$ for $\any\in\{{\bal},\unb\}$.

\begin{defn}\label{def:Sel-bu}
Let $\any\in\{{\bal},\unb\}$, and define the Selmer group $\Sel^\any(\Q,\Vdag)$ by
\[
\Sel^\any(\Q,\Vdag):={\mathrm{ker}}\biggl\{\rH^1(\Q,\Vdag)\rightarrow\frac{\rH^1(\Q_p,\Vdag)}{\rH^1_\any(\Q_p,\Vdag)}\times\prod_{v\neq p}\rH^1(\Q_v^{\mathrm{nr}},\Vdag)\biggr\}.
\]
We call $\Sel^{\bal}(\Q,\Vdag)$ (resp. $\Sel^{\bff}(\Q,\Vdag)$) the \emph{balanced} (resp. \emph{$\bff$-unbalanced}) Selmer group.
\end{defn}

Let $\Adag={\mathrm{Hom}}_{\Z_p}(\Vdag,\mu_{p^\infty})$ and for $\any\in\{{\bal},\bff\}$ define $\rH^1_{\any^*}(\Q_p,\Adag)\subset\rH^1(\Q_p,\Adag)$ to be the orthogonal complement of $\rH^1_\any(\Q_p,\Vdag)$ under the local Tate duality
\[
\rH^1(\Q_p,\Vdag)\times\rH^1(\Q_p,\Adag)\rightarrow\Q_p/\Z_p.
\]
We then define the balanced and $\bff$-unbalanced Selmer groups with coefficients in $\Adag$ by
\[
\Sel^{\any}(\Q,\Adag):={\mathrm{ker}}\biggl\{\rH^1(\Q,\Adag)\rightarrow\frac{\rH^1(\Q_p,\Adag)}{\rH_{\any^*}^1(\Q_p,\Adag)}\times\prod_{v\neq p}\rH^1(\Q_v^{\mathrm{nr}},\Adag)\biggr\},
\]
and let \[X^{\any}(\Q,\Adag)={\mathrm{Hom}}_{\Z_p}(\Sel^{\any}(\Q,\Adag),\Q_p/\Z_p)\] denote the Pontryagin dual of $\Sel^{\any}(\Q,\Adag)$.

\subsection{Explicit reciprocity law}
\label{subsec:diag}

We continue to denote by $(\bff,\bfg,\bfh)$ a triple of primitive Hida families as in $\S\ref{subsubsec:triple-hida}$ satisfying (\ref{eq:a}), and put $N={\mathrm{lcm}}(N_f,N_g,N_h)$. Let
\begin{equation}\label{eq:3-diag}
\kappa(\bff,\bfg,\bfh)\in\rH^1(\Q,\Vdag(\mathscr{N}))
\end{equation}
be the big diagonal class constructed in \cite[\S{8.1}]{BSV}, where $\Vdag(\mathscr{N})$ denotes a free $\mathcal{R}$-module isomorphic to finitely many copies of $\Vdag$. 

\begin{rmk}
By construction,  $\kappa(\bff,\bfg,\bfh)$ is  the same as the  $(\bff,\bfg,\bfh)$-isotypic projection of the class ${\boldsymbol{\kappa}}_m^{(1)}$ in (\ref{eq:good-big-1}) with $m=1$. 
\end{rmk}

The definition of the Selmer groups in $\S\ref{subsec:tripleSelmer}$ extends immediately to $\Vdag(\mathscr{N})$, and by Corollary~8.2 in \emph{loc.\,cit.} one knows that $\kappa(\bff,\bfg,\bfh)\in\Sel^{\bal}(\Q,\Vdag(\mathscr{N}))$.

Put 
\[
\mathscr{F}_p^{3}(\Vdag)=V_\bff^+\hat\otimes_{\cO} V_{\bfg}^+\hat\otimes_{\cO} V_{\bfh}^+\otimes\mathcal{X}^{-1}\subset\Vdag.
\] 
Then clearly $\mathscr{F}_p^3(\Vdag)\subset\mathscr{F}_p^{\bal}(\Vdag)$, with quotient given by
\begin{equation}\label{eq:gr2}
\mathscr{F}^{\bal}_p(\Vdag)/\mathscr{F}_p^3(\Vdag)\cong
\mathbf{V}_{\bff}^{\bfg\bfh}\oplus\mathbf{V}_{\bfg}^{\bff\bfh}\oplus\mathbf{V}_{\bfh}^{\bff\bfg},
\end{equation}
where 
\begin{equation}\label{eq:3-summands}
\begin{aligned}
\mathbf{V}_{\bff}^{\bfg\bfh}&=V_\bff^-\hat\otimes_{\cO} V_{\bfg}^+\hat\otimes_{\cO} V_{\bfh}^+\otimes\mathcal{X}^{-1},\\
\mathbf{V}_{\bfg}^{\bff\bfh}&=V_\bff^+\hat\otimes_{\cO} V_{\bfg}^-\hat\otimes_{\cO} V_{\bfh}^+\otimes\mathcal{X}^{-1},\\
\mathbf{V}_{\bfh}^{\bff\bfg}&=V_\bff^+\hat\otimes_{\cO} V_{\bfg}^+\hat\otimes_{\cO} V_{\bfh}^-\otimes\mathcal{X}^{-1}.
\end{aligned}
\end{equation}
We similarly define the level-$N$ version $\mathscr{F}^{\bal}_p(\Vdag)(\mathscr{N})$, $\mathbb{V}_{\bff}^{\bfg\bfh}(\mathscr{N})$, etc..

\subsubsection{Three-variable reciprocity law} 

As explained in \cite[\S{7.3}]{BSV} (see also \cite[\S{5.1}]{DR3}), for every choice of level-$N$ test vectors $(\breve{\bff},\breve{\bfg},\breve{\bfh})$ for $(\bff,\bfg,\bfh)$ one can deduce from results 
in \cite{KLZ} the construction of an injective three-variable $p$-adic regulator map with pseudo-null cokernel
\begin{equation}\label{eq:Log}
{\mathrm{Log}}^\bff_{(\breve{\bff},\breve{\bfg},\breve{\bfh})}:\rH^1(\Q_p,\mathbf{V}_{\bff}^{\bfg\bfh}(\mathscr{N}))\rightarrow C(\bff)^{-1}\cR_f\hat\otimes_{\cO}\cR_g\hat\otimes_{\cO}\cR_h,
\end{equation}
where $C(\bff)\subset\cR_f$ is the congruence ideal of $\bff$, characterized by the property that for all $\mathfrak{Z}\in\rH^1(\Q_p,\mathbf{V}_{\bff}^{\bfg\bfh}(\mathscr{N}))$ and all points $\underline{Q}=(Q_0,Q_1,Q_2)\in\mathfrak{X}_{\mathcal{R}}^{+}$ of weight $(k_0,k_1,k_2)$ with $\epsilon_{Q_i}=1$ ($i=0,1,2$) we have
\begin{align*}
{\mathrm{Log}}^\bff_{(\breve{\bff},\breve{\bfg},\breve{\bfh})}&(\mathfrak{Z})(\underline{Q})=(p-1)\alpha_{Q_0}\biggl(1-\frac{\beta_{Q_0} \alpha_{Q_1} \alpha_{Q_2}}{p^{c_{\underline{Q}}}}\biggr)\biggl(1-\frac{\alpha_{Q_0} \beta_{Q_1}\beta_{Q_2}}{p^{c_{\underline{Q}}}}\biggr)^{-1}\\
&\quad\times\begin{cases}
\frac{(-1)^{c_{\underline{Q}}-k_0}}{(c_{\underline{Q}}-k_0)!}\cdot\left\langle{\mathrm{Log}}_p(\mathfrak{Z}_{\underline{Q}}),\eta_{\breve{\bff}_{Q_0}}\otimes\omega_{\breve{\bfg}_{Q_1}}\otimes\omega_{\breve{\bfh}_{Q_2}}\right\rangle_{\mathrm{dR}}, &\textrm{if $\underline{Q}\in\mathfrak{X}_{\mathcal{R}}^{\bal}$,}\\[0.5em]
(k_0-c_{\underline{Q}}-1)!\cdot\left\langle{\mathrm{exp}}_p^*(\mathfrak{Z}_{\underline{Q}}),\eta_{\breve{\bff}_{Q_0}}\otimes\omega_{\breve{\bfg}_{Q_1}}\otimes\omega_{\breve{\bfh}_{Q_2}}\right\rangle_{\mathrm{dR}},&\textrm{if $\underline{Q}\in\mathfrak{X}_{\mathcal{R}}^{\bff}$.}
\end{cases}
\end{align*}

Here, 
\begin{itemize}
\item $c_{\underline{Q}}=(k_0+k_1+k_2-2)/2$ is as in Theorem~\ref{thm:hsieh-triple}; 
\item $\alpha_{Q_0}$ denotes the specialization of $a_p(\bff)$ at $Q_0$ and we put $\beta_{Q_0}=\chi_f'(p)p^{k_0-1}\alpha_{Q_0}^{-1}$; and $(\alpha_{Q_1},\beta_{Q_1})$ (resp. $(\alpha_{Q_2},\beta_{Q_2})$) are defined likewise with  $\bfg$ (resp. $\bfh$) in place of $\bff$;
\item ${\mathrm{Log}}_p$ and ${\mathrm{exp}}^*_p$ are the Bloch--Kato logarithm and dual exponential maps as reviewed in \cite[pp.\,51-52]{BSV}; 
\item $\eta_{\breve{\bff}_{Q_0}}$ (resp. $\omega_{\breve{\bfg}_{Q_1}}, \omega_{\breve{\bfh}_{Q_2}}$) is the differential attached to $\breve{\bff}_{Q_0}$ (resp. $\breve{\bfg}_{Q_1},\breve{\bfh}_{Q_2}$) as in \cite[Eq.~(30)]{BSV} (resp. \cite[Eq.~(34)]{BSV}); and 
\item $\langle-,-\rangle_{\mathrm{dR}}$ denotes the de Rham pairing of \cite[Eq.~(32)]{BSV}.
\end{itemize}

Denote by ${\mathrm{res}}_p(\kappa(\bff,\bfg,\bfh))_{\bff}$ the image of $\kappa(\bff,\bfg,\bfh)$ under natural maps
\begin{equation}\label{eq:map-ERL}
\begin{aligned}
\Sel^{\bal}(\Q,\Vdag(\mathscr{N}))&\xrightarrow{{\mathrm{res}}_p}\rH^1(\Q_p,\mathscr{F}_p^{\bal}(\Vdag(\mathscr{N})))\\ &\rightarrow\rH^1(\Q_p,\mathscr{F}_p^{\bal}(\Vdag(\mathscr{N}))/\mathscr{F}_p^3(\Vdag(\mathscr{N})))
\rightarrow\rH^1(\Q_p,\mathbf{V}_{\bff}^{\bfg\bfh}(\mathscr{N}))
\end{aligned}
\end{equation}
arising from the restriction at $p$ and the projection onto the first direct summand in (\ref{eq:gr2}).

\begin{thm}
\label{thm:ERL}
Let $(\bff,\bfg,\bfh)$ be a triple of primitive Hida families satisfying \eqref{eq:a}. Then for every triple $(\breve{\bff},\breve{\bfg},\breve{\bfh})$ of level-$N$ test vectors  for $(\bff,\bfg,\bfh)$ we have
\[
{\mathrm{Log}}^{\bff}_{(\breve{\bff},\breve{\bfg},\breve{\bfh})}({\mathrm{res}}_p(\kappa(\bff,\bfg,\bfh))_\bff)=\mathscr{L}_p^{\bff}(\breve{\bff},\breve{\bfg},\breve{\bfh}),
\]
where $\mathscr{L}_p^{\bff}(\breve{\bff},\breve{\bfg},\breve{\bfh})$ 
is as in \eqref{eq:general-triple}. 
\end{thm}

\begin{proof}
This is Theorem~A in \cite{BSV} (see also \cite[Thm.\,5.1]{DR3}).
\end{proof}

\begin{rmk}\label{rmk:harris-kudla-diag}
In particular, if $\bar{\rho}_f$ is absolutely irreducible and $p$-distinguished, then Theorem~\ref{thm:ERL} gives
\begin{equation}\label{eq:ERL-f}
\eta_{\bff}^{cong}\cdot{\mathrm{Log}}^{\bff}_{(\breve{\bff}^\star,\breve{\bfg}^\star,\breve{\bfh}^\star)}({\mathrm{res}}_p(\kappa(\bff,\bfg,\bfh))_\bff)=\mathscr{L}_p^{\bff}(\bff,\bfg,\bfh),
\end{equation}
where $\mathscr{L}_p^{\bff}(\bff,\bfg,\bfh)$ is as in Theorem~\ref{thm:hsieh-triple}. For the proof of the arithmetic  applications in this paper, in addition to \eqref{eq:ERL-f}  we shall use its counterpart in the $\bfg$-unbalanced case.

\end{rmk}

\subsubsection{Iwasawa--Greenberg main conjectures}
\label{subsec:IMC}

Let $(\bff,\bfg,\bfh)$ be a triple of primitive Hida families as in Theorem~\ref{thm:hsieh-triple}, and assume that the associated ring $\mathcal{R}$ is  regular. As explained in \cite[\S{7.3}]{ACR}, the following result can be seen as the equivalence between two different formulation of the Iwasawa main conjecture in the style of Greenberg \cite{Greenberg55} for the $p$-adic deformation $\Vdag$. 


For every triple of level-$N$ test vectors $(\breve{\bff},\breve{\bfg},\breve{\bfh})$ for $(\bff,\bfg,\bfh)$, let $\kappa(\breve{\bff},\breve{\bfg},\breve{\bfh})$ denote the image of the class $\kappa({\bff},{\bfg},{\bfh})$ in \eqref{eq:3-diag} under the resulting projection $\rH^1(\Q,\Vdag(\mathscr{N}))\rightarrow\rH^1(\Q,\Vdag)$; thus $\kappa(\breve{\bff},\breve{\bfg},\breve{\bfh})\in{\mathrm{Sel}}_{\bal}(\Q,\Vdag)$.

\begin{prop}\label{prop:equiv}
Suppose $\mathscr{L}_p^{\bff}(\breve{\bff},\breve{\bfg},\breve{\bfh})$ is nonzero. Then following statements ${\mathrm{(I)}}$ and ${\mathrm{(II)}}$ are equivalent:
\begin{enumerate}
\item[(I)] The modules ${\mathrm{Sel}}^{\unb}(\Q,\Vdag)$ and $X^{\unb}(\Q,\Adag)$ are both $\mathcal{R}$-torsion, and 
\[
{\mathrm{char}}_\mathcal{R}\bigl(X^{\unb}(\Q,\Adag)\bigr)=\bigl(\mathscr{L}_p^\unb(\breve{\bff},\breve{\bfg},\breve{\bfh})^2\bigr)
\]
in $\mathcal{R}\otimes_{\Z_p}\Q_p$.

\item[(II)] 
The modules ${\mathrm{Sel}}^{\mathrm{bal}}(\Q,\Vdag)$ and $X^{\mathrm{bal}}(\Q,\Adag)$ have both $\mathcal{R}$-rank one, and
\[
{\mathrm{char}}_{\mathcal{R}}\bigl(X^{\mathrm{bal}}(\Q,\Adag)_{\mathrm{tors}}\bigr)={\mathrm{char}}_{\mathcal{R}}\biggl(\frac{{\mathrm{Sel}}^{\mathrm{bal}}(\Q,\Vdag)}{\mathcal{R}\cdot\kappa(\breve{\bff},\breve{\bfg},\breve{\bfh})}\biggr)^2
\]
in $\mathcal{R}\otimes_{\Z_p}\Q_p$, where the subscript ${\mathrm{tors}}$ denotes the $\mathcal{R}$-torsion submodule.
\end{enumerate}
The same equivalence holds with equalities replaced by the same one-sided divisibilities.
\end{prop}

\begin{proof}
This follows from Theorem~\ref{thm:ERL} and global duality in the same way as the proof of \cite[Thm.~7.15]{ACR}. See \cite{Lai-PhD} for the details in the stated level of generality.
%
\end{proof}

\begin{rmk}
A divisibility towards these Iwasawa main conjectures is obtained in \cite{ACR} in the case where $\bff$ has CM. The Iwasawa-theoretic results in the next two sections also yield a divisibility in these main conjectures, in the case where both $\bfg$ and $\bfh$ have CM by $K$; we leave the formulation of the precise result to the interested reader. A divisibility result is also obtained in \cite{Do-biquadratic} in the case where $\bfg$ and $\bfh$ have CM by different imaginary quadratic fields, and in \cite{GKC-cm} cases of the full conjecture are obtained when $\bff$, $\bfg$, and $\bfh$ have all CM by the same $K$.
\end{rmk}

\section{Definite case}\label{sec:definite}

In this section we deduce our applications to the Bloch--Kato conjecture and the Iwasawa main conjecture for anticyclotomic twists of $f/K$ in the case where $\epsilon(f/K)=+1$.

Throughout this section, we let $f=\sum_{n=1}^\infty a_nq^n\in S_{2r}(\Gamma_0(N_f))$, with $p\nmid N_f$, be a $p$-ordinary newform of weight $2r\geq 2$ defined over $\cO$, and $K$ be an imaginary quadratic field satisfying \eqref{eq:spl} and \eqref{eq:p-nmid-h}.

\subsection{Anticyclotomic \texorpdfstring{$p$}{p}-adic \texorpdfstring{$L$}{L}-functions}\label{subsec:Lp-def}


Recall that $\Gamma^-={\mathrm{Gal}}(K_\infty^-/K)$ denotes the Galois group of the anticyclotomic $\bZ_p$-extension of $K$, and $\gamma_-\in\Gamma^-$ is a topological generator. Write 
\[
N_f=N^+ N^-
\] 
with $N^+$ (resp. $N^-$) divisible only by primes which are split (resp. inert) in $K$, and fix an ideal $\mathfrak{N}^+\subset\cO_K$ with $\cO_K/\mathfrak{N}^+\iso\bZ/N^+\bZ$. 



\begin{thm}\label{thm:BD-theta}
Let $\chi_0$ be an $\cO$-valued ring class character of $K$ of conductor $c\cO_K$, and suppose: 
\begin{itemize}
	\item[(i)] {} $(pN_f,cD_K)=1$,
    \item[(ii)] {} $N^-$ is the squarefree product of an odd number of primes.
\end{itemize}
Then there exists a unique element $\Theta_p^{\BD}(f/K,\chi_0)(W)\in\cO\dBr{W}$ such that for every character $\phi$ of $\Gamma^-$ of infinity type $(-j,j)$ with $0\leq j<r$ and conductor $p^n$, we have
\begin{align*}
    \Theta_p^{\BD}(f/K,\chi_0)^2(\phi(\gamma_-)-1)=\frac{p^{(2r-1)n}}{\alpha_p^{2n}}\cdot&\Gamma(r+j)\Gamma(r-j)\cdot\mathcal{E}_p(f,\chi_0\phi)^{2}\\ \times &\frac{L(f/K,\chi_0\phi,r)}{(2\pi)^{2r}\cdot \Omega_{f,N^-}}\cdot u_K^2\sqrt{D_K}\chi_0\phi(\sigma_{\mathfrak{N}^+})\cdot\eps_p,
\end{align*}
where:
\begin{itemize}
\item $\alpha_p\in\cO^\times$ is the $p$-adic unit root of $x^2-a_px+p^{2r-1}$,
\item 
$\mathcal{E}_p(f,\chi_0\phi)=
\begin{cases}
(1-\alpha_p^{-1}p^{r-1}\chi_0\phi(\pp))(1-\alpha_p p^{r-1}\chi_0\phi(\overline{\pp}))&
\textrm{if $n=0$,}\\[0.2em] 
1&\textrm{if $n>0$},
\end{cases}$

\item $\Omega_{f,N^-}=2^{2r}\cdot\Vert f\Vert_{\Gamma_0(N_f)}^2\cdot\eta_{f,N^-}^{-1}$ is the \emph{Gross period} of $f$ (see \cite[p.\,524]{hsieh-triple}),
\item $u_K=\vert\cO_K^\times\vert/2$, 
\item $\sigma_{\mathfrak{N}^+}\in\Gamma_\infty^-$ is Artin symbol of $\mathfrak{N}^+$, 
\item $\eps_p\in\{\pm 1\}$ is the local root number of $f$ at $p$. 
\end{itemize}
\end{thm}

\begin{proof}
This is Theorem~A in \cite{ChHs1} (as extended in \cite[Thm.~A]{hung} for $c>1$), extending and refining a construction in \cite{BDmumford-tate} in weight $2$. 
\end{proof}



\subsection{Factorization of triple product \texorpdfstring{$p$}{p}-adic \texorpdfstring{$L$}{L}-functions}\label{subsec:factor-L-def}

%

Let $\bff\in S^o(N_f,\omega^{2r-2},\cR)$ be the primitive Hida family specialising to the ordinary $p$-stabilization of $f$ at an arithmetic point $Q_0\in\mathfrak{X}_\cR^+$ of weight $2r$.  Let $\ch_1,\ch_2$ be ray class characters of $K$ of conductors dividing the ideals $\mathfrak{f}_1,\mathfrak{f}_2\subset\cO_K$ coprime to $p$  satisfying \eqref{eq:sd-xi}, and let
\begin{equation}\label{eq:gg*}
(\bfg,\bfh)=(\boldsymbol{\theta}_{\ch_1}(Z_1),\boldsymbol{\theta}_{\ch_2}(Z_2))\in\cO\dBr{Z_1}\dBr{q}\times\cO\dBr{Z_2}\dBr{q}
\end{equation}
be the CM Hida families 
attached to $\ch_1,\ch_2$ as in \eqref{eq:CM-explicit}. Then  $(\bff,\bfg,\bfh)$ satisfies conditions (\ref{eq:a}) and (\ref{eq:+1}). Assume also that  
\begin{equation}\label{eq:irred}
\textrm{$\bar{\rho}_f$ is absolutely irreducible and $p$-distinguished},\tag{irr-dist}
\end{equation}
so the hypotheses in Theorem~\ref{thm:hsieh-triple} are satisfied. The ensuing $\bff$-unbalanced triple product $p$-adic $L$-function $\mathscr{L}_p^{\bff}(\bff,\bfg,\bfh)$ is an element in $\mathcal{R}=\cR\hat\otimes_{\cO}\cO\dBr{Z_1}\hat\otimes_{\cO}\cO\dBr{Z_2}\simeq\cR\dBr{Z_1,Z_2}$, and in the following we let 
\begin{equation}\label{eq:triple-2var}
\mathscr{L}_p^{\unb}(f,\bfg,\bfh)\in\cO\dBr{Z_1,Z_2}
\end{equation}
denote its image  
under the natural map $\cR\dBr{Z_1,Z_2}\rightarrow\cO\dBr{Z_1,Z_2}$ defined by $Q_0$. 
More generally (in particular, without assuming \eqref{eq:irred}), for any choice of level-$N$ test vectors, we let
\begin{equation}\label{eq:genera-triple-2var}
\mathscr{L}_p^\bff(\breve{f},\breve{\bfg},\breve{\bfh})\in L_\mathfrak{P}\otimes_{\cO}\cO\dBr{Z_1,Z_2}
\end{equation}
be the image of \eqref{eq:general-triple} under the map induced by $Q_0$.

\begin{prop}\label{prop:factor-L-def}
Assume that $N^-$ is the square-free product of an odd number of primes. Set
\[
S_i=\mathbf{u}^2(1+Z_i)-1
\]
for $i=1,2$, and
\begin{align*}
    W_1&=\mathbf{u}^{-1}(1+S_1)^{1/2}(1+S_2)^{1/2}-1,\quad\quad
    W_2=(1+S_1)^{1/2}(1+S_2)^{-1/2}-1.
\end{align*}
Then 
\[
\mathscr{L}_p^{\unb}(f,\bfg,\bfh)(S_1,S_2)=
\pm\mathbf{w}^{}\cdot\Theta_p^{\BD}(f/K,\ch_1\ch_2)(W_1)\cdot\Theta_p^{\BD}(f/K,\ch_1\ch_2^\cc)(W_2)\cdot\frac{\eta_{f}^{cong}}{\eta_{f,N^-}},
\]
where $\mathbf{w}$ is a unit in $\cO\dBr{Z_1,Z_2}\otimes_{\Z_p}\Q_p$.
\end{prop}

\begin{proof} 	
This is an immediate extension of Proposition 8.1 in \cite{hsieh-triple}, where the case $\ch_2=\ch_1^{-1}$ is treated (and where the weight map is centered at $0$ rather than $2$, accounting for the change of variables from $Z_i$ to $S_i)$. The need to invert $p$ in the above equality arises from the term $\prod_{q\in\Sigma_{\mathrm{exc}}}(1+q^{-1})^2$ in Theorem~\ref{thm:hsieh-triple}. 
%
%
\end{proof}


\subsection{Selmer group decompositions}
\label{subsec:Selmer-decomposition-def}

Assume further that the characters $\xi_1$ and $\xi_2$ satisfy \eqref{eq:dist}, so the associated big Galois representations $V_{\bfg}$ and $V_{\bfh}$ are such that
\begin{equation}\label{eq:CM-ind}
V_{\bfg}\cong{\mathrm{Ind}}_K^\Q(\xi_1^{-1}\psi_0\Psi_{Z_1}),\quad\quad
V_{\bfh}\cong{\mathrm{Ind}}_K^\Q(\xi_2^{-1}\psi_0\Psi_{Z_2}),
\end{equation}
where $\Psi_{Z_i}$ and $\psi_0=\Psi_{\mathbf{u}-1}$ are as in $\S\ref{subsec:CM}$.

Recall from $\S\ref{subsec:Galreps}$ that the $p$-adic Galois representation $V_f^\vee$ associated to $f$ satisfies $\det(V_f^\vee)=\varepsilon_{\mathrm{cyc}}^{2r-1}$. On the other hand,  $\det(V_{\bfg}\otimes V_{\bfh})=\psi_0^2\Psi_{Z_1}\Psi_{Z_2}\circ\mathscr{V}$. Thus, writing $\Vsdag$ for the specialization of $\mathbf{V}^\dagger$ to $Q_0$ we find
\begin{equation}\label{eq:dec-V-first}
\begin{aligned}
\Vsdag&\simeq T_f^{\vee}\otimes({\mathrm{Ind}}_K^\Q\ch_1^{-1}\psi_0\Psi_{Z_1})\otimes({\mathrm{Ind}}_K^\Q\ch_2^{-1}\psi_0\Psi_{Z_2})\otimes\varepsilon_{\mathrm{cyc}}^{1-r}(\psi_0^{-1}\Psi_{Z_1}^{-1/2}\Psi_{Z_2}^{-1/2}\circ\mathscr{V})\\
&\simeq\bigl(T_f^{\vee}(1-r)\otimes{\mathrm{Ind}}_{K}^\Q(\ch_1^{-1}\ch_2^{-1}\psi_0^{1-\cc}\Psi_{V_1}^{1-\cc})\bigr)\oplus\bigl(T_f^{\vee}(1-r)\otimes{\mathrm{Ind}}_K^\Q(\ch_1^{-1}\ch_2^{-\cc}\Psi_{W_2}^{1-\cc})\bigr),
\end{aligned}
\end{equation}
where 
we put
\begin{equation}\label{eq:V-W}
V_1=(1+Z_1)^{1/2}(1+Z_2)^{1/2}-1=\mathbf{u}^{-1}(1+W_1)-1,
\end{equation}
(previously, in \S \ref{subsubsec:Iw-classes}, we let $Z_2=0$ so $V_1=(1+Z_1)^{1/2}-1$) and $W_1,W_2$ are as in Proposition~\ref{prop:factor-L-def}. 
In particular, since \eqref{eq:V-W} gives
\begin{equation}\label{eq:V-W-psi}
\Psi_{W_1}^{1-\cc}=\Psi_{\mathbf{u}-1}^{1-\cc}\Psi_{V_1}^{1-\cc}=\psi_0^{1-\cc}\Psi_{V_1}^{1-\cc}
\end{equation}
(see $\S\ref{subsec:CM}$ for the second equality), we get
\begin{equation}\label{eq:shapiro}
\rH^1(\Q,\Vsdag)\simeq\rH^1(K,T_f^{\vee}(1-r)\otimes\ch_1^{-1}\ch_2^{-1}\Psi_{W_1}^{1-\cc})\oplus\rH^1(K,T_f^{\vee}(1-r)\otimes\ch_1^{-1}\ch_2^{-\cc}\Psi_{W_2}^{1-\cc})
\end{equation}
by Shapiro's lemma.

\begin{prop}\label{prop:factor-S}
Under (\ref{eq:shapiro}), the balanced Selmer group $\Sel^{\bal}(\Q,\Vsdag)$ decomposes as
\begin{align*}
\Sel^{\bal}(\Q,\Vsdag)\simeq\Sel_{\relstr}(K,&T_f^{\vee}(1-r)\otimes\ch_1^{-1}\ch_2^{-1}\Psi_{W_1}^{1-\cc})\\ &\oplus\Sel_{\ord}(K,T_f^{\vee}(1-r)\otimes\ch_1^{-1}\ch_2^{-\cc}\Psi_{W_2}^{1-\cc}),
\end{align*}
and the $\unb$-unbalanced Selmer group $\Sel^{\unb}(\Q,\Vsdag)$ decomposes as
\begin{align*}
\Sel^{\unb}(\Q,\Vsdag)\simeq
\Sel_{\ord}(K,&T_f^{\vee}(1-r)\otimes\ch_1^{-1}\ch_2^{-1}\Psi_{W_1}^{1-\cc})\\ &\oplus\Sel_{\ord}(K,T_f^{\vee}(1-r)\otimes\ch_1^{-1}\ch_2^{-\cc}\Psi_{W_2}^{1-\cc}).\end{align*}
\end{prop}

\begin{proof}
From (\ref{eq:dec-V-first}) we see that the balanced local condition is given by
\begin{equation}\label{eq:bal-def}
\begin{aligned}
\mathscr{F}^{\bal}_p(\Vsdag)
\simeq\bigl(T_f^{\vee}(1-r)\otimes\ch_1^{-1}\ch_2^{-1}\Psi_{W_1}^{1-\cc}\bigr)&\oplus\bigl(T_f^{\vee,+}(1-r)\otimes\ch_1^{-1}\ch_2^{-\cc}\Psi_{W_2}^{1-\cc}\bigr)\\
&\oplus\bigl(T_f^{\vee,+}(1-r)\otimes\ch_1^{-\cc}\ch_2^{-1}\Psi_{W_2}^{\cc-1}\bigr).
\end{aligned}
\end{equation}

Put \[\widetilde{\mathbf{V}}_{Q_0}^\dagger=\bigl(T_f^{\vee}(1-r)\otimes\ch_1^{-1}\ch_2^{-1}\Psi_{W_1}^{1-\cc}\bigr)\oplus\bigl(T_f^{\vee}(1-r)\otimes\ch_1^{-1}\ch_2^{-\cc}\Psi_{W_2}^{1-\cc}\bigr),\]
so by (\ref{eq:dec-V-first}) we have
\begin{equation}\label{eq:shapiro-bis}
\rH^1(\Q,\mathbf{V}_{Q_0}^\dagger)\simeq\rH^1(K,\widetilde{\mathbf{V}}_{Q_0}^\dagger).
\end{equation}
Then from (\ref{eq:bal-def}) we obtain
\begin{equation}\label{eq:bal-shapiro}
\begin{split}
\mathscr{F}^{\bal}_\pp(\widetilde{\mathbf{V}}_{Q_0}^\dagger)&\simeq\bigl(T_f^{\vee}(1-r)\otimes\ch_1^{-1}\ch_2^{-1}\Psi_{W_1}^{1-\cc}\bigr)\\
\mathscr{F}^{\bal}_{\ppbar}(\widetilde{\mathbf{V}}_{Q_0}^\dagger)&\simeq\{0\}
\end{split}
\begin{split}
&\oplus\bigl(T_f^{\vee,+}(1-r)\otimes\ch_1^{-1}\ch_2^{-\cc}\Psi_{W_2}^{1-\cc}\bigr),\\
&\oplus \bigl( T_f^{\vee,+}(1-r)\otimes\ch_1^{-1}\ch_2^{-\cc}\Psi_{W_2}^{1-\cc}\bigr),\nonumber
\end{split}
\end{equation}
%
which yields the claimed description of $\Sel^{\bal}(\Q,\Vsdag)$. 

Similarly, we find that the $\unb$-unbalanced local condition is given by
\begin{equation}
\begin{aligned}
\mathscr{F}^{\unb}_\pp(\widetilde{\mathbf{V}}_{Q_0}^\dagger)&\simeq\bigl(T_f^{\vee,+}(1-r)\otimes\ch_1^{-1}\ch_2^{-1}\Psi_{W_1}^{1-\cc}\bigr)\oplus\bigl(T_f^{\vee,+}(1-r)\otimes\ch_1^{-1}\ch_2^{-\cc}\Psi_{W_2}^{1-\cc}\bigr),\\
\mathscr{F}^{\unb}_{\ppbar}(\widetilde{\mathbf{V}}_{Q_0}^\dagger)&\simeq\bigl(T_f^{\vee,+}(1-r)\otimes\ch_1^{-1}\ch_2^{-1}\Psi_{W_1}^{1-\cc}\bigr)\oplus\bigl(T_f^{\vee,+}(1-r)\otimes\ch_1^{-1}\ch_2^{-\cc}\Psi_{W_2}^{1-\cc}\bigr),\nonumber
\end{aligned}
\end{equation}
from where the claimed description of $\Sel^\unb(\Q,\Vsdag)$ follows.
\end{proof}

As a consequence, we also obtain the following decomposition for the Selmer groups with coefficients in $\mathbf{A}_{Q_0}^\dagger={\mathrm{Hom}}_{\Z_p}(\Vsdag,\mu_{p^\infty})$, mirroring in the case of $\Sel^{\unb}(\Q,\mathbf{A}_{Q_0}^\dagger)$ the factorization of $p$-adic $L$-functions in Proposition~\ref{prop:factor-L-def}.

\begin{cor}\label{cor:factor-S}
The balanced Selmer group $\Sel^{\bal}(\Q,\mathbf{A}_{Q_0}^\dagger)$ decomposes as
\begin{align*}
\Sel^{\bal}(\Q,\mathbf{A}_{Q_0}^\dagger)&\simeq\Sel_{\strrel}(K,A_f(r)\otimes\ch_1\ch_2\Psi_{W_1}^{\cc-1})\oplus\Sel_{\ord}(K,A_f(r)\otimes\ch_1\ch_2^\cc\Psi_{W_2}^{\cc-1}),
\end{align*}
where $A_f(r)={\mathrm{Hom}}_{\Z_p}(T_f^\vee(1-r),\mu_{p^\infty})$.

The $\unb$-unbalanced Selmer group $\Sel^{\unb}(\Q,\mathbf{A}_{Q_0}^\dagger)$ decomposes as
\begin{align*}
\Sel^{\unb}(\Q,\mathbf{A}_{Q_0}^\dagger)&\simeq
\Sel_{\ord}(K,A_f(r)\otimes\ch_1\ch_2\Psi_{W_1}^{\cc-1})\oplus\Sel_{\ord}(K,A_f(r)\otimes\ch_1\ch_2^\cc\Psi_{W_2}^{\cc-1}).
\end{align*}
\end{cor}

\begin{proof}
This is immediate from Proposition~\ref{prop:factor-S} and local Tate duality.
\end{proof}

\subsection{Explicit reciprocity law}
\label{subsec:ERL-def}

Now, we put
\[
\mathbb{V}^\dagger=\Vsdag\otimes_{\cO\dBr{Z_1,Z_2}}\cO\dBr{Z_1,Z_2}/(Z_2), 
\]
where we let $h_2$ be the weight $2$ CM form obtained by specialising $\bfh=\boldsymbol{\theta}_{\xi_2}(Z_2)$ to $Z_2=0$, and let 
\begin{equation}\label{eq:diag1}
\kappa(f,\bfg,h_2)\in\rH^1(\Q,\mathbb{V}^\dagger(\mathscr{N}))
\end{equation}
be the resulting $1$-variable specialization of the big diagonal class $\kappa(\bff,\bfg,\bfh)$ in (\ref{eq:3-diag}). 
%
Likewise, we denote by $\mathscr{L}_p^{\bff}(f,\bfg,h_2)$ the image of (\ref{eq:triple-2var}) in $\cO\dBr{Z_1,Z_2}/(Z_2)\simeq\cO\dBr{Z_1}$, and similarly for any choice of level-$N$ test vectors we let 
\[
\mathscr{L}_p^{\bff}(\breve{f},\breve{\bfg},\breve{h}_2)\in L_{\mathfrak{P}}\otimes_{\cO}\cO\dBr{Z_1}
\]
be the natural image of \eqref{eq:genera-triple-2var}. 
Since we have the inclusion $\kappa(f,\bfg,h_2)\in\Sel^{\bal}(\Q,\VVdag(\mathscr{N}))$ as a consequence of \cite[Cor.~8.2]{BSV}, we can write
\begin{equation}\label{eq:dec-kato1}
\kappa(f,\bfg,h_2)=(\kappa_1(f,\bfg,h_2),\kappa_2(f,\bfg,h_2))
\end{equation}
according to the decomposition from Proposition~\ref{prop:factor-S};  in particular, we have
\begin{equation}\label{eq:kap-1}
\kappa_1(f,\bfg,h_2)\in\Sel_{\relstr}(K,T_f^{\vee}(1-r)\otimes\ch_1^{-1}\ch_2^{-1}\Psi_{W_1}^{1-\cc}(\mathscr{N})).
\end{equation}

Let  $\mathfrak{X}_{\cO\dBr{W_1}}^+$ be the set of ring homomorphisms $Q\in{\mathrm{Spec}}(\cO\dBr{W_1})(\overline{\Q}_p)$ with $Q(1+W_1)=\zeta_Q\mathbf{u}^{j_Q}$ for some $\zeta_Q\in\mu_{p^\infty}$ and $j_Q\in\Z_{\geq 0}$, and for any $\cO\dBr{W_1}$-module $M$ we let $M_Q$ denote the corresponding  specialization. Write 
\[
T_f^{\vee,-}:=T_f^\vee/T_f^{\vee,+},
\]
%
and denote by $p_f^-:T_f^{\vee}(1-r)\rightarrow T_f^{\vee,-}(1-r)$ the natural projection. Let $c_f=\eta_f^{cong}\in\mathcal{O}$ be the congruence number of $f$. 

With a slight abuse, we shall refer to as a triple of `level-$N$' test vectors for $(\breve{f},\breve{\bfg},\breve{h}_2)$ the triple obtained by specialising level-$N$ test vectors $(\breve{\boldsymbol{\bff}},\breve{\boldsymbol{\bfg}},\breve{\boldsymbol{\bfh}})$ for $(\boldsymbol{\bff},\boldsymbol{\bfg},\boldsymbol{\bfh})$.

\begin{thm}\label{thm:ERL-def}
For every triple $(\breve{f},\breve{\bfg},\breve{h}_2)$ of level-$N$ test vectors for $(f,\bfg,h_2)$ there is an injective $\cO\dBr{W_1}$-module homomorphism with pseudo-null cokernel
\[
{\mathrm{Log}}^f_{\pp,(\breve{f},\breve{\bfg},\breve{h}_2)}:\rH^1(K_\pp,T_f^{\vee,-}(1-r)\otimes\ch_1^{-1}\ch_2^{-1}\Psi_{W_1}^{1-\cc}(\mathscr{N}))\rightarrow c_f^{-1}\cO\dBr{W_1}
\]
such that for all $\mathfrak{Z}\in\rH^1(K_\pp,T_f^{\vee,-}(1-r)\otimes\ch_1^{-1}\ch_2^{-1}\Psi_{W_1}^{1-\cc}(\mathscr{N}))$ and $Q\in\mathfrak{X}_{\cO\dBr{W_1}}^{+}$ with $0\leq j_Q< r$ we have
\[
{\mathrm{Log}}^f_{\pp,(\breve{f},\breve{\bfg},\breve{h}_2)}(\mathfrak{Z})_Q=c_Q\cdot\left\langle{\mathrm{exp}}^*_p(\mathfrak{Z}_{Q}),\eta_{\breve{f}}\otimes\omega_{\breve{\bfg}_{Q'}}\otimes\omega_{h_2}\right\rangle_{\mathrm{dR}},
\]
where $c_Q$ is an explicit nonzero constant, and $Q'\in{\mathrm{Spec}}(\cO\dBr{Z_1})(\overline{\Q}_p)$ is the weight $2j_Q$ specialization given by $Q'(1+Z_1)=\zeta_Q^2\mathbf{u}^{2j_Q-2}$. Moreover, we have 
the \emph{explicit reciprocity law}
\[
{\mathrm{Log}}^f_{\pp,(\breve{f},\breve{\bfg},\breve{h}_2)}\bigl(p_f^-({\mathrm{res}}_\pp(\kappa_1(f,\bfg,h_2)))\bigr)(W_1)=\mathscr{L}_p^{\unb}(\breve{f},\breve{\bfg},\breve{h}_2)(S_1),
\]
where $S_1=\mathbf{u}^2(1+Z_1)-1=(1+W_1)^2-1$.
\end{thm}

\begin{proof}
In terms of (\ref{eq:dec-V-first}), we find that 
\[
\mathscr{F}_p^3(\VVdag)=T_f^{\vee,+}(1-r)\otimes\ch_1^{-1}\ch_2^{-1}\psi_0^{1-\cc}\Psi_{V_1}^{1-\cc}
=T_f^{\vee,+}(1-r)\otimes\ch_1^{-1}\ch_2^{-1}\Psi_{W_1}^{1-\cc}.
\] 
Together with (\ref{eq:bal-def}), this gives the decomposition
\begin{align*}
\mathscr{F}_p^{\bal}(\VVdag)/\mathscr{F}_p^3(\VVdag)&\cong
\bigl(T_f^{\vee,-}(1-r)\otimes\ch_1^{-1}\ch_2^{-1}\Psi_{W_1}^{1-\cc}\bigr)\\
&\quad\oplus\bigl(T_f^{\vee,+}(1-r)\otimes\ch_1^{-1}\ch_2^{-\cc}\Psi_{W_2}^{1-\cc}\bigr)\oplus\bigl(T_f^{\vee,+}(1-r)\otimes\ch_1^{-\cc}\ch_2^{-1}\Psi_{W_2}^{\cc-1}\bigr),
\end{align*}
with the terms in the direct sum corresponding to $\mathbb{V}_{f}^{\bfg h_2}$, $\mathbb{V}_{h_2}^{f\bfg}$, and $\mathbb{V}_{\bfg}^{fh_2}$ from (\ref{eq:gr2}), respectively, and where $W_2=(1+Z_1)^{1/2}-1$ is as in Proposition~\ref{prop:factor-L-def} (with $Z_2=0$).

Thus we find that under the first isomorphism of Proposition~\ref{prop:factor-S}, the composite map in (\ref{eq:map-ERL}) corresponds to the projection to $\Sel_{\relstr}(K,T_f^{\vee}(1-r)\otimes\ch_1^{-1}\ch_2^{-1}\Psi_{W_1}^{1-\cc})$ (the first factor in that decomposition) composed with the natural map
\begin{align*}
\Sel_{\relstr}(K,T_f^{\vee}(1-r)\otimes\ch_1^{-1}\ch_2^{-1}\Psi_{W_1}^{1-\cc})\xrightarrow{{\mathrm{res}}_\pp}
&\rH^1(K_\pp,T_f^{\vee}(1-r)\otimes\ch_1^{-1}\ch_2^{-1}\Psi_{W_1}^{1-\cc})\\
&\xrightarrow{p_f^-}
\rH^1(K_\pp,T_f^{\vee,-}(1-r)\otimes\ch_1^{-1}\ch_2^{-1}\Psi_{W_1}^{1-\cc}),
\end{align*}
and so under the corresponding isomorphisms we have \[{\mathrm{res}}_p(\kappa(f,\bfg,h_2))_f=p_f^-({\mathrm{res}}_\pp(\kappa_1(f,\bfg,h_2)))\] in
\[
\rH^1(\Q_p,\mathbb{V}_f^{\bfg h_2}(\mathscr{N}))\cong\rH^1(K_\pp,T_f^{\vee,-}(1-r)\otimes\ch_1^{-1}\ch_2^{-1}\Psi_{W_1}^{1-\cc}(\mathscr{N})).
\] 
Finally, the construction of ${\mathrm{Log}}^f_{\pp,(\breve{f},\breve{\bfg},\breve{h}_2)}$ is deduced from a specialization of the $3$-variable $p$-adic regulator map ${\mathrm{Log}}^{\bff}_{(\breve{\bff},\breve{\bfg},\breve{\bfh})}$ in \eqref{eq:Log} by the same argument as in \cite[Prop.~7.3]{ACR}, and the stated explicit reciprocity law then follows from Theorem~\ref{thm:ERL}.
\end{proof}


\subsection{On the Bloch--Kato conjecture in rank 0}
\label{subsec:BK-def}

In this section we deduce our first applications to the Bloch--Kato conjecture in analytic rank zero for the twisted $G_K$-representation
\[
V_{f,\chi}:=V_f^\vee(1-r)\otimes\chi^{-1}.
\]

Denote by $K[c]$ the ring class field of $K$ of conductor $c$. 
If $\chi$ is a Hecke character of conductor $c\cO_K$, then its $p$-adic avatar is a locally algebraic character of ${\mathrm{Gal}}(K[cp^\infty]/K)$. The Galois group $\Gamma^-={\mathrm{Gal}}(K_\infty^-/K)$ of the anticyclotomic $\Z_p$-extension of $K$ arises as the maximal $\Z_p$-free quotient of ${\mathrm{Gal}}(K[cp^\infty]/K)$. Fix a (non-canonical) splitting
\begin{equation}\label{eq:split-c}
{\mathrm{Gal}}(K[cp^\infty]/K)\simeq\Delta_c\times\Gamma^-,
\end{equation}
where $\Delta_c$ is the torsion subgroup of ${\mathrm{Gal}}(K[cp^\infty]/K)$. Then every character of $\Delta_c$ can be viewed as the $p$-adic avatar of a ring class character of $K$ of conductor dividing $cp^s\cO_K$ for sufficiently large $s$. If $\chi$ is as above, we then write $\chi=\chi_t\cdot\chi_{w}$ according to the decomposition (\ref{eq:split-c}).

\begin{thm}\label{thm:BK-def}
Let $f\in S_{2r}(\Gamma_0(N_f))$, with $p\nmid N_f$, be a $p$-ordinary newform of weight $2r\geq 2$, let $K$ be an imaginary quadratic field satisfying $\eqref{eq:spl}$ and $\eqref{eq:p-nmid-h}$, and let $\chi$ be an anticyclotomic Hecke character of conductor $c\cO_K$ and infinity type $(-j,j)$, $j\geq 0$. 
Assume that:
\begin{itemize}
    \item $N^-$ is a square-free product of an odd number of primes;
    \item $(pN_f,cD_K)=1$;
    \item $\chi_t$ has conductor prime-to-$p$;
    \item $\bar{\rho}_f$ is absolutely irreducible and $p$-distinguihed;
    \item $p>2r-2$;
    \item $f$ is not of CM-type.
\end{itemize}
Then
\[
L(f/K,\chi,r)\neq 0\quad\Longrightarrow\quad\Sel_{\mathrm{BK}}(K,V_{f,\chi})=0,
\] 
and hence the Bloch--Kato conjecture for $V_{f,\chi}$ holds in analytic rank zero.
\end{thm}

\begin{proof}
We begin by noting that for $j\geq r$ the sign in the functional equation of $L(f/K,\chi,s)$ is $-1$ (and so $L(f/K,\chi,r)=0$, in which case there is nothing to show),
so without loss of generality below we assume that $0\leq j<r$. 

Write $\chi_t=\alpha/\alpha^\cc$ with $\alpha$ a ray class character of $K$ of conductor $\fkf\subset\cO_K$ prime-to-$p$ (as is possible by e.g. \cite[Lem.\,6.9]{DR2} or \cite[Lem.\,5.31]{hida-HMI}  
and our assumption on $\chi_t$). For a prime $q\neq p$ split in $K$ and an auxiliary ring class character $\beta$ of $q$-power conductor (both to be further specified below), we consider the setting of $\S$\ref{subsec:factor-L-def} with the CM Hida families
$(\bfg,\bfh)=(\boldsymbol{\theta}_{\ch_1}(Z_1),\boldsymbol{\theta}_{\ch_2}(Z_2))$ for the ray class characters
\[
\ch_1:=\beta\alpha,\quad\quad\ch_2:=\beta^{-1}\alpha^{-\cc}.
\]
Using $\xi_1\xi_2=\chi_t$ and $\xi_1\xi_2^\cc=\beta^2$, when specialized to $Z_2=0$, the factorization in Proposition~\ref{prop:factor-L-def} becomes
\[
\mathscr{L}_p^{\unb}(f,\bfg,h_2)(S_1)=
\pm\mathbf{w}^{}\cdot\Theta_p^{\BD}(f/K,\chi_t)(W_1)\cdot\Theta_p^{\BD}(f/K,\beta^2)(W_2)\cdot\frac{\eta_{f}^{cong}}{\eta_{f,N^-}},
\]
where $S_1=\mathbf{u}^2(1+Z_1)-1$, $W_1=\mathbf{u}(1+Z_1)^{1/2}-1$, $W_2=(1+Z_1)^{1/2}-1$, and $\mathbf{w}$ is a unit in $\cO\dBr{Z_1}\otimes_{\Z_p}\Q_p$. 
By \cite[Thm.~D]{ChHs1} we may take $q$ and $\beta$ so that $\Theta_p^{\BD}(f/K,\beta^2)(W_2)$ is a unit in $\cO\dBr{W_2}$, and with such a choice the explicit reciprocity law of Theorem~\ref{thm:ERL-def} can be rewritten as
\[
{\mathrm{Log}}^f_{\pp,(\breve{f}^\star,\breve{\bfg}^\star,\breve{h}_2^\star)}\bigl(p_f^-({\mathrm{res}}_\pp(\kappa_1(f,\bfg,h_2)))\bigr)(W_1)=\pm\mathbf{w}'\cdot\Theta_p^{\BD}(f/K,\chi_t)(W_1)
\]
with $\mathbf{w}'$ a unit in $\cO\dBr{W_1}\otimes_{\Z_p}\Q_p$ and  $(\breve{f}^\star,\breve{\bfg}^\star,\breve{h}_2)$ the triple of level-$N$ test vectors from Theorem~\ref{thm:hsieh-triple}. 

Denote by $Q\in\mathfrak{X}_{\cO\dBr{W_1}}^{+}$ the specialization $W_1\mapsto\zeta_Q\mathbf{u}^{j}-1$  ($\zeta_Q\in\mu_{p^\infty}$)  corresponding to $\chi_w$, in the sense that
\[
\chi_w=\Psi_{W_1}^{\cc-1}\vert_{W_1=\zeta_Q\mathbf{u}^j-1}=\psi_0^{\cc-1}\Psi_{V_1}^{\cc-1}\vert_{V_1=\zeta_Q\mathbf{u}^{j-1}-1}.
\]
Then from the above together with Theorem~\ref{thm:BD-theta} and Theorem~\ref{thm:ERL-def} we find
\begin{equation}\label{eq:L-implies-def}
\begin{aligned}
L(f/K,\chi,r)\neq 0\quad&\Longrightarrow\quad\Theta_p^{\BD}(f/K,\chi_t)(\chi_w(\gamma_-)-1)\neq 0\\
&\Longrightarrow\quad p_f^-({\mathrm{res}}_\pp(\kappa_1(\breve{f}^\star,\breve{\bfg}^\star,\breve{h}_2^\star)_Q))\neq 0,
\end{aligned}
\end{equation}
where $\kappa_1(\breve{f}^\star,\breve{\bfg}^\star,\breve{h}_2^\star)$
denotes the image of the class $\kappa_1(f,\bfg,h_2)$ in \eqref{eq:kap-1} under the projection
\begin{align*}
\Sel_{\relstr}(K,T_f^{\vee}(1-r)\otimes\chi_t^{-1}\Psi_{W_1}^{1-\cc}(\mathscr{N}))&\rightarrow\Sel_{\relstr}(K,T_f^{\vee}(1-r)\otimes\chi_t^{-1}\Psi_{W_1}^{1-\cc})\\
&=\Sel_{\relstr}(K,T_f^{\vee}(1-r)\otimes\chi_t^{-1}\psi_0^{1-\cc}\Psi_{V_1}^{1-\cc})
\end{align*}
associated to $(\breve{f}^\star,\breve{\bfg}^\star,\breve{h}_2^\star)$.

As 
noted in Remark~\ref{rem:diag-components}, the class $\kappa_1(\breve{f}^\star,\breve{\bfg}^\star,\breve{h}_2^\star)$ is the bottom class of the anticyclotomic Euler system $\{\mathbf{z}_{f,\xi_1,\xi_2,\psi_0^{\cc-1},m}\}_m$ of Theorem~\ref{maintheorem2} for $T_{f,\xi_1\xi_2\psi_0^{\cc-1}}$ (and the given choice of level-$N$ test vectors). Therefore, letting ${\mathrm{tw}}_{V_1,\psi_0^{\cc-1}\chi_w^{-1}}(\kappa_1(\breve{f}^\star,\breve{\bfg}^\star,\breve{h}_2^\star))$ denote the image of $\kappa_1(\breve{f}^\star,\breve{\bfg}^\star,\breve{h}^\star_2)$ under the `twisting' map
\[
\Sel_{\relstr}(K,T_f^{\vee}(1-r)\otimes\chi_t^{-1}\psi_0^{1-\cc}\Psi_{V_1}^{1-\cc})\rightarrow\Sel_{\relstr}(K,T_f^{\vee}(1-r)\otimes\chi^{-1}\Psi_{V_1}^{1-\cc})
\]
induced by the change of variables $V_1\mapsto\zeta_Q^{-1}\mathbf{u}^{1-j}(1+V_1)-1$, it follows that \[{\mathrm{tw}}_{V_1,\psi_0^{\cc-1}\chi_w^{-1}}(\kappa_1(\breve{f}^\star,\breve{\bfg}^\star,\breve{h}_2^\star))\] is the bottom class of the twisted Euler system of Theorem~\ref{maintheorem2}
\begin{equation}\label{eq:tw-ES-def}
\bigl\{\mathbf{z}_{f,\chi,m}\bigr\}_m:=\bigl\{\mathbf{z}_{f,\xi_1,\xi_2,\psi_0^{\cc-1},m}\otimes\psi_0^{\cc-1}\chi_w^{-1}\bigr\}_m
\end{equation}
for $T_{f,\xi_1\xi_2\psi_0^{\cc-1}}\otimes\psi_0^{\cc-1}\chi_w^{-1}=T_{f,\chi}$. 

Since the class $\kappa_1(\breve{f}^\star,\breve{\bfg}^\star,\breve{h}_2^\star)_Q$ in \eqref{eq:L-implies-def} is the same as the image of the bottom class $\mathbf{z}_{f,\chi,1}$ of the system \eqref{eq:tw-ES-def} under natural map
\[
\Sel_{\relstr}(K,T_f^{\vee}(1-r)\otimes\chi^{-1}\Psi_{V_1}^{1-\cc})\cong\Sel_{\relstr}(K_\infty^-,T_{f,\chi})\rightarrow\Sel_{\relstr}(K,T_{f,\chi}),
\]
from Theorem~\ref{thm:rank-1-general} 
we deduce that $\Sel_{\relstr}(K,V_{f,\chi})$ is one-dimensional, spanned by $\kappa_1(\breve{f}^\star,\breve{\bfg}^\star,\breve{h}_2^\star)_Q$. 
Since we have in fact shown that $p_f^-({\mathrm{res}}_\pp(\kappa_1(\breve{f}^\star,\breve{\bfg}^\star,\breve{h}_2^\star)_Q))\neq 0$, from the global duality exact sequence
\begin{align*}
0\rightarrow\Sel_{\Ord,{\mathrm{str}}}(K,V_{f,\chi})
\rightarrow&\Sel_{\relstr}(K,V_{f,\chi})
\xrightarrow{{\mathrm{res}}_\pp}\frac{H^1(K_\pp,V_{f,\chi})}{H^1_{\Ord}(K_\pp,V_{f,\chi})}\\
&\rightarrow\Sel_{{\mathrm{rel}},\Ord}(K,V_{f,\chi})^\vee\rightarrow\Sel_{\relstr}(K,V_{f,\chi})^\vee\rightarrow 0,
\end{align*}
we deduce that $\Sel_{{\mathrm{rel}},\Ord}(K,V_{f,\chi})$ is also one-dimensional and spanned by $\kappa_1(\breve{f}^\star,\breve{\bfg}^\star,\breve{h}_2^\star)_Q$ (hence equal to $\Sel_{\relstr}(K,V_{f,\chi})$). Finally, from another global duality exact sequence 
\begin{align*}
0\rightarrow\Sel_{\ord}(K,V_{f,\chi})
\rightarrow &\Sel_{\mathrm{rel,ord}}(K,V_{f,\chi})
\xrightarrow{{\mathrm{res}}_\pp}\frac{H^1(K_\pp,V_{f,\chi})}{H^1_{\Ord}(K_\pp,V_{f,\chi})}\\
&\rightarrow\Sel_{\ord}(K,V_{f,\chi})^\vee\rightarrow\Sel_{\mathrm{ord,str}}(K,V_{f,\chi})^\vee\rightarrow 0,
\end{align*}
we deduce the vanishing of $\Sel_{\ord}(K,V_{f,\chi})$; 
since by Lemma~\ref{lem:BK-Gr}, for $0\leq j<r$ the latter group agrees with $\Sel_{\mathrm{BK}}(K,V_{f,\chi})$, this yields the result.
\end{proof}

\subsection{On the Iwasawa main conjecture}
\label{subsec:IMC-def}

Our next application is to a divisibility in the anticyclotomic Iwasawa main conjecture for modular forms in the definite setting. 

For any eigenform $f$ of weight $2r\geq 2$ and an anticyclotomic Hecke character $\chi$,
put 
\[
A_{f,\chi}={\mathrm{Hom}}_{\Z_p}(T_f^\vee(1-r)\otimes\chi^{-1},\mu_{p^\infty}),
\]
and writing $\chi=\chi_t\cdot\chi_w$ as in Theorem~\ref{thm:BK-def}, let $\Theta_p^{\BD}(f/K,\chi)$ denote the image of the $p$-adic $L$-function $\Theta_p^{\BD}(f/K,\chi_t)$ of Theorem~\ref{thm:BD-theta} attached to the ring class character $\chi_t$ under the twisting homomorphism
${\mathrm{tw}}_{\chi_w}:\cO\dBr{W_1}\rightarrow\cO\dBr{W_1}$ given by $W_1\mapsto\chi_w(\gamma_-)(1+W_1)-1$. 

\begin{thm}\label{thm:IMC-def}
Let the hypotheses be as in Theorem~\ref{thm:BK-def}, and assume in addition that $f$ has big image. 
Then $\Sel_{\ord}(K_\infty^-,A_{f,\chi})$ is cotorsion over $\Lambda_K^-$, and we have the divisibility
\[
{\mathrm{char}}_{\Lambda_K^-}\bigl(\Sel_{\ord}(K_\infty^-,A_{f,\chi})^\vee\bigr)\supset\bigl(\Theta_p^{\BD}(f/K,\chi)^2\bigr)
\]
in $\Lambda_K^-\otimes_{\Z_p}\Q_p$.
\end{thm}

\begin{proof}
Repeating the argument in the proof of Theorem~\ref{thm:BK-def}, we arrive at the equality
\begin{equation}\label{eq:ERL-def}
{\mathrm{Log}}^f_{\pp,(\breve{f}^\star,\breve{\bfg}^\star,\breve{h}_2^\star)}\bigl(p_f^-({\mathrm{res}}_\pp(\kappa_1(f,\bfg,h_2)))\bigr)(W_1)=\pm\mathbf{w}'\cdot\Theta_p^{\BD}(f/K,\chi_t)(W_1)
\end{equation}
with $\mathbf{w}'$ a unit in $\cO\dBr{W_1}\otimes_{\Z_p}\Q_p$. 

It follows from Vatsal's result \cite[Thm.~1.1]{vatsal-special} (as extended   
in \cite[Thm.~C]{ChHs1} and \cite[Thm.~B]{hung} to higher weights) that the $p$-adic $L$-function $\Theta_p^{\BD}(f/K,\chi_t)(W_1)$ is nonzero. Thus, letting 
\[
\kappa_1(\breve{f}^\star,\breve{\bfg}^\star,\breve{h}_2^\star)\in\Sel_{\relstr}(K,T_f^{\vee}(1-r)\otimes\chi_t^{-1}\psi_0^{1-\cc}\Psi_{V_1}^{1-\cc}) 
\]
be as in the proof of Theorem~\ref{thm:BK-def}, from \eqref{eq:ERL-def} 
 we conclude that $\kappa_1(\breve{f}^\star,\breve{\bfg}^\star,\breve{h}_2)$ is non-torsion. 

As noted in the proof of Theorem~\ref{thm:BK-def}, the twisted class ${\mathrm{tw}}_{V_1,\psi_0^{\cc-1}\chi_w^{-1}}(\kappa_1(\breve{f}^\star,\breve{\bfg}^\star,\breve{h}_2^\star))$ is the bottom class of the Euler system $\{\mathbf{z}_{f,\chi,m}\}_m$ for $T_{f,\chi}$ constructed in Theorem~\ref{maintheorem2}. 
Hence from Theorem~\ref{thm:IMC-general} we deduce that $\Sel_{\relstr}(K,T_{f,\chi})$ and $X_{\strrel}(K,A_{f,\chi})$ have both $\Lambda_K^-$-rank one, and we have the divisibility
\begin{equation}\label{eq:IMC-div-relstr}
{\mathrm{char}}_{\Lambda_{K}^-}\bigl(X_{\strrel}(K,A_{f,\chi})_{\mathrm{tors}}\bigr)\supset{\mathrm{char}}_{\Lambda_K^-}\biggl(\frac{\Sel_{\relstr}(K,T_{f,\chi})}{\Lambda_K^-\cdot{\mathrm{tw}}_{V_1,\psi_0^{\cc-1}\chi_w^{-1}}(\kappa_1(\breve{f}^\star,\breve{\bfg}^\star,\breve{h}_2^\star))}\biggr)^2
\end{equation}
in $\Lambda_K^-$. Since from \eqref{eq:ERL-def} we deduce an explicit reciprocity law relating 
\[
{\mathrm{res}}_\pp({\mathrm{tw}}_{V_1,\psi_0^{\cc-1}\chi_w^{-1}}(\kappa_1(\breve{f}^\star,\breve{\bfg}^\star,\breve{h}_2^\star)))=
{\mathrm{res}}_\pp({\mathrm{tw}}_{W_1,\chi_w^{-1}}(\kappa_1(\breve{f}^\star,\breve{\bfg}^\star,\breve{h}_2^\star)))
\]
to ${\mathrm{tw}}_{\chi_w}(\Theta_p^{\BD}(f/K,\chi_t))=\Theta_p^{\BD}(f/K,\chi)$, the result now follows from \eqref{eq:IMC-div-relstr} and global duality by the same argument as in \cite[Thm.~5.1]{BCK-PRconj}.
\end{proof}

\begin{rmk} An upper bound divisibility in the anticyclotomic Iwasawa main conjecture for $V_{f,\chi}$ as in Theorem~\ref{thm:IMC-def} 
was first obtained by Bertolini--Darmon \cite{bdIMC} for finite order character $\chi$ in weight $2$ and by Chida--Hsieh \cite{ChHs2} in higher weights $2\leq k<p-1$ using Heegner cycles and 
level-raising congruences. Our proof of Theorem~\ref{thm:IMC-def} is completely different from theirs (instead, it is more in line with Kolyvagin's original arguments), and allows us to dispense with their ramification hypotheses on $\bar{\rho}_f$. 
\end{rmk}

\subsection{On the Bloch--Kato conjecture in rank 1} 

The arguments in the proof of Theorem~\ref{thm:IMC-def} give the following result towards the Bloch--Kato conjecture in rank $1$.

\begin{thm}\label{thm:BK-def-1}
Let the hypotheses be as in Theorem~\ref{thm:BK-def}. 
If $j\geq r$ (which implies the vanishing of $L(f/K,\chi,r)$), then 
\[
\dim_{L_\mathfrak{P}}\,\Sel_{\mathrm{BK}}(K,V_{f,\chi})\geq 1.
\]
Moreover, there exists a class $z_{f,\chi}\in\Sel_{\mathrm{BK}}(K,V_{f,\chi})$ such that
\[
z_{f,\chi}\neq 0\quad\Longrightarrow\quad\dim_{L_\mathfrak{P}}\,\Sel_{\mathrm{BK}}(K,V_{f,\chi})=1.
\]
\end{thm}

\begin{proof}
The proof of Theorem~\ref{thm:IMC-def} showed that the class
\[
\mathbf{z}_{f,\chi}:={\mathrm{tw}}_{V_1,\psi_0^{\cc-1}\chi_w^{-1}}(\kappa_1(\breve{f}^\star,\breve{\bfg}^\star,\breve{h}_2^\star))\in\Sel_{\relstr}(K_\infty^-,T_{f,\chi})
\]
is non-torsion over $\Lambda_K^-$ (note that $f$ is not required to have big image for this). On the other hand, one readily checks that the natural map
\begin{equation}\label{eq:proj}
\Sel_{\relstr}(K_\infty^-,T_{f,\chi})/(\gamma_--1)\Sel_{\relstr}(K_\infty^-,T_{f,\chi})\rightarrow\Sel_{\relstr}(K,T_{f,\chi})
\end{equation}
is injective. Thus we conclude that $\Sel_{\relstr}(K,T_{f,\chi})$ has positive $\cO$-rank, which together with  Lemma~\ref{lem:BK-Gr} yields the first part of the theorem. Letting $z_{f,\chi}\in\Sel_{\relstr}(K,T_{f,\chi})$ be the image of $\mathbf{z}_{f,\chi}$ under (\ref{eq:proj}), 
the second claim follows from Theorem~\ref{thm:rank-1-general}.
\end{proof}

\begin{rmk}\label{rmk:Heeg-BF}
%
From the Euler system of Beilinson--Flach elements constructed by Lei--Loeffler--Zerbes and Kings--Loeffler--Zerbes \cite{LLZ,KLZ} attached to the Rankin--Selberg convolution of $f$ and a suitable CM form, one can produce a class ${\mathrm{BF}}_{f,\chi}\in\rH^1(K,V_{f,\chi})$. As shown in \cite{LLZ-K} and \cite{BL}, this class extends to a full Euler system for the $G_K$-representation $V_{f,\chi}$ , but not for the correct local conditions at $p$. Indeed, with notations as in the proof of Theorem~\ref{thm:BK-def}, it follows from the explicit reciprocity law of \cite{KLZ} that, for $j\geq r$, the class ${\mathrm{BF}}_{f,\chi}$ lands in $\Sel_{\relstr}(K,V_{f,\chi})=\Sel_{\mathrm{BK}}(K,V_{f,\chi})$ precisely when 
\begin{equation}\label{eq:theta=0}
\Theta_p^{\BD}(f/K,\chi_t)(\chi_w(\gamma_-)-1)=0
\end{equation}
(see \cite[Thm.~2.4]{cas-BF} and \cite[Thm.~3.11]{BL} for a specialization of the results of \cite{KLZ} to this case).  
However, for $j\geq r$ the character $\chi_w$ is \emph{outside} the range of interpolation of $\Theta_p^{\BD}(f/K,\chi_t)$, 
and so the vanishing \eqref{eq:theta=0} is \emph{not} a consequence of  $L(f/K,\chi,r)=0$. As a result,  Theorem~\ref{thm:BK-def-1} seems to fall outside the scope of methods building on these classes.
(On the other hand, Heegner cycles also seem to not be enough, since $N^-$ is assumed to have an odd number of prime factors, rending Heegner cycles not directly accessible, and in this definite setting the level-raising techniques of Bertolini--Darmon \cite{bdIMC} are only known to yield results towards the Bloch--Kato conjecture  in rank $0$, see e.g. \cite{LV-JNT}.) 
\end{rmk}

\section{Indefinite case}

In this section we deduce our applications to the Bloch--Kato conjecture (in ranks $0$ and $1$) for anticyclotomic twists of $f/K$ when  $\epsilon(f/K)=-1$. Since the nonvanishing results we shall need 
from \cite{hsieh} are currently only available in the literature under the classical Heegner hypothesis, in the following we shall restrict to this case, but we note that with the required extension of \cite{hsieh} at hand (see \cite{burungale-II,Magr} for progress in this direction), our results directly extend to the general indefinite case.

\subsection{Anticyclotomic \texorpdfstring{$p$}{p}-adic \texorpdfstring{$L$}{L}-functions}
\label{subsec:Lp-indef}

Keeping the setting introduced in Section~\ref{sec:definite}, 
we assume now that $K$ satisfies the classical \emph{Heegner hypothesis}:
\begin{equation}\label{eq:Heeg}
\textrm{every prime $\ell\mid N_f$ splits in $K$,}\tag{Heeg}
\end{equation}
and fix an ideal $\mathfrak{N}\subset\cO_K$ with $\cO_K/\mathfrak{N}\simeq\Z/N_f\Z$. 

Let $\Omega_p$ and $\Omega_K$ be the CM periods attached to $K$ as in \cite[\S{2.5}]{cas-hsieh1}, and put  
\[
\Lambda_{K}^{-,\ur}=\Lambda_K^-\hat\otimes_{\Z_p}\Z_p^{\ur},
\]
where $\Z_p^{\ur}$ is the completion of the ring of integers of the maximal unramified extension of $\Q_p$.

\begin{thm}\label{thm:BDP}
Let $\chi_0$ be an $\cO$-valued ring class character of $K$ of conductor $c\cO_K$ with $(pN_f,cD_K)=1$. Then there exists a unique element 
$\mathscr{L}_\pp^{\BDP}(f/K,\chi_0)\in\Z_p^{\ur}\dBr{W}\otimes_{\Z_p}\cO$ such that every character $\phi$ of $\Gamma^-$ of infinity type $(-j,j)$ with $j\geq r$ and conductor $p^n$, we have
\begin{align*}
\mathscr{L}_\pp^{\BDP}(f/K,\chi_0)^2(\phi(\gamma_-)-1)=\frac{\Omega_p^{4j}}{\Omega_K^{4j}}\cdot&\frac{\Gamma(r+j)\Gamma(j+1-r)\phi(\mathfrak{N}^{-1})}{4(2\pi)^{2j+1}\sqrt{D_K}^{2j-1}}\\
\times &e_\pp(f,\chi_0\phi)\cdot L(f/K,\chi_0\phi,r),
\end{align*} 
where
\[
e_\pp(f,\chi_0\phi)=\begin{cases}
\bigl(1-a_p\chi_0\phi(\bar{\pp})p^{-r}+\chi_0\phi(\bar{\pp})^2p^{-1}\bigr)^2&\textrm{if  $n=0$,}\\[0.3em]
\varepsilon(\frac{1}{2},(\chi_0\phi)_\pp)^{-2}&\textrm{else},
\end{cases}  
\]
with $\varepsilon(\frac{1}{2},(\chi_0\phi)_\pp)$ the epsilon-factor in \cite[p.\,570]{cas-hsieh1} attached to the component of $\chi_0\phi$ at $\pp$. 
Moreover, $\mathscr{L}_\pp^{\BDP}(f/K,\chi_0)$ is a nonzero element of $\Lambda_K^{-,{\ur}}$.
\end{thm}

\begin{proof}
This is a reformulation of results contained in \cite[\S{3}]{cas-hsieh1}. Note that since $(N_f,D_K)=1$ is a consequence of  (\ref{eq:Heeg}), the nonvanishing of $\mathscr{L}_\pp^{\BDP}(f/K,\chi_0)$ follows from \cite[Thm.~3.9]{cas-hsieh1}.
\end{proof}

\begin{rmk}\label{rem:periods}
The CM period $\Omega_{K}\in\bC^\times$ in Theorem~\ref{thm:BDP} agrees with that in \cite[(5.1.16)]{bdp1}, but is \emph{different} from the period $\Omega_\infty$ defined in \cite[p.\,66]{deshalit} and \cite[(4.4b)]{HT-ENS}. In fact, one has 
\[
\Omega_{\infty}=2\pi i\cdot\Omega_K.
\]
In terms of $\Omega_\infty$, the interpolation formula in Theorem~\ref{thm:BDP} reads
\begin{align*}
\mathscr{L}_{\pp}^{\BDP}(f/K,\chi_0)^2(\phi(\gamma_-)-1)=\frac{\Omega_p^{4j}}{\Omega_\infty^{4j}}\cdot&\frac{\Gamma(r+j)\Gamma(j+1-r)\phi(\mathfrak{N}^{-1})}{4(2\pi)^{1-2j}\sqrt{D_K}^{2j-1}}\\ \times &e_\pp(f,\chi_0\phi)\cdot L(f/K,\chi_0\phi,r).
\end{align*}
This is the form of the interpolation that we shall use later.
\end{rmk}

\subsection{Factorization of triple product \texorpdfstring{$p$}{p}-adic \texorpdfstring{$L$}{L}-function}\label{subsec:factor-L-indef}

As in $\S\ref{subsec:factor-L-def}$, we consider the triple  $(\bff,\bfg,\bfh)$, with $\bff\in S^o(N_f,\omega^{2r-2},\mathbb{I})$ the Hida family specialising to $f$ at an arithmetic point $Q_0\in\mathfrak{X}_\cR^+$ of weight $2r$, and 
\[
(\bfg,\bfh)=(\boldsymbol{\theta}_{\xi_1}(Z_1),\boldsymbol{\theta}_{\xi_2}(Z_2))\in\cO\dBr{Z_1}\dBr{q}\times\cO\dBr{Z_2}\dBr{q}
\]
the CM Hida families of $\S\ref{subsec:CM}$ attached to the ray class characters $\xi_1,\xi_2$ satisfying \eqref{eq:sd-xi}, and also \eqref{eq:dist}.

The triple product $p$-adic $L$-function of interest in this section is the $\bfg$-unbalanced $p$-adic $L$-function 
\begin{equation}\label{eq:g-unb}
\mathscr{L}_p^{\bfg}(\bff,\bfg,\bfh)\in\mathcal{R}=\cR\hat\otimes_{\cO}\cO\dBr{Z_1}\hat\otimes_{\cO}\cO\dBr{Z_2}\simeq\cR\dBr{Z_1,Z_2}
\end{equation}
obtained from Theorem~\ref{thm:hsieh-triple} with the roles of $\bff$ and $\bfg$ reversed (note that the conditions in Theorem~\ref{thm:hsieh-triple} in this setting are ensured by \eqref{eq:dist} and our hypothesis on the conductor of $\xi_1$). In the following we let 
\[
\mathscr{L}_p^{\bfg}(f,\bfg,\bfh)\in\cO\dBr{Z_1,Z_2}
\] 
be the image of $\mathscr{L}_p^{\bfg}(\bff,\bfg,\bfh)$ under the map $\cR\dBr{Z_1,Z_2}\rightarrow\cO\dBr{Z_1,Z_2}$ given by 
$Q_0:\cR\rightarrow\cO$. 

\subsubsection{Anticyclotomic Katz $p$-adic $L$-function}

Before we can state and prove the main result of this section, we need to recall the interpolation property of the Katz $p$-adic $L$-functions \cite{Katz49}, following the exposition in \cite{deshalit}. For any ideal $\mathfrak{c}\subset\cO_K$ coprime to $p$, we let $Z(\mathfrak{c})$ denote the ray class group of $K$ of conductor $\mathfrak{c}p^\infty$ (so $Z(\mathfrak{c})\simeq{\mathrm{Gal}}(K_{\mathfrak{c}p^\infty}/K)$). 

\begin{thm}\label{thm:katz}
There exists an element $\Lcal_{\pp,\mathfrak{c}}^{\mathrm{Katz}}
\in\cO\dBr{Z(\mathfrak{c})}\hat\otimes_{\Z_p}\Z_p^{\ur}$ 
such that for every character $\xi$ of $Z(\mathfrak{c})$ that is crystalline at both $\pp$ and $\ppbar$, corresponding to a Hecke character of infinity type $(k,j)$ with $k>0$ and $j\leq 0$, then $\xi$  satisfies 
\[
\Lcal_{\pp,\mathfrak{c}}^{\mathrm{Katz}}(\xi)=\frac{\Omega_p^{k-j}}{\Omega_\infty^{k-j}}\cdot\Gamma(k)\cdot\biggl(\frac{\sqrt{D_K}}{2\pi}\biggr)^j\cdot(1-\xi^{-1}(\pp)p^{-1})(1-\xi(\bar{\pp}))\cdot L_{\mathfrak{c}}(\xi,0),
\]
where $L_{\mathfrak{c}}(\xi,s)$ denotes the Hecke $L$-function of $\xi$ with the Euler factors at the primes dividing $\mathfrak{c}$ removed.  
Moreover, we have the functional equation 
\[
\Lcal_{\pp,\mathfrak{c}}^{\mathrm{Katz}}(\xi)=\Lcal_{\pp,\bar{\mathfrak{c}}}^{\mathrm{Katz}}(\xi^{-\cc}\mathbf{N}^{-1}),
\] 
where the equality is up to a $p$-adic unit.
\end{thm}

\begin{proof}
Our $\Lcal_{\pp,\mathfrak{c}}^{\mathrm{Katz}}$ corresponds to the measure denoted $\mu(\mathfrak{c}\bar{\pp}^\infty)$ in \cite[Thm.~II.4.14]{deshalit}, 
and the functional equation is given in \cite[Thm.~II.6.4]{deshalit} (which allows to extend the interpolation property from $k>-j\geq 0$ to the entire range in the statement).
\end{proof}


Let $\Gamma_\mathfrak{c}$ be the maximal torsion-free subgroup of $Z(\mathfrak{c})$, and fix a (non-canonical) splitting
\[
Z(\mathfrak{c})\simeq\Delta_{\mathfrak{c}}\times\Gamma_{\mathfrak{c}}
\]
with $\Delta_{\mathfrak{c}}$ a finite group and $\Gamma_{\mathfrak{c}}\simeq\Z_p^2$. For $\mathfrak{c}'\supset\mathfrak{c}$ the natural projection $Z(\mathfrak{c})\twoheadrightarrow Z(\mathfrak{c}')$ takes $\Delta_{\mathfrak{c}}$ to $\Delta_{\mathfrak{c}'}$, inducing an isomorphism $\Gamma_{\mathfrak{c}}\xrightarrow{\sim}\Gamma_{\mathfrak{c}'}$. Thus in the following we shall identify $\Gamma_{\mathfrak{c}}$ with $\Gamma_\infty:=\Gamma_{(1)}$, the Galois group of the $\Z_p^2$-extension of $K$ as introduced in $\S\ref{subsec:CM}$. 

Suppose $\eta$ is a Hecke character of $K$ of conductor dividing $\mathfrak{c}p^\infty$. Viewing $\eta$ as a character on $Z(\mathfrak{c})\simeq\Delta_{\mathfrak{c}}\times\Gamma_K$, we put $\bar{\eta}:=\eta\vert_{\Delta_\mathfrak{c}}$, and denote by $\mathcal{L}_{\pp,\bar{\eta}}^{{\mathrm{Katz}},-}$ the image of $\mathcal{L}_{\pp,\mathfrak{c}}^{\mathrm{Katz}}$ under the composite map 
\[
\cO\dBr{Z(\mathfrak{c})}\hat\otimes_{\Z_p}\Z_p^{\ur}\rightarrow\cO\dBr{\Gamma_K}\hat\otimes_{\Z_p}\Z_p^{\ur}\rightarrow\Lambda_K^{-,{\ur}},
\] 
where the first arrow is the natural projection defined by $\bar{\eta}$, and the second arrow is defined by $\gamma\mapsto\gamma^{1-\cc}$ for $\gamma\in\Gamma_\infty$. Put also $\bar{\eta}^-:=\bar{\eta}^{\cc-1}$.

\begin{lem}\label{lem:HT}
Let $\ch$ be a ray class character of $K$ such that $\bar{\ch}^-$ has conductor $\mathfrak{c}$ prime-to-$p$. Assume that:
\begin{enumerate}
\item[(i)] $\mathfrak{c}$ is only divisible by primes that are split in $K$; 
\item[(ii)] $\Delta_{\mathfrak{c}}$ has order prime-to-$p$; 
\item[(iii)] $\bar{\xi}^-\vert_{G_{K_v}}\neq 1$ for all primes $v\mid p$ in $K$;
\item[(iv)] $\bar{\ch}^-$ has order at least $3$. 
\end{enumerate}
Then the congruence ideal of the CM Hida family $\boldsymbol{\theta}_{\ch}(Z)$ in \eqref{eq:CM-explicit} is principal, generated by $\mathcal{L}_{\pp,\bar{\ch}^-}^{{\mathrm{Katz}},-}$.
\end{lem}

\begin{proof}
As explained in \cite[Prop.~4.6]{ACR2}, this is a consequence of the proof of the anticyclotomic Iwasawa main conjecture for Hecke characters by Hida--Tilouine \cite{HT-ENS,HT-117} and Hida \cite{hida-coates} (recall that here we assume \eqref{eq:p-nmid-h}, so the omitted term $h_K$ is a $p$-adic unit). 
\end{proof}

\subsubsection{The factorization result}

We now fix our choice of generator of the congruence ideal of $\bfg=\boldsymbol{\theta}_{\xi_1}(Z_1)$.

\begin{defn}\label{def:gen-triple}
For $\xi_1$ satisfying the conditions of Lemma~\ref{lem:HT} (in particular, note that (iii) is equivalent to \eqref{eq:dist}), put
\[
\mathscr{L}_p^{\bfg}(\bff,\bfg,\bfh):=\mathcal{L}_{\pp,\bar{\xi}_1^-}^{{\mathrm{Katz}},-}\cdot\mathscr{L}_p^{\bfg}(\breve{\bff}^\star,\breve{\bfg}^\star,\breve{\bfh}^\star),
\]
where $(\breve{\bff}^\star,\breve{\bfg}^\star,\breve{\bfh}^\star)$ is the triple of level-$N$ test vectors from Theorem~\ref{thm:hsieh-triple} (see also Remark~\ref{rmk:harris-kudla}), and let $\mathscr{L}_p^{\bfg}(f,\bfg,\bfh)$ denote its image under the map induced by $Q_0:\cR\rightarrow\cO$.
\end{defn}

Note that $\ch_1$ can be replaced by a twist $\ch_1\cdot\phi\circ\mathbf{N}$ for a Dirichlet character $\phi$ without changing $\bar\ch_1^-$, and thus in the following we may assume that $\ch_1$ satisfies the following minimality hypotheses:
\begin{equation}\label{eq:minimal}
\textrm{the conductor of $\ch_1$ is minimal among Dirichlet twists.} 
\end{equation}

The following is an analogue of Proposition~\ref{prop:factor-L-def} in the indefinite setting. Note that a variant of this result first appeared in the work of Darmon--Lauder--Rotger (see \cite[Thm.~3.9]{DLR-stark}), but the formulation of their result is not well-suited for our Iwasawa-theoretic purposes in this paper. 

\begin{prop}\label{prop:factor-L-indef}
Assume that $\ch_1$ satisfies the conditions in Lemma~\ref{lem:HT}.  Set 
\[
S_i=\mathbf{u}^2(1+Z_i)-1
\]
for $i=1,2$, and 
\[
W_1=\mathbf{u}^{-1}(1+S_1)^{1/2}(1+S_2)^{1/2}-1,\quad\quad
W_2=(1+S_1)^{1/2}(1+S_2)^{-1/2}-1.
\]
Then
\begin{align*}
\mathscr{L}_p^{\bfg}(f,\bfg,\bfh)(S_1,S_2)
&=\pm\mathbf{w}\cdot\mathscr{L}_\pp^{\BDP}(f/K,\ch_1\ch_2)(W_1)\cdot\mathscr{L}_\pp^{\BDP}(f/K,\ch_1\ch_2^\cc)(W_2),
\end{align*}
where $\mathbf{w}$ is a unit in $\cO\dBr{Z_1,Z_2}\otimes_{\Z_p}\bQ_p$.
\end{prop}

\begin{proof}
Let $k_1, k_2$ be integers with $k_1\equiv k_2\pmod{2}$ and $k_1\geq k_2+2r$. Set $S_i=\mathbf{u}^{k_i}-1$ for $i=1,2$, so the corresponding specializations of $W_i$ are given by 
\[
W_1=\mathbf{u}^{(k_1+k_2-2)/2}-1,\quad\quad W_2=\mathbf{u}^{(k_1-k_2)/2}-1,
\]
and denote by $\VQdag$ the specialization of $\Vdag$ at $\underline{Q}=(Q_0,S_1,S_2)$. 
Putting 
\[
T_i=\mathbf{u}^{-1}(1+S_i)-1=\mathbf{u}(1+Z_i)-1
\]
for the ease of notation, we have 
\[
\det(T_f^{\vee}\otimes V_{\bfg_{T_1}}\otimes V_{\bfh_{T_2}})=\varepsilon_{\mathrm{cyc}}^{2r-1}\cdot(\ch_1\ch_2\Psi_{T_1}\Psi_{T_2}\circ\mathscr{V})=\varepsilon_{\mathrm{cyc}}^{2r-1}\cdot(\Psi_{T_1}\Psi_{T_2}\circ\mathscr{V}),
\]
using that the central characters of $\ch_1$ and $\ch_2$ are inverses of each other for the second equality, and so
\begin{equation}\label{eq:dec-V}
\begin{aligned}
\VQdag&= T_f^{\vee}\otimes({\mathrm{Ind}}_K^\Q\ch_1^{-1}\Psi_{T_1})\otimes({\mathrm{Ind}}_K^\Q\ch_2^{-1}\Psi_{T_2})\otimes\varepsilon_{\mathrm{cyc}}^{1-r}(\Psi_{T_1}^{-1/2}\Psi_{T_2}^{-1/2}\circ\mathscr{V})\\
&\simeq\bigl(T_f^{\vee}(1-r)\otimes{\mathrm{Ind}}_{K}^\Q\ch_1^{-1}\ch_2^{-1}\Psi_{W_1}^{1-\cc}\bigr)\oplus\bigl(T_f^{\vee}(1-r)\otimes{\mathrm{Ind}}_K^\Q\ch_1^{-1}\ch_2^{-\cc}\Psi_{W_2}^{1-\cc}\bigr).
\end{aligned}
\end{equation} 
Thus we find that the completed $L$-value appearing in the interpolation formula of Theorem~\ref{thm:hsieh-triple} is given by
\begin{equation}\label{eq:explicit}
\begin{aligned}
\Gamma_{\VQdag}(0)\cdot L(\VQdag,0)&=\frac{\Gamma\bigl(\frac{k_1+k_2}{2}+r-1\bigr)\Gamma\bigl(\frac{k_1-k_2}{2}-r+1\bigr)\Gamma\bigl(\frac{k_1+k_2}{2}-r\bigr)\Gamma\bigl(\frac{k_1-k_2}{2}+r\bigr)}{2^4\cdot(2\pi)^{2k_1}}\\
&\quad\times L(f/K,\ch_1\ch_2\Psi_{W_1}^{\cc-1},r)\cdot L(f/K,\ch_1\ch_2^\cc\Psi_{W_2}^{\cc-1},r),
\end{aligned}
\end{equation}
and similarly the modified Euler factor decomposes as
\begin{equation}\label{eq:Ep}
\begin{aligned}
\mathcal{E}_p(\mathscr{F}_p^{\bfg}(\VQdag))&=\,\bigr(1-a_p(\ch_1\ch_2\Psi_{W_1}^{\cc-1})(\bar{\pp})p^{-r}+(\ch_1\ch_2\Psi_{W_2}^{\cc-1})(\bar{\pp})^2p^{-1}\bigl)^2\\
&\quad\,\times\bigr(1-a_p(\ch_1\ch_2^\cc\Psi_{W_2}^{\cc-1})(\bar{\pp})p^{-r}+(\ch_1\ch_2^\cc\Psi_{W_2}^{\cc-1})(\bar{\pp})^2p^{-1}\bigl)^2.
\end{aligned}
\end{equation}
Moreover, letting $\chi_g'$ be the prime-to-$p$ part of the nebentypus character of $\bfg_{T_1}$, we have
\begin{align*}
\biggl(1-\frac{\chi_g'(p)p^{k_1-1}}{\ch_1\Psi^{-1}_{T_1}(\bar{\pp})^2}\biggr)\biggl(1-\frac{\chi_g'(p)p^{k_1-2}}{\ch_1\Psi^{-1}_{T_1}(\bar{\pp})^2}\biggr)
&=\bigl(1-\ch_1^{\cc-1}\Psi_{T_1}^{1-\cc}(\ppbar)\bigr)\bigl(1-\ch_1^{\cc-1}\Psi^{1-\cc}_{T_1}(\ppbar)p^{-1}\bigr),
\end{align*}
and therefore the canonical period $\Omega_{\bfg_{T_1}}$ in Theorem~\ref{thm:hsieh-triple} (associated with the generator $\eta_{\bfg}^\star=\mathcal{L}_{\pp,\bar{\xi}_1^-}^{{\mathrm{Katz}},-}$ of $C(\bfg)$ from Lemma~\ref{lem:HT}) is given by
\begin{equation}\label{eq:period}
\Omega_{\bfg_{T_1}}=(-2\sqrt{-1})^{k_1+1}\cdot\frac{\Vert\bfg_{T_1}^\circ\Vert_{\Gamma_0(C)}^2}{\eta_{\bfg_{T_1}}^{\star}}\cdot\bigl(1-\ch_1^{\cc-1}\Psi_{T_1}^{1-\cc}(\ppbar)\bigr)\bigl(1-\ch_1^{\cc-1}\Psi^{1-\cc}_{T_1}(\ppbar)p^{-1}\bigr),
\end{equation}
where $C=N_{K/\Q}(\fkf_1)D_K$, and we note that we may ignore the term $\prod_{q\in\Sigma_{\mathrm{exc}}}(1+q^{-1})^2$ from Theorem~\ref{thm:hsieh-triple}, since (up to a $p$-adic unit) it only contributes a fixed power of $p$ (in particular, independent of $k_1,k_2$).
%

On the other hand, since $g_{T_1}$ has weight $k_1$, from Hida's formula for the adjoint $L$-value \cite[Thm.~7.1]{HT-ENS} (using that $\xi_1$ satisfies the minimality condition (\ref{eq:minimal})) and Dirichlet's class number formula we obtain 
\[
\Vert \bfg_{T_1}^\circ\Vert^2_{\Gamma_0(C)}=\Gamma(k_1)\cdot\frac{D_K^2}{2^{2k_1}\pi^{k_1+1}}\cdot\frac{2\pi h_K}{w_K\sqrt{D_K}}\cdot L(\ch_1^{\cc-1}\Psi_{T_1}^{1-\cc},1),
\]
where 
$w_K=\vert\cO_K^\times\vert$. 

Note that $L(\ch_1^{\cc-1}\Psi_{T_1}^{1-\cc},1)=L(\ch_1^{\cc-1}\Psi_{T_1}^{1-\cc}\mathbf{N}^{-1},0)$, and $\ch_1^{\cc-1}\Psi_{T_1}^{1-\cc}\mathbf{N}^{-1}$ has infinity type $(k_1,2-k_1)$. Hence for $k_1\geq 2$ this character lies in the range of interpolation of $\cL_{\pp,\mathfrak{f}}^{\mathrm{Katz}}$, where $\mathfrak{f}$ denotes the conductor of $\xi_1^{1-\cc}$, and from the right above formula for $L(\ch_1^{\cc-1}\Psi_{T_1}^{1-\cc},1)$ and Theorem~\ref{thm:katz} we obtain
\begin{equation}\label{eq:hida-formula}
\begin{aligned}
\cL_{\pp,\mathfrak{f}}^{\mathrm{Katz}}(\ch_1^{\cc-1}\Psi_{T_1}^{1-\cc}\mathbf{N}^{-1})&
=\biggl(\frac{\Omega_p}{\Omega_\infty}\biggr)^{2k_1-2}\cdot\frac{\pi^{2k_1-2}\cdot 2^{3k_1-3}}{\sqrt{D_K}^{k_1+1}}\\
&\quad\times\bigl(1-\ch_1^{\cc-1}\Psi^{1-\cc}_{T_1}(\ppbar)\bigr)\bigl(1-\ch_1^{\cc-1}\Psi^{1-\cc}_{T_1}(\ppbar)p^{-1}\bigr)\cdot\Vert\bfg^\circ_{T_1}\Vert_{\Gamma_0(C)}^2\cdot\frac{w_K}{h_K}.
\end{aligned}
\end{equation}
Moreover, by the functional equation of Theorem~\ref{thm:katz} 
and the definition of $\eta_{\bfg}^{\star}$ we have the relation 
\begin{equation}\label{eq:Katz-and-congruence}
    \frac{h_K}{w_K}\cdot\Lcal_{\pp,\mathfrak{f}}^{\mathrm{Katz}}(\ch_1^{\cc-1}\Psi_{T_1}^{1-\cc}\mathbf{N}^{-1})\sim_p\eta_{\bfg_{T_1}}^{\star},
\end{equation}
where $\sim_p$ denotes equality up to a $p$-adic unit. Therefore, equations \eqref{eq:hida-formula} and \eqref{eq:Katz-and-congruence} imply that
\begin{equation}\label{eq:subs-for-hida-formula}
\frac{\Vert\bfg_{T_1}^\circ\Vert_{\Gamma_0(C)}^2}{\eta_{\bfg_{T_1}}^{\star}}\cdot\bigl(1-\ch_1^{\cc-1}\Psi_{T_1}^{1-\cc}(\ppbar)\bigr)\bigl(1-\ch_1^{\cc-1}\Psi^{1-\cc}_{T_1}(\ppbar)p^{-1}\bigr)\sim_p  \biggl(\frac{\Omega_\infty}{\Omega_p}\biggr)^{2k_1-2}\cdot\frac{\sqrt{D_K}^{k_1+1}}{(2\pi)^{2k_1-2}}   
\end{equation}

Hence, from (\ref{eq:period}) and (\ref{eq:subs-for-hida-formula}) we arrive at
\begin{equation}\label{eq:Omega-bis}
\frac{1}{\Omega_{\bfg_{T_1}}}\sim_p\biggl(\frac{\Omega_p}{\Omega_\infty}\biggr)^{2k_1-2}\cdot\frac{(2\pi)^{2k_1-2}}{\sqrt{-D_K}^{k_1+1}}.
\end{equation}

Finally, note that the characters $\ch_1\ch_2\Psi_{W_1}^{\cc-1}$ and $\ch_1\ch_2^\cc\Psi_{W_2}^{\cc-1}$ in the right-hand side of (\ref{eq:explicit}) are both  anticyclotomic, and of infinity type $(-(k_1+k_2-2)/2,(k_1+k_2-2)/2)$ and $(-(k_1-k_2)/2,(k_1-k_2)/2)$, respectively, and so for $k_1\geq k_2+2r$ they are in the range of interpolation for $\mathscr{L}_\pp^{\BDP}(f/K,\ch_1\ch_2)$ and $\mathscr{L}_\pp^{\BDP}(f/K,\ch_1\ch_2^\cc)$, respectively. Thus substituting (\ref{eq:explicit}), (\ref{eq:Ep}), and (\ref{eq:Omega-bis}) into the interpolation formula for $\mathscr{L}_p^\bfg(f,\bfg,\bfh)$ in Theorem~\ref{thm:hsieh-triple} and comparing with Theorem~\ref{thm:BDP} we finally arrive at
\begin{align*}
\mathscr{L}_p^{\bfg}(f,\bfg,\bfh)^2(S_1,S_2)
\sim_p\frac{-1}{D_K^{k_1+1}}\cdot\mathscr{L}_\pp^{\BDP}(f/K,\ch_1\ch_2)^2(W_1)\cdot\mathscr{L}_\pp^{\BDP}(f/K,\ch_1\ch_2^\cc)^2(W_2),
\end{align*}
and this yields the proof of the result.
\end{proof}



\subsection{Selmer group decomposition}
\label{subsec:Selmer-decomposition-indef}



\begin{prop}\label{prop:factor-S-indef}
Under the direct sum decomposition
\[
\rH^1(\Q,\Vsdag)\simeq\rH^1(K,T_f^{\vee}(1-r)\otimes\ch_1^{-1}\ch_2^{-1}\Psi_{W_1}^{1-\cc})\oplus\rH^1(K,T_f^{\vee}(1-r)\otimes\ch_1^{-1}\ch_2^{-\cc}\Psi_{W_2}^{1-\cc})
\]
of \eqref{eq:shapiro}, 
the balanced Selmer group $\Sel^{\bal}(\Q,\Vsdag)$ decomposes as
\begin{align*}
\Sel^{\bal}(\Q,\Vsdag)\simeq\Sel_{\relstr}&(K,T_f^{\vee}(1-r)\otimes\ch_1^{-1}\ch_2^{-1}\Psi_{W_1}^{1-\cc})\\ \oplus&\Sel_{\ord}(K,T_f^{\vee}(1-r)\otimes\ch_1^{-1}\ch_2^{-\cc}\Psi_{W_2}^{1-\cc});
\end{align*}
and the $\bfg$-unbalanced Selmer group $\Sel^{\bfg}(\Q,\Vsdag)$ decomposes as
\begin{align*}
\Sel^{\bfg}(\Q,\Vsdag)\simeq
\Sel_{\relstr}&(K,T_f^{\vee}(1-r)\otimes\ch_1^{-1}\ch_2^{-1}\Psi_{W_1}^{1-\cc})\\ \oplus&\Sel_{\relstr}(K,T_f^{\vee}(1-r)\otimes\ch_1^{-1}\ch_2^{-\cc}\Psi_{W_2}^{1-\cc}).\end{align*}
\end{prop}

\begin{proof}
The result for $\Sel^{\bal}(\Q,\Vsdag)$ is given in Proposition~\ref{prop:factor-S}.
It suffices to focus on $\Sel^{\bfg}(\Q,\Vsdag)$. Put 
\[
\widetilde{\mathbf{V}}_{Q_0}^\dagger=\bigl(T_f^{\vee}(1-r)\otimes\ch_1^{-1}\ch_2^{-1}\Psi_{W_1}^{1-\cc}\bigr)\oplus\bigl(T_f^{\vee}(1-r)\otimes\ch_1^{-1}\ch_2^{-\cc}\Psi_{W_2}^{1-\cc}\bigr),
\] 
so by Shapiro's lemma we have $\rH^1(\Q,\Vsdag)\simeq\rH^1(K,\widetilde{\mathbf{V}}_{Q_0}^\dagger)$.
 
Putting $T_i=\mathbf{u}(1+Z_i)-1$ as in the proof of Proposition~\ref{prop:factor-L-indef}, so \eqref{eq:CM-ind} can be rewritten as
\[
V_{\bfg}\cong{\mathrm{Ind}}_K^\Q(\xi_1^{-1}\Psi_{T_1}),\quad\quad
V_{\bfh}\cong{\mathrm{Ind}}_K^\Q(\xi_2^{-1}\Psi_{T_2}),
\]
a direct computation shows that the $\bfg$-unbalanced local condition is given by
\begin{align*}
\mathscr{F}_p^\bfg(\Vsdag)&=T_f^{\vee}\otimes\ch_1^{-1}\Psi_{T_1}\otimes\bigl(\ch_2^{-1}\Psi_{T_2}\oplus\ch_2^{-\cc}\Psi_{T_2}^\cc\bigr)\otimes\varepsilon_{\mathrm{cyc}}^{1-r}(\Psi_{T_1}^{-1/2}\Psi_{T_2}^{-1/2}\circ\mathscr{V})\\
&=\bigl(T_f^{\vee}(1-r)\otimes\ch_1^{-1}\ch_2^{-1}\Psi_{W_1}^{1-\cc}\bigr)\oplus \bigl(T_f^{\vee}(1-r)\otimes\ch_1^{-1}\ch_2^{-\cc}\Psi_{W_2}^{1-\cc}\bigr).
\end{align*}
Therefore, we have 
\[ \mathscr{F}_\pp(\widetilde{\mathbf{V}}_{Q_0}^\dagger)=\widetilde{\mathbf{V}}_{Q_0}^\dagger,\quad  \mathscr{F}_{\ppbar}(\widetilde{\mathbf{V}}_{Q_0}^\dagger)=0, 
\]
and this yields the stated decomposition for $\Sel^{\bfg}(\Q,\Vsdag)$.
\end{proof}

\begin{cor}\label{cor:factor-S-indef}
The balanced Selmer group $\Sel^{\bal}(\Q,\mathbf{A}_{Q_0}^\dagger)$ decomposes as
\begin{align*}
\Sel^{\bal}(\Q,\mathbf{A}_{Q_0}^\dagger)&\simeq\Sel_{\strrel}(K,A_f(r)\otimes\ch_1\ch_2\Psi_{W_1}^{\cc-1})\oplus\Sel_{\ord}(K,A_f(r)\otimes\ch_1\ch_2^\cc\Psi_{W_2}^{\cc-1});
\end{align*}
and the $\bfg$-unbalanced Selmer group $\Sel^{\bfg}(\Q,\mathbf{A}_{Q_0}^\dagger)$ decomposes as
\begin{align*}
\Sel^{\bfg}(\Q,\mathbf{A}_{Q_0}^\dagger)&\simeq
\Sel_{\strrel}(K,A_f(r)\otimes\ch_1\ch_2\Psi_{W_1}^{\cc-1})\oplus\Sel_{\strrel}(K,A_f(r)\otimes\ch_1\ch_2^\cc\Psi_{W_2}^{\cc-1}).
\end{align*}
\end{cor}

\begin{proof}
As in Corollary~\ref{cor:factor-S}, this is immediate from Proposition~\ref{prop:factor-S-indef} and local Tate duality.
\end{proof}

\subsection{Explicit reciprocity law}
\label{subsec:ERL-indef}

As in $\S\ref{subsec:ERL-def}$, we put 
\[
\mathbb{V}^\dagger=\Vsdag\otimes_{\cO\dBr{Z_1,Z_2}}\cO\dBr{Z_1,Z_2}/(Z_2),
\]
let $h_2$ be specialization of $\bfh=\boldsymbol{\theta}_{\xi_2}(Z_2)$ of weight $2$ given by $Z_2=0$, but now consider the second component $\kappa_2(f,\bfg,h_2)$ of the specialized big diagonal class
\[
\kappa(f,\bfg,h_2)=(\kappa_1(f,\bfg,h_2),\kappa_2(f,\bfg,h_2))
\]
according to the decomposition of $\Sel^{\bal}(\Q,\mathbb{V}^\dagger(\mathscr{N}))$ from Proposition~\ref{prop:factor-S-indef}; in particular, we have
\begin{equation}\label{eq:k2-ord-ord}
\kappa_2(f,\bfg,h_2)\in\Sel_{\ord}(K,T_f^{\vee}(1-r)\otimes\ch_1^{-1}\ch_2^{-\cc}\Psi_{W_2}^{1-\cc}(\mathscr{N})),
\end{equation}
where $W_2=(1+Z_1)^{1/2}-1$. 

Let $\mathfrak{X}_{\cO\dBr{W_2}}^+$ be the set of ring homomorphisms $Q\in{\mathrm{Spec}}(\cO\dBr{W_2})(\overline{\Q}_p)$ with $Q(1+W_2)=\zeta_Q\mathbf{u}^{j_Q}$ for some $\zeta_Q\in\mu_{p^\infty}$ and $j_Q\in\Z_{\geq 0}$, and for any $\cO\dBr{W_2}$-module $M$ we let $M_Q$ denote the corresponding  specialization.



\begin{thm}\label{thm:ERL-indef}
For every triple $(\breve{f},\breve{\bfg},\breve{h}_2)$ of level-$N$ test vectors for $(f,\bfg,h_2)$ there is an injective $\cO\dBr{W_2}$-module homomorphism with pseudo-null cokernel
\[
{\mathrm{Log}}^{\bfg}_{\ppbar,(\breve{f},\breve{\bfg},\breve{h}_2)}:\rH^1(K_{\ppbar},T_f^{\vee,+}(1-r)\otimes\ch_1^{-1}\ch_2^{-\cc}\Psi_{W_2}^{1-\cc}(\mathscr{N}))\rightarrow C(\bfg)^{-1}\cO\dBr{W_2}
\]
such that for any $\mathfrak{Z}\in\rH^1(K_{\ppbar},T_f^{\vee,+}(1-r)\otimes\ch_1^{-1}\ch_2^{-\cc}\Psi_{W_2}^{1-\cc}(\mathscr{N}))$ and $Q\in\mathfrak{X}_{\cO\dBr{W_2}}^+$ of weight $j_Q\geq r$ we have
\[
{\mathrm{Log}}^\bfg_{\ppbar,(\breve{f},\breve{\bfg},\breve{h}_2)}(\mathfrak{Z})_Q=
c_Q\cdot\left\langle{\mathrm{exp}}^*_p(\mathfrak{Z}_{Q}),\omega_{\breve{f}}\otimes\eta_{\breve{\bfg}_{Q'}}\otimes\omega_{\breve{h}_2}\right\rangle_{\mathrm{dR}},
\]
where $c_Q$ is an explicit nonzero constant, and $Q'\in{\mathrm{Spec}}(\cO\dBr{Z_1})(\overline{\Q}_p)$ is given by $Q'(1+Z_1)=\zeta_Q^2\mathbf{u}^{2j_Q}$. Moreover, we have 
the \emph{explicit reciprocity law}
\[
{\mathrm{Log}}^\bfg_{\ppbar,(\breve{f},\breve{\bfg},\breve{h}_2)}\bigl({\mathrm{res}}_{\ppbar}(\kappa_2(f,\bfg,h_2))\bigr)(W_2)=\mathscr{L}_p^{\bfg}(\breve{f},\breve{\bfg},\breve{h}_2)(S_1),
\]
where $S_1=\mathbf{u}^2(1+Z_1)-1=\mathbf{u}^2(1+W_2)^2-1$.
\end{thm}

\begin{proof}
As shown in the proof of Theorem~\ref{thm:ERL-def}, we have
\begin{align*}
\mathscr{F}_p^{\bal}(\VVdag)/\mathscr{F}_p^3(\VVdag)&\cong
\bigl(T_f^{\vee,-}(1-r)\otimes\ch_1^{-1}\ch_2^{-1}\Psi_{W_1}^{1-\cc}\bigr)\\
&\quad\oplus\bigl(T_f^{\vee,+}(1-r)\otimes\ch_1^{-1}\ch_2^{-\cc}\Psi_{W_2}^{1-\cc}\bigr)\oplus\bigl(T_f^{\vee,+}(1-r)\otimes\ch_1^{-\cc}\ch_2^{-1}\Psi_{W_2}^{\cc-1}\bigr),
\end{align*}
with the direct summands corresponding to $\mathbb{V}_{f}^{\bfg h_2}$, $\mathbb{V}_{h_2}^{f\bfg}$, and $\mathbb{V}_{\bfg}^{fh_2}$ from (\ref{eq:gr2}), respectively. As a result, the analogue of the composite map (\ref{eq:map-ERL}) in the present $\bfg$-unbalanced case:
\begin{equation}\label{eq:res-g-unb}
\begin{aligned}
\Sel^{\bal}(\Q,\VVdag(\mathscr{N}))\xrightarrow{{\mathrm{res}}_p}&\rH^1(\Q_p,\mathscr{F}_p^{\bal}(\VVdag(\mathscr{N})))\\ \rightarrow &\rH^1(\Q_p,\mathscr{F}_p^{\bal}(\VVdag(\mathscr{N}))/\mathscr{F}_p^3(\VVdag(\mathscr{N})))
\rightarrow\rH^1(\Q_p,\mathbb{V}_{\bfg}^{fh_2}(\mathscr{N}))
\end{aligned}
\end{equation} 
corresponds, under the isomorphism of Proposition~\ref{prop:factor-S-indef}, to the projection onto \[\Sel_{\ord}(K,T_f^{\vee}(1-r)\otimes\ch_1^{-1}\ch_2^{-\cc}\Psi_{W_2}^{1-\cc}(\mathscr{N}))\] (the second factor in that decomposition) composed with the restriction map
\[
\Sel_{\ord}(K,T_f^{\vee}(1-r)\otimes\ch_1^{-1}\ch_2^{-\cc}\Psi_{W_2}^{1-\cc}(\mathscr{N}))\xrightarrow{{\mathrm{res}}_{\ppbar}}\rH^1(K_{\ppbar},T_f^{\vee,+}(1-r)\otimes\ch_1^{-1}\ch_2^{-\cc}\Psi_{W_2}^{1-\cc}(\mathscr{N})).
\]
In particular, under the corresponding identifications the image \[{\mathrm{res}}_p(\kappa(f,\bfg,h_2))_\bfg\] of $\kappa(f,\bfg,h_2)$ under \eqref{eq:res-g-unb} is such that 
\[{\mathrm{res}}_p(\kappa(f,\bfg,h_2))_\bfg={\mathrm{res}}_{\ppbar}(\kappa_2(f,\bfg,h_2))\] 
in
\[
\rH^1(\Q_p,V_\bfg^{fh_2}(\mathscr{N}))\simeq\rH^1(K_{\ppbar},T_f^{\vee,+}(1-r)\otimes\ch_1^{-1}\ch_2^{-\cc}\Psi_{W_2}^{1-\cc}(\mathscr{N})).
\] 
On the other hand, the construction of ${\mathrm{Log}}^{\bfg}_{\ppbar,(\breve{f},\breve{\bfg},\breve{h}_2)}$ is deduced from a specialization of the $3$-variable $p$-adic regulator map ${\mathrm{Log}}_{(\breve{\bff},\breve{\bfg},\breve{\bfh})}^{\bfg}$ in \S\ref{subsec:diag} by the same argument as in \cite[Prop.~7.3]{ACR}), and the associated explicit reciprocity law then follows from Theorem~\ref{thm:ERL}.
\end{proof}

In particular, for the choice of level-$N$ test vectors from Theorem~\ref{thm:hsieh-triple} we deduce the following.

\begin{cor}\label{cor:ERL-indef}
Assume that $\ch_1$ satisfies the conditions in Lemma~\ref{lem:HT}, and put
\[
S_1=\mathbf{u}^2(1+Z_1)-1,\quad 
W_1=\mathbf{u}(1+Z_1)^{1/2}-1,\quad
W_2=(1+Z_1)^{1/2}-1.
\]
Then 
\begin{align*}
\mathcal{L}_{\pp,\bar{\xi}_1^-}^{{\mathrm{Katz}},-}\cdot{\mathrm{Log}}^\bfg_{\ppbar,(\breve{f}^\star,\breve{\bfg}^\star,\breve{h}_2^\star)}\bigl(&{\mathrm{res}}_{\ppbar}(\kappa_2(f,\bfg,h_2))\bigr)(W_2)=\mathscr{L}_p^\bfg(f,\bfg,h_2)(S_1)\\
&=\pm\mathbf{w}\cdot\mathscr{L}_\pp^{\BDP}(f/K,\xi_1\xi_2)(W_1)\cdot\mathscr{L}_\pp^{\BDP}(f/K,\xi_1\xi_2^\cc)(W_2),
\end{align*}
where $\mathscr{L}_p^{\bfg}(f,\bfg,h_2)$ is the specialization of $\mathscr{L}_p^\bfg(\bff,\bfg,\bfh)$ in Definition~\ref{def:gen-triple} and $\mathbf{w}$ is a unit in $\cO\dBr{Z_1}\otimes_{\Z_p}\Q_p$.
\end{cor}

\begin{proof}
The first equality is immediate from Lemma~\ref{lem:HT} and Theorem~\ref{thm:ERL-indef}, and the second follows from  Proposition~\ref{prop:factor-L-indef}.
\end{proof}

\subsection{On the Bloch--Kato conjecture in rank 0}
\label{subsec:BK-indef}

As another application of the Euler system construction in this paper, we now deduce a result towards the Bloch--Kato conjecture for  
\[
V_{f,\chi}=V_f^\vee(1-r)\otimes\chi^{-1}
\]
analogous to Theorem~\ref{thm:BK-def} but in the indefinite setting. 

\begin{thm}\label{thm:BK-indef}
Let $f\in S_{2r}(\Gamma_0(N_f))$, with $p\nmid N_f$, be a $p$-ordinary newform of weight $2r\geq 2$, let $K$ be an imaginary quadratic field satisfying $\eqref{eq:spl}$ and $\eqref{eq:p-nmid-h}$, and let $\chi$ be an anticyclotomic Hecke character of conductor $c\cO_K$ and infinity type $(-j,j)$, $j\geq 0$. 
Assume that:
\begin{itemize}
    \item every prime $\ell\mid N_f$ splits in $K$;
    \item $(pN_f,cD_K)=1$;
    \item $\chi_t$ has conductor prime-to-$p$;
    \item $\bar{\rho}_f$ is absolutely irreducible;
    \item $f$ is not of CM-type.
\end{itemize}
Then
\[
L(f/K,\chi,r)\neq 0\quad\Longrightarrow\quad\Sel_{\mathrm{BK}}(K,V_{f,\chi})=0,
\] 
and hence the Bloch--Kato conjecture for $V_{f,\chi}$ holds in analytic rank zero.
\end{thm}

\begin{proof}
We argue similarly as in the proof of Theorem~\ref{thm:BK-def}, with some modifications. By our assumption on $N_f$, the sign in the functional equation of $L(f/K,\chi,s)$ is $-1$ for $0\leq j<r$, so without loss of generality we assume that 
$j\geq r$.

Write $\chi_t=\alpha/\alpha^\cc$ for a ray class character $\alpha$ as in the proof of Theorem~\ref{thm:BK-def}, but now put
\begin{equation}\label{eq:xi-i}
\ch_1:=\beta\alpha,\quad\quad\ch_2:=(\beta^{-1}\alpha^{-\cc})^\cc=\beta\alpha^{-1},
\end{equation}
with $\beta$ an auxiliary ring class character of $K$ of $q$-power conductor for a suitable prime $q\neq p$ split in $K$. Consider the setting of $\S$\ref{subsec:factor-L-indef} with the CM Hida families
\[
(\bfg,\bfh)=(\boldsymbol{\theta}_{\ch_1}(Z_1),\boldsymbol{\theta}_{\ch_2}(Z_2)).
\]
By \cite[Thm.~C]{ChHs1} we may take $q$ and $\beta$ so that $\mathscr{L}_\pp^{\BDP}(f/K,\beta^2)(W_1)$ is a unit in $\Z_p^{\ur}\dBr{W_1}\otimes_{\Z_p}\cO$ and $\xi_1$ satisfies the hypotheses of Lemma~\ref{lem:HT}. With such a choice, using the equalities $\xi_1\xi_2=\beta^2$ and $\xi_1\xi_2^\cc=\chi_t$ the explicit reciprocity law of Corollary~\ref{cor:ERL-indef} becomes
\begin{equation}\label{eq:ERL-infef-unit}
\mathcal{L}_{\pp,\bar{\chi}_1^-}^{{\mathrm{Katz}},-}\cdot{\mathrm{Log}}^\bfg_{\ppbar,(\breve{f}^\star,\breve{\bfg}^\star,\breve{h}_2^\star)}\bigl({\mathrm{res}}_{\ppbar}(\kappa_2(f,\bfg,h_2))\bigr)(W_2)
=\pm\mathbf{w}'\cdot\mathscr{L}_\pp^{\BDP}(f/K,\chi_t)(W_2),
\end{equation}
where $W_2=(1+Z_1)^{1/2}-1=V_1$, with $\mathbf{w}'$ is a unit in $\Z_p^{\ur}\dBr{W_2}\otimes_{\Z_p}L_\mathfrak{P}$.

Denoting by $Q\in\mathfrak{X}_{\cO\dBr{W_2}}^{+}$ the specialization $W_2\mapsto\zeta_Q\mathbf{u}^{j}-1$  ($\zeta_Q\in\mu_{p^\infty}$) such that
\[
\chi_w=\Psi_{W_2}^{\cc-1}\vert_{W_2=\zeta_Q\mathbf{u}^j-1},
\]
from \eqref{eq:ERL-infef-unit}, Theorem~\ref{thm:BDP}, and Theorem~\ref{thm:ERL-indef} we find
\begin{equation}\label{eq:L-implies-indef}
\begin{aligned}
L(f/K,\chi,r)\neq 0\quad&\Longrightarrow\quad\mathscr{L}_\pp^{\BDP}(f/K,\chi_t)(\chi_w(\gamma_-)-1)\neq 0\\
&\Longrightarrow\quad{\mathrm{res}}_{\ppbar}(\kappa_2(\breve{f}^\star,\breve{\bfg}^\star,\breve{h}_2^\star)_Q)\neq 0,
\end{aligned}
\end{equation}
where $\kappa_2(\breve{f}^\star,\breve{\bfg}^\star,\breve{h}_2^\star)$
denotes the image of the class $\kappa_2(f,\bfg,h_2)$ in \eqref{eq:k2-ord-ord} under the projection
\begin{align*}
\Sel_{\ord}(K,T_f^{\vee}(1-r)\otimes\chi_t^{-1}\Psi_{W_2}^{1-\cc}(\mathscr{N}))&\rightarrow\Sel_{\ord}(K,T_f^{\vee}(1-r)\otimes\chi_t^{-1}\Psi_{W_2}^{1-\cc})
\end{align*}
associated to $(\breve{f}^\star,\breve{\bfg}^\star,\breve{h}_2^\star)$. 

As noted in Remark~\ref{rem:diag-components}, the class  $\kappa_2(\breve{f}^\star,\breve{\bfg}^\star,\breve{h}_2^\star)$ is the bottom class of the anticyclotomic Euler system $\{ {}^\cc\mathbf{z}_{f,\xi_1,\xi_2,m}\}_m$ of Theorem~\ref{maintheorem2} for $T_{f,\xi_1\xi_2^\cc}=T_{f,\chi_t}$ (and the given choice of test vectors). Therefore, letting ${\mathrm{tw}}_{\chi_w^{-1}}(\kappa_2(\breve{f}^\star,\breve{\bfg}^\star,\breve{h}_2^\star))$ denote the image of $\kappa_2(\breve{f}^\star,\breve{\bfg}^\star,\breve{h}^\star_2)$ under the `twisting' map
\[
\Sel_{\ord}(K,T_f^{\vee}(1-r)\otimes\chi_t^{-1}\Psi_{W_2}^{1-\cc})\rightarrow\Sel_{\ord}(K,T_f^{\vee}(1-r)\otimes\chi^{-1}\Psi_{W_2}^{1-\cc})
\]
given by the change of variables $W_2\mapsto\zeta_Q^{-1}\mathbf{u}^{-j}(1+W_2)-1$, it follows that \[{\mathrm{tw}}_{\chi_w^{-1}}(\kappa_2(\breve{f}^\star,\breve{\bfg}^\star,\breve{h}_2^\star))\] is the bottom class of the Euler system 
\begin{equation}\label{eq:tw-ES-indef}
\bigl\{{}^\cc\mathbf{z}_{f,\chi,m}\bigr\}_m:=\bigl\{ {}^\cc\mathbf{z}_{f,\xi_1,\xi_2,,m}\otimes\chi_w^{-1}\bigr\}_m
\end{equation}
of Theorem~\ref{maintheorem2} for $T_{f,\xi_1\xi_2^\cc}\otimes\chi_w^{-1}=T_{f,\chi}$. 

Since by construction the class   $\kappa_2(\breve{f}^\star,\breve{\bfg}^\star,\breve{h}_2^\star)_Q$ in \eqref{eq:L-implies-indef} agrees with the image of the bottom class ${}^\cc\mathbf{z}_{f,\chi,1}$ of the system \eqref{eq:tw-ES-indef} under natural map
\[
\Sel_{\ord}(K,T_f^{\vee}(1-r)\otimes\chi^{-1}\Psi_{W_2}^{1-\cc})\cong\Sel_{\ord}(K_\infty^-,T_{f,\chi})\rightarrow\Sel_{\ord}(K,T_{f,\chi}),
\]
from Theorem~\ref{thm:rank-1-general} 
we deduce that $\Sel_{\ord}(K,V_{f,\chi})$ is one-dimensional, spanned by $\kappa_2(\breve{f}^\star,\breve{\bfg}^\star,\breve{h}_2^\star)_Q$. 
Since we have in fact shown that ${\mathrm{res}}_{\ppbar}(\kappa_2(\breve{f}^\star,\breve{\bfg}^\star,\breve{h}_2^\star)_Q)\neq 0$, the vanishing of $\Sel_{\relstr}(K,V_{f,\chi})$ then follows by global duality similarly as in the proof of Theorem~\ref{thm:BK-def}; and since by Lemma~\ref{lem:BK-Gr}, for $j\geq r$ the latter group agrees with $\Sel_{\mathrm{BK}}(K,V_{f,\chi})$, this yields the result.
\end{proof}

\subsection{On the Iwasawa main conjecture}
\label{subsec:IMC-indef}

Writing any anticyclotomic Hecke character $\chi$ of $K$ as $\chi=\chi_t\cdot\chi_w$ as in $\S\ref{subsec:BK-def}$, we let $\mathscr{L}_\pp^{\BDP}(f/K,\chi)$ denote the image of $\mathscr{L}_\pp^{\BDP}(f/K,\chi_t)$ under the twisting homomorphism ${\mathrm{tw}}_{\chi_w}:\Z_p^{\ur}\dBr{W_2}\otimes_{\Z_p}\cO\rightarrow\Z_p^{\ur}\dBr{W_2}\otimes_{\Z_p}\cO$ given by $W_2\mapsto\chi_w(\gamma_-)(1+W_2)-1$. 

Our next application is to the Iwasawa--Greenberg main conjecture for $\mathscr{L}_\pp^{\BDP}(f/K,\chi)$. 

\begin{thm}\label{thm:IMC-indef}
Let the hypotheses be an in Theorem~\ref{thm:BK-indef}, and assume in addition that $f$ has big image.  Then $\Sel_{\strrel}(K,A_{f,\chi})$ is cotorsion over $\Lambda_K^-$, and we have the divisibility
\[
{\mathrm{char}}_{\Lambda_K^-}\bigl(\Sel_{\strrel}(K,A_{f,\chi})^\vee\bigr)\supset\bigl(\mathscr{L}_\pp^{\BDP}(f/K,\chi)^2\bigr)
\]
in $\Lambda_K^{-,\ur}\otimes_{\Z_p}\Q_p$.
\end{thm}

\begin{proof}
Repeating the argument in the proof of Theorem~\ref{thm:BK-def}, we arrive at the equality
\begin{equation}\label{eq:ERL-indef-unit-bis}
\mathcal{L}_{\pp,\bar{\xi}_1^-}^{{\mathrm{Katz}},-}\cdot{\mathrm{Log}}^\bfg_{\ppbar,(\breve{f}^\star,\breve{\bfg}^\star,\breve{h}_2^\star)}\bigl({\mathrm{res}}_{\ppbar}(\kappa_2(f,\bfg,h_2))\bigr)(W_2)
=\pm\mathbf{w}'\cdot\mathscr{L}_\pp^{\BDP}(f/K,\chi_t)(W_2),
\end{equation}
with $\mathbf{w}'$ is a unit in $\Z_p^{\ur}\dBr{W_2}\otimes_{\Z_p}L_\mathfrak{P}$. Since $\mathscr{L}_\pp^{\BDP}(f/K,\chi_t)(W_2)$ is nonzero by Theorem~\ref{thm:BDP}, letting 
\[
\kappa_2(\breve{f}^\star,\breve{\bfg}^\star,\breve{h}_2^\star)\in\Sel_{\ord}(K,T_f^{\vee}(1-r)\otimes\chi_t^{-1}\Psi_{W_2}^{1-\cc}) 
\]
be as in the proof of Theorem~\ref{thm:BK-indef}, from \eqref{eq:ERL-indef-unit-bis} we conclude that $
\kappa_2(\breve{f}^\star,\breve{\bfg}^\star,\breve{h}_2)$ is non-torsion. 

Since ${\mathrm{tw}}_{\chi_w^{-1}}(\kappa_2(\breve{f}^\star,\breve{\bfg}^\star,\breve{h}_2^\star))$ is the bottom class of the Euler system $\{{}^\cc\mathbf{z}_{f,\chi,m}\}_m$ for $T_{f,\chi}$ constructed in Theorem~\ref{maintheorem2}, 
from Theorem~\ref{thm:IMC-general} 
we deduce that $\Sel_{\ord}(K,T_{f,\chi})$ and $X_{\ord}(K,A_{f,\chi})$ have both $\Lambda_K^-$-rank one, and we have the divisibility
\begin{equation}\label{eq:IMC-div-ord}
{\mathrm{char}}_{\Lambda_{K}^-}\bigl(X_{\ord}(K,A_{f,\chi})_{\mathrm{tors}}\bigr)\supset{\mathrm{char}}_{\Lambda_K^-}\biggl(\frac{\Sel_{\ord}(K,T_{f,\chi})}{\Lambda_K^-\cdot{\mathrm{tw}}_{\chi_w^{-1}}(\kappa_2(\breve{f}^\star,\breve{\bfg}^\star,\breve{h}_2^\star))}\biggr)^2
\end{equation}
in $\Lambda_K^-$. Since from \eqref{eq:ERL-indef-unit-bis} we deduce an explicit reciprocity law relating 
\[
{\mathrm{res}}_{\ppbar}({\mathrm{tw}}_{\chi_w^{-1}}(\kappa_2(\breve{f}^\star,\breve{\bfg}^\star,\breve{h}_2^\star)))
\]
to ${\mathrm{tw}}_{\chi_w}(\mathscr{L}_\pp^{\BDP}(f/K,\chi_t))=\mathscr{L}_\pp^{\BDP}(f/K,\chi)$, the result now follows from \eqref{eq:IMC-div-ord} and global duality by the same argument as in \cite[Thm.~5.1]{BCK-PRconj}.
\end{proof}

\begin{rmk}
Note that Theorem~\ref{thm:IMC-indef} also yields a proof of a divisibility towards the Perrin-Riou main conjecture for generalized Heegner cycles formulated in \cite{LV-kyoto} (see \cite[Thm.~5.2]{BCK-PRconj} for the argument), removing some of the hypotheses in the main result of \cite{LV-kyoto}.
\end{rmk}

\subsection{On the Bloch--Kato conjecture in rank 1}

We can also give an analogue of Theorem~\ref{thm:BK-def-1} in the indefinite case. 

\begin{thm}\label{thm:BK-indef-1}
Let the hypotheses be as in Theorem~\ref{thm:BK-indef}. If $0\leq j<r$ (which implies $L(f/K,\chi,r)=0$), then 
\[
\dim_{L_\mathfrak{P}}\,\Sel_{\mathrm{BK}}(K,V_{f,\chi})\geq 1.
\]
Moreover, there exists a class ${}^\cc z_{f,\chi}\in\Sel_{\mathrm{BK}}(K,V_{f,\chi})$ such that
\[
{}^\cc z_{f,\chi}\neq 0\quad\Longrightarrow\quad\dim_{L_\mathfrak{P}}\,\Sel_{\mathrm{BK}}(K,V_{f,\chi})=1.
\]
\end{thm}

\begin{proof}
The proof of Theorem~\ref{thm:IMC-indef} showed that the class
\[
{}^\cc\mathbf{z}_{f,\chi}:={\mathrm{tw}}_{\chi_w^{-1}}(\kappa_2(\breve{f}^\star,\breve{\bfg}^\star,\breve{h}_2^\star))\in\Sel_{\ord}(K_\infty^-,T_{f,\chi})
\]
is non-torsion over $\Lambda_K^-$. On the other hand, one readily checks that the natural map
\begin{equation}\label{eq:proj-indef}
\Sel_{\ord}(K_\infty^-,T_{f,\chi})/(\gamma_--1)\Sel_{\ord}(K_\infty^-,T_{f,\chi})\rightarrow\Sel_{\ord}(K,T_{f,\chi})
\end{equation}
is injective. Thus we conclude that $\Sel_{\ord}(K,T_{f,\chi})$ has positive $\cO$-rank, which together with  Lemma~\ref{lem:BK-Gr} yields the first part of the theorem. Letting ${}^\cc z_{f,\chi}\in\Sel_{\ord}(K,T_{f,\chi})$ be the image of ${}^\cc\mathbf{z}_{f,\chi}$ under (\ref{eq:proj-indef}), 
the last claim follows from Theorem~\ref{thm:rank-1-general}. 
\end{proof}

\newpage
\bibliographystyle{amsalpha}
\bibliography{Schoen-cm-final}

\end{document}